\documentclass[11pt, reqno]{amsart}
\usepackage{wrapfig, lipsum,booktabs}
\usepackage{graphicx}
\usepackage{amssymb}
\usepackage{epstopdf}
\usepackage{verbatim}
\usepackage{bm}
\usepackage{multicol}
\usepackage{multirow}
\usepackage{subfigure}
\usepackage{float}
\usepackage{color}
\usepackage{soul}
\usepackage{tikz}
\usepackage{tikz-3dplot}
\usepackage{pgfplots}
\newtheorem{theorem}{Theorem}[section]

\usepackage{multirow}
\usepackage{amsmath,amssymb,eucal}
\usepackage{graphicx,subfigure,epsfig}
\usepackage{url}
\usepackage[top=1in, bottom=1.in, left=1in, right=1in]{geometry}

\newcommand{\bld}[1]{\hbox{\boldmath$#1$}}
\newcommand{\Th}{\mathcal{T}_h}

\newcommand{\Eh}{\mathcal{E}_h}

\setulcolor{red}
\newcommand{\jmp}[1]{[\![#1 ]\!]}

\begin{document}
\title[HybridMixed]{A hybrid-mixed finite element method
  for single-phase Darcy flow in fractured porous media}
\author{Guosheng Fu}
\address{Department of Applied and Computational Mathematics and
Statistics, University of Notre Dame, USA.}
\email{gfu@nd.edu}
\author{Yang Yang}
\address{Department of Mathematical Sciences, Michigan Technological University, USA.}
\email{yyang7@mtu.edu}
 \thanks{
 G. Fu was partially supported by the NSF grant DMS-2012031.
 Y. Yang was partially supported by the NSF grant DMS-1818467.
 }

 \keywords{Hybrid-mixed finite element method; fractured porous media,
 hybrid-dimensional model}
\subjclass{65N30, 65N12, 76S05, 76D07}
\begin{abstract}
  We present a hybrid-mixed finite element method for
  a novel hybrid-dimensional model of single-phase Darcy flow in a fractured porous media.
In this model, the fracture is treated as an
$(d-1)$-dimensional interface within the $d$-dimensional fractured porous
domain, for $d=2, 3$.
Two classes of fracture are distinguished based on the permeability magnitude
ratio between the fracture and its surrounding medium:
when the permeability in the fracture is (significantly) larger than in
its surrounding  medium, it is considered as a {\it conductive} fracture;
when the permeability in the fracture is (significantly)
smaller than in its surrounding medium, it is considered as a {\it blocking}
fracture.
The conductive fractures are treated using the classical hybrid-dimensional
approach of the interface model where pressure is assumed to be continuous
across the fracture interfaces, while the blocking fractures are treated using
the recent Dirac-$\delta$ function approach where normal component
of Darcy velocity is assumed to be continuous across the interface.
Due to the use of Dirac-$\delta$ function approach for the blocking fractures, our numerical scheme allows for nonconforming meshes with respect to the blocking fractures. This is the major novelty of our model and numerical discretization. Moreover, our numerical scheme produces locally conservative velocity approximations and leads to a symmetric positive definite linear system involving pressure degrees of freedom on the mesh skeleton only.
The performance of the proposed method is demonstrated by various benchmark
test cases in both two- and three-dimensions.
Numerical results indicate that the proposed scheme is highly competitive with existing methods in the literature.
\end{abstract}
\maketitle

\section{Introduction}
\label{sec:intro}

Numerical simulations of single- and multi-phase flows in porous media have many applications in contaminant transportation, oil recovery and underground radioactive waste deposit. Due to the highly conductive and blocking fractures in the porous media underground, it is still challenging to construct accurate numerical approximations \cite{R2,R3, R4}.

There are several commonly used mathematical models for simulating flows in porous media with conductive fractures, such as the dual porosity model \cite{DualPoro1,DualPoro2,geiger}, single porosity model \cite{SinglePoro1}, traditional discrete fracture model (DFM) \cite{FEMDFM1,FEMDFM2,SuperposeStiffnessMat,FEMDFM,DFMpaper1,FEMDFM3,FEMDFM4}, embedded DFM (EDFM) \cite{firstEDFM,secondEDFM,EDFM3,pEDFM,EDFM4,CrossShaped,EDFM5}, the interface models \cite{Alboin99,Interfaces2,Interfaces3,benchmark2} and extended finite element DFM (XDFM) based on the interface models \cite{XFEMDFM1,XFEMDFM2,ThesisXFEM,XFEMDFM3,XFEMDFM4}, finite element method based on Lagrange multipliers \cite{LMFEM2D,LMFEM2D2, LMFEM3D}, etc. Among the above methods, the traditional DFM and the interface models have been intensively studied in the past decades.

The DFM is based on the principle of superposition. It uses a hybrid dimensional representation of the Darcy's law, and treats the fractures as lower dimensional entries, with the thickness of the fracture as the dimensional homogeneity factor. The first DFM was introduced by Noorishad and Mehran \cite{FEMDFM1} in 1982 for single phase flows. Later, Baca et al. \cite{FEMDFM2} considered the heat and solute transport in fractured media. Subsequently, several significant numerical methods were applied to the DFM, such as the finite element methods \cite{SuperposeStiffnessMat,FEMDFM,DFMpaper1,FEMDFM3,FEMDFM4}, vertex-centered finite volume methods \cite{BoxDFM1,BoxDFM2,BoxDFM4,BoxDFM6}, cell-centered finite volume methods \cite{IntersectingFractures,CCDFM2,CCDFM3,CCDFM4,CCDFM5}, mixed finite element methods \cite{MFEMDGDFM1,MFEMDGDFM2,MFEMDGDFM3,MFEMDGDFM4,MFEMDGDFM6,MFEMDGDFM7,MFEMDGDFM8,MFEMDGDFM9}, discontinuous Galerkin methods
\cite{DGinterfaces2}.
All the above works are limited on conforming meshes, i.e. the fractures are aligned with the interfaces of the background matrix cells. Therefore, it may suffer from low quality cells. Recently, Xu and Yang introduced the line Dirac-$\delta$ functions \cite{YangJCP20} to represent the conductive fractures and reinterpreted the DFM (RDFM) on nonconforming meshes. The basic idea is to superpose the conductivity of the fracture to that of the matrix. The main contribution in \cite{YangJCP20} is to explicitly represent the DFM introduced in \cite{DFMpaper1} as a scalar partial differential equation. Therefore, with suitable numerical discretizations, such as the discontinuous Galerkin method, the RDFM can be applied to arbitrary meshes. To demonstrate that the RDFM is exactly the traditional DFM if the mesh is conforming, in \cite{YangJCP20} only finite element methods were considered. Therefore, local mass conservation was missing. Later, the enriched Galerkin and interior penalty discontinuous Galerkin methods were applied to RDFM in \cite{YangAWR21} and the contaminant transportation was also simulated.

Different from the traditional DFM, the interface model \cite{Alboin99,Interfaces2,Interfaces3,benchmark2} explicitly represent the fractures as interfaces of the porous media. Then the governing equation of the flow in the lower dimensional fracture was constructed. In the interface model, the matrix and fractures are considered as two systems, and the communication between them was given as the jump the normal velocity along the fractures. Therefore, different from RDFM, the interface model, though hanging nodes are allowable, cannot be applied to structured meshes and the fracture must be aligned with the interfaces of the meshes for the matrix. To fixed this limitation, the XDFM was proposed \cite{XFEMDFM1,XFEMDFM2,ThesisXFEM,XFEMDFM3,XFEMDFM4}. However, these methods may increase the degrees of freedom (DOFs) significantly, and can hardly be applied to fracture networks with high geometrical complexity \cite{FLEMISCH2018239}. As an alternative, the CutFEM \cite{CutFEM} can be applied to non-conforming meshes. It couples the fluid flow in all lower dimensional manifolds. However, this method requires the fractures to cut the domain into completely disjoint subdomains, thus it is not applicable for media with complicated fractures.

Most of the above ideas work for problems with conductive fractures. However, if the media contains blocking fractures, most methods may not be suitable. To fix this gap, the projection-based EDFM (pEDFM) was introduced in \cite{pEDFM,Jiang2017}. The effective flow area between adjacent matrix grids is computed as the difference between the original interface area and the projected area of the fracture segment. It will be zero if the fracture fully penetrates through the matrix cell. Olorode et al. \cite{Olorode2020} extended the pEDFM into three-dimensional compositional simulation of fractured reservoirs. However the pEDFM still cannot describe the complex multiphase flow behavior in the matrix blocks within barrier fractures. Another approach is to follow the interface model introduced in \cite{Interfaces5,Interfaces6,Boon2018,EG20}. However, as demonstrated above, the interface model can only handle hanging nodes, and the fractures must align with the interfaces of the background mesh. Recently, Xu and Yang extended the RDFM \cite{YangJCP20,YangAWR21} to problems with blocking fractures in \cite{Yang21}. The basic idea is to apply Ohm's law and superpose the resistance (the reciprocal of the permeability) of the blocking fracture to that of the matrix. Then a modified partial differential equation system was introduced and the local discontinuous Galerkin methods with suitable penalty were perfectly applied. 
If the problems contains only blocking fractures, the mixed finite element methods can easily be combined with RDFM.

In this paper, we combine the ideas in \cite{Alboin99} and \cite{Yang21}
to propose a novel model
for
single phase flows with both conductive and blocking fractures.
In particular, the conductive fractures are modeled by using the interface model \cite{Alboin99} where pressure continuity is enforced across the conductive fractures, and the blocking fractures are modelled as resistance terms involving Dirac-$\delta$ functions following the main idea in\cite{Yang21}.
The separate treatment of conductive and blocking fractures, and the seamless combination of the conductive fracture interface model and the blocking fracture Dirac-$\delta$ function approach is the major novelty of our proposed model.
We further discretize this new model using a hybrid-mixed finite element method, which produces locally conservative velocity approximations and leads to a symmetric positive definite linear system with globally coupled degrees of freedom (DOFs) only those of pressure on the mesh skeletons.
Moreover, due to the use of Dirac-$\delta$ function approach for blocking fractures, the method does not require any mesh conformity with respect to the blocking fractures, which is the major novelty of our proposed
scheme. We believe our approach is the simplest non-conforming mesh approach to blocking fractures that still yield locally conservative velocity approximations.
We note that mesh conformity with respect to the conductive fractures is still required for our method, which is typical for interface models.
We numerically demonstrate that our hybrid-mixed finite element scheme is highly competitive both in terms of computational efficiency and accuracy.
We finally emphasis that the proposed hybrid-mixed formulation is different from the mixed method in \cite{Boon2018} due to the use of different model for the interface conditions.
We believe that our model is significantly simpler for complex fracture networks since we only use one matrix domain and one (codimension 1) conductive fracture domain throughout, while the mixed method formulation \cite{Boon2018} needs to split the matrix and fracture domains into multiple disjoint sub-domains and require the modeling of codimension 1-3 fracture flows, which might be very tedious to perform for complex fracture networks.

The rest of the paper is organized as follows.
In Section \ref{sec:model}, we present the hybrid-dimensional model under
consideration.
We then formulate in Section \ref{sec:method}
the hybrid-mixed finite element discretization of the model
proposed in Section \ref{sec:model}.
Numerical results for various benchmark test cases are presented in Section \ref{sec:num}.
We conclude in Section \ref{sec:conclude}.

\section{The hybrid-dimensional model}
\label{sec:model}
%
\subsection{Notation}
We consider a bounded open domain
$\Omega_m\subset \mathbb{R}^d$, $d=2,3$, which contains
several $(d-1)$-dimensional
conductive or blocking fractures.
For simplicity, the fractures are assumed to be hyperplanes with smooth boundaries.
We denote $\Omega_{c}$ as the $(d-1)$-dimensional open set
containing all the conductive fractures, and
$\Omega_{b}$ as the set containing all the  blocking fractures.
Assume the $(d-1)$-dimensional domain boundary $\partial \Omega_m = \Gamma_D\cup
\Gamma_N$, with
$\Gamma_D\cap \Gamma_N=\emptyset$.
Furthermore, we denote the following sets of $(d-2)$-dimensional
boundaries (intersections) associated with the set of conductive fractures  $\Omega_{c}$:
\begin{itemize}
  \item $\Gamma_{cc}$ is the set containing the intersections among
    conductive fractures.
  \item $\Gamma_{cb}$ is the set containing the  intersections between
2

    conductive and blocking fractures.
  \item $\Gamma_{cm}$ is the set containing the intersections between
    conductive fractures and domain boundary $\partial \Omega_m$,
    which is further split to
    $\Gamma_{cm} = \Gamma_{cm}^N\cup \Gamma_{cm}^D$ with
    $\Gamma_{cm}^N\in \Gamma_N$ and $\Gamma_{cm}^D\in \Gamma_D$.
  \item $\Gamma_{ci}$ is the boundary of  $\Omega_c$ that does not intersect with
    the  domain boundary $\partial \Omega_m$.
\end{itemize}
We set $\Gamma_c= \Gamma_{cc}\cup\Gamma_{cb}\cup\Gamma_{cm}\cup\Gamma_{ci}$ as the
collections of all intersections of $\Omega_c$.
An illustration of a typical hybrid-dimensional domain in two-dimensions is
given in Figure \ref{fig:geo}.
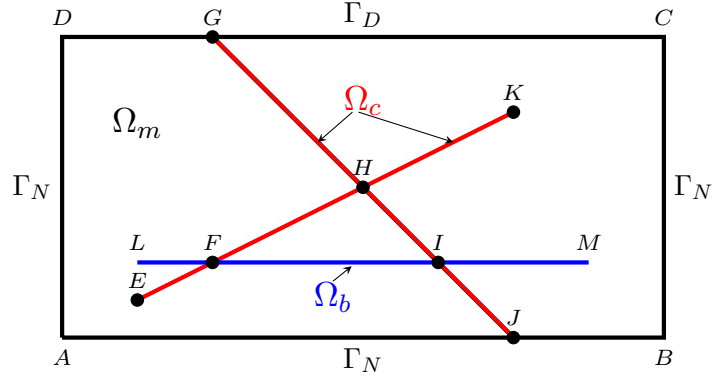
\begin{figure}[ht]
\centering
  \begin{tikzpicture}
    \draw[ultra thick,draw=black]
    (0, 0)
    to (6, 0)
    to (2, 4)
    to (0, 4)
    to (0, 0)
    ;
  \draw[ultra thick,draw=black]
    (6, 0)
    to (8, 0)
    to (8, 4)
    to (2, 4)
    to (6, 0)
    ;
  \draw[ultra thick, draw=red]
    (6,0) to (2,4);
  \draw[ultra thick, draw=red]
    (1,0.5) to (6,3);
  \draw[ultra thick, draw=blue]
    (1,1) to (7,1);
  \draw[ fill=black] (2, 1) circle (.5ex);
  \draw[ fill=black] (5, 1) circle (.5ex);
  \draw[ fill=black] (2, 4) circle (.5ex);
  \draw[ fill=black] (6, 0) circle (.5ex);
  \draw[ fill=black] (1, 0.5) circle (.5ex);
  \draw[ fill=black] (6, 3) circle (.5ex);
  \draw[ fill=black] (4, 2) circle (.5ex);
  \node at (1,0.5)[above,scale=1.2] {\tiny $E$};
  \node at (2,1)[above,scale=1.2] {\tiny $F$};
  \node at (2,4)[above,scale=1.2] {\tiny $G$};
  \node at (4,2)[above,scale=1.2] {\tiny $H$};
  \node at (5,1)[above,scale=1.2] {\tiny $I$};
  \node at (6,0)[above,scale=1.2] {\tiny $J$};
  \node at (6,3)[above,scale=1.2] {\tiny $K$};
  \node at (1,1)[above,scale=1.2] {\tiny $L$};
  \node at (7,1)[above,scale=1.2] {\tiny $M$};
  \node at (0,0)[below,scale=1.2] {\tiny $A$};
  \node at (8,0)[below,scale=1.2] {\tiny $B$};
  \node at (8,4)[above,scale=1.2] {\tiny $C$};
  \node at (0,4)[above,scale=1.2] {\tiny $D$};
  \node at (1,2.5)[above,scale=1.2] {$\Omega_m$};
  \node at (4.0,2.8)[above,scale=1.2] {{\color{red}$\Omega_c$}};
  \node at (3.6,0.9)[below,scale=1.2] {{\color{blue}$\Omega_b$}};
  \draw[-stealth] (3.86, 3.0) -- (3.4,2.6);
  \draw[-stealth] (3.94, 3.0) -- (5.2,2.6);
  \draw[-stealth] (3.6, 0.75) -- (3.85,0.95);
  \node at (4,4)[above,scale=1] {$\Gamma_D$};
  \node at (4,0)[below,scale=1] {$\Gamma_N$};
  \node at (0,2)[left,scale=1] {$\Gamma_N$};
  \node at (8,2)[right,scale=1] {$\Gamma_N$};
\end{tikzpicture}
\caption{
  A typical two dimensional fractured domain $\Omega_m$ (the rectangular domain).
 The domain boundary $\Gamma_D = \{CD\}$, $\Gamma_N=\{AB\}\cup \{BC\}\cup
 \{AD\}$, where $\{AB\}$ denotes the line segment connecting nodes $A$
  and $B$.
Here $\Omega_c=\{EK\}\cup \{JG\}$,
 $\Omega_b = \{LM\}$,
 $\Gamma_{cc}= H$, $\Gamma_{cb}=F\cup I$,
 $\Gamma_{cm}^D = G$, $\Gamma_{cm}^N=J$, and
 $\Gamma_{ci} = E\cup K$.
 }
\label{fig:geo}
\end{figure}

We denote $\bld n_\Gamma$ as a uniquely oriented unit normal vector
on a $(d-1)$-dimensional interface/boundary $\Gamma$, and denote
$\bld \eta_\Gamma$ as the {\it in-plane} unit (outer) normal
vector on the
$(d-2)$ dimensional boundary  $\partial \Gamma$ of $\Gamma$, see Figure
\ref{fig:normal}.
\begin{figure}[ht]
\centering
\begin{tabular}{cc}
\tdplotsetmaincoords{70}{110}
\begin{tikzpicture}[tdplot_main_coords,font=\sffamily]
\draw[-latex, thick] (0,0,0) -- (1,-1,0) node[left] {$\bld n_\Gamma$};
\draw[thick,red,opacity=0.4] (-3,-3,0) -- (3,3,0);
 \draw[fill=black] (-3, -3,0) circle (.5ex);
 \draw[fill=black] (3, 3,0) circle (.5ex);
\draw[-latex, thick] (-3,-3,0) -- (-4,-4, 0) node[left] {$\bld \eta_\Gamma$};
\draw[-latex, thick] (3,3,0) -- (4,4, 0) node[left] {$\bld \eta_\Gamma$};
\node at (-2,-1.5, 0) {$\Gamma$};
\node at (-3.8,-3.2, 0) {$\partial\Gamma$};
\end{tikzpicture}
&
\tdplotsetmaincoords{70}{110}
\begin{tikzpicture}[tdplot_main_coords,font=\sffamily]
\draw[-latex, thick] (0,0,0) -- (0,0,1.5) node[left] {$\bld n_\Gamma$};
\draw[fill=red,opacity=0.1] (-3,-3,0) -- (-3,3,0) -- (3,3,0) -- (3,-3,0) -- cycle;
\draw[-latex, thick] (0,-3,0) -- (0,-4, 0) node[left] {$\bld \eta_\Gamma$};
\draw[-latex, thick] (0,3,0) -- (0,4, 0) node[right] {$\bld \eta_\Gamma$};
\draw[-latex, thick] (-3,0,0) -- (-4,0, 0) node[left] {$\bld \eta_\Gamma$};
\draw[-latex, thick] (3, 0,0) -- (4,0, 0) node[left] {$\bld \eta_\Gamma$};
\node at (-0.5,-1.5, 0) {$\Gamma$};
\node at (-3.2,-3.2, 0) {$\partial\Gamma$};
\end{tikzpicture}
\\
  (a) 2D case & (b) 3D case
\end{tabular}
\caption{
Normal direction $\bld n_\Gamma$ and in-plane normal direction $\bld \eta_\Gamma$ for a
$(d-1)$-dimensional interface/boundary $\Gamma$.
 Left: $d=2$, $\Gamma$ is a line segment. Right: $d=3$,
 $\Gamma$ is a planar quadrangle.
 }
\label{fig:normal}
\end{figure}
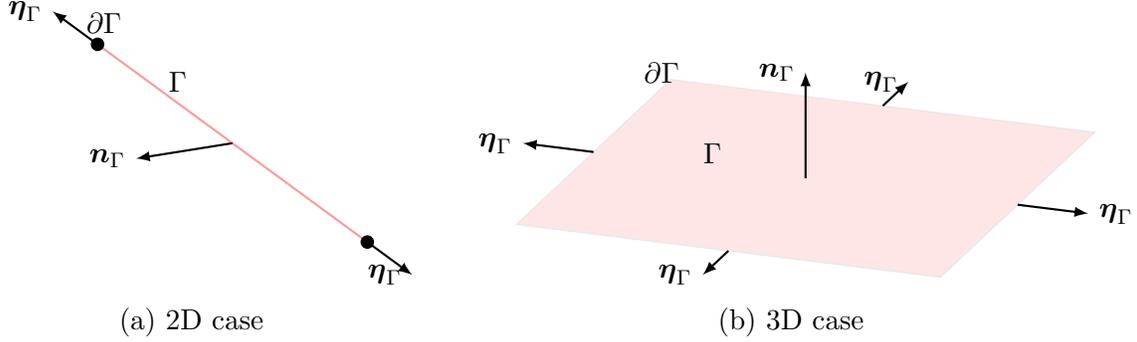

Let $\epsilon$ be the thickness of the fractures, which is assumed to be
a small positive constant for simplicity.
Let $\mathbb{K}_m$ be the permeability tensor of the domain
excluding the fractures $\Omega_m\backslash \{\Omega_c\cup \Omega_b\}$,
$K_b\ll \mathbb{K}_m$ be the (scalar) permeability in the normal direction of blocking
fractures $\Omega_b$, and $\mathbb{K}_c\gg \mathbb{K}_m$ be the permeability
tensor in the tangential direction of the
conductive fractures $\Omega_c$.

\subsection{The hybrid-dimensional flow model}
The following hybrid-dimensional model is a combination of the conductive
fracture treatment in \cite{Alboin99}  and blocking fracture treatment in
\cite{Yang21}.
In the bulk domain $\Omega_m\backslash\Omega_c$ excluding
conductive fractures, we use the following barrier model:
\begin{subequations}
  \label{model}
  \begin{alignat}{2}
  \label{model1}
    (\mathbb{K}_m^{-1}+\frac{\epsilon}{K_b}\delta_{\Omega_b} \bld n_{\Omega_b}\otimes \bld
    n_{\Omega_b})\bld u = &\; -\nabla p,&&\quad \text{ in }\Omega_m\backslash\Omega_c,  \\
  \label{model2}
    \nabla \cdot \bld u = &\; f, &&\quad \text{ in }\Omega_m\backslash\Omega_c,
  \end{alignat}
where $\bld u$ is the Darcy velocity, $p$ is the pressure, $f$ is the volume source term, $\delta_{\Omega_b}$ is the Dirac-$\delta$ function that takes values $\infty$
on the blocking fractures $\Omega_b$ and zero elsewhere, and
$\bld n_{\Omega_b}$ is the unit normal vector on $\Omega_b$.
Within the conductive fractures excluding
intersections $\Omega_c\backslash\Gamma_c$, we use the following
$(d-1)$-dimensional Darcy's law:
  \begin{alignat}{2}
  \label{model3}
    (\epsilon\mathbb{K}_c)^{-1}\bld u_c = &\; -\nabla_{\Gamma}\, p_c,&&\quad \text{ in }
    \Omega_c\backslash\Gamma_c,  \\
  \label{model4}
    \nabla_{\Gamma} \cdot \bld u_c = &\; \jmp{\bld u}, &&\quad \text{ in }\Omega_c\backslash\Gamma_c,
  \end{alignat}
where $\bld u_c$ is the (tangential) Darcy velocity in the conductive
fractures, $p_c$ is the associated pressure,
and the velocity jump $\jmp{\bld u} = (\bld u^+-\bld u^-)\cdot\bld n_\Gamma$
represents the mass exchange between the conductive fractures and the
surrounding media, where $\bld u^{\pm}(\bld x) = \lim_{\tau\rightarrow 0^{\pm}}
\bld u(\bld x-\tau\bld n_\Gamma)$ for all $\bld x \in \Omega_c$ is the bulk
Darcy velocity evaluated on one side of the conductive fractures.
Moreover, $\nabla_{\Gamma}$ and $\nabla_{\Gamma}\cdot$ are the usual surface
gradient and surface divergence operators.
The above hybrid-dimensional system is closed with the following set of
boundary/interface conditions:
\begin{alignat}{2}
  \label{bc0}
  p = &\; p_D, && \quad \text{ on }\Gamma_D,\\
  \label{bc1}
  \bld u\cdot\bld n = &\; q_N, && \quad \text{ on }\Gamma_N,\\
  \label{bc2}
  p = &\; p_c, && \quad \text{ on }\Omega_c,\\
  \label{bc3}
  \jmp{\bld u_c} = &\; 0, && \quad \text{ on }\Gamma_{cc},\\
  \label{bc4}
  p_c = &\; p_D, && \quad \text{ on }\Gamma_{cm}^D,\\
  \label{bc5}
  \bld u_c\cdot \bld \eta_\Gamma = &\; 0, && \quad \text{ on
  }\Gamma_{cb}\cup\Gamma_{cm}^N\cup\Gamma_{ci},
\end{alignat}
where \eqref{bc2} ensures continuity of bulk pressure across conductive
fractures,
the  no-flow boundary condition in  \eqref{bc5}
is imposed on the intersections $\Gamma_{cb}$, $\Gamma_{cm}^N$ and
$\Gamma_{ci}$,
and the jump term in \eqref{bc3} is
\[
 \jmp{\bld u_c}\Big|_{e}:= \sum_{\Gamma\subset \Omega_c\backslash\Gamma_c, \; e\in
 \overline{\Gamma}}
 \bld u_c|_{\Gamma}\cdot \bld \eta_\Gamma, \quad \forall e\in \Gamma_{cc},
\]
which represents mass conservation along intersections $\Gamma_{cc}$.
Note in particular that each conductive fracture containing the intersection $e$
appears exactly twice in the above summation, and
the in-plane normal velocity on  the fracture is allowed to be discontinuous
along the intersection $e$. For example, the jump $\jmp{\bld u_c}|_H$ at node $H$ in the configuration
in Figure \ref{fig:geo} is
\[
 \jmp{\bld u_c}|_{H}:=
 \sum_{\Gamma\in \left\{\{EH\}, \{HK\}, \{GH\}, \{HJ\}\right\}}
 \bld u_c|_{\Gamma}\cdot\bld \eta_{\Gamma}.
\]
\end{subequations}
We note that in the above model \eqref{model}, the flow in the tangential
direction in the blocking fractures is completely ignored as the permeability therein is much smaller than that of the surroundings,
on the other hand, the flow in the normal direction is ignored
on conductive fractures by the pressure
continuity condition \eqref{bc2} since
the permeability is much larger than that of the surroundings
and the fluid has a tendency to flow along the tangential direction therein.

\subsection{The hybrid-dimensional transport model}
We now consider a scalar quantity $c$ that is transported through the porous medium subject to the velocity fields in the flow model
\eqref{model}.
Here $c$ usually represents the concentration of a generic passive tracer.
Similar to the flow treatment in the previous subsection, transport inside the blocking fractures is ignored.
The concentrations $c$ in the matrix and $c_c$ in the
conductive fractures are governed by the following advection equations, see e.g. \cite{AJRS02,FS11,benchmark2},
\begin{subequations}
\label{transport}
\begin{alignat}{2}
\phi_m\frac{\partial c}{\partial t} +\nabla\cdot(\bld u c) &= c f,
&&\quad  \text{in } \Omega_m\backslash\Omega_c\times (0, T],  \\
\epsilon \phi_c\frac{\partial c_c}{\partial t} +\nabla_{\Gamma}\cdot(\bld u_c c_c) - \jmp{c \bld u}& = 0,
&& \quad  \text{in } \Omega_c\times (0, T],
\end{alignat}
with the following initial, interface, and boundary conditions
\begin{alignat}{2}
\label{cc}
c = &c_c \quad\text{on }\Omega_c\times (0, T],\\
c = &c_0 \quad\text{on }\Omega\times{0},\quad
&&
c_c = c_{c,0} \quad\text{on }\Omega_c\times{0},\\
c = &c_B \quad\text{on }\partial\Omega_{in}\times (0,T],
\quad &&c_c = c_{c,B} \quad\text{on }\Gamma_{in}\times (0,T],
\end{alignat}
\end{subequations}
where $\{\phi_m, c_0, c_B, \partial\Omega_{in}\}$ and
$\{\phi_c, c_{c,0}, c_{c,B}, \Gamma_{in}\}$ represent the \{porosity,
initial concentration, inflow concentration, and inflow boundary\} in the matrix and conductive fractures, respectively.
Observe that concentration continuity \eqref{cc} across the conductive fractures are enforced in the model \eqref{transport}.

\newcommand{\Vh}{\bld V_h}
\newcommand{\Wh}{W_h}
\newcommand{\Mh}{M_h}

\newcommand{\Vhc}{\bld V_h^c}
\newcommand{\Whc}{W_h^c}
\newcommand{\Mhc}{M_h^c}

\section{The hybrid-mixed finite element method}
\label{sec:method}
\subsection{Preliminaries}
Let $\Th:=\{K\}$ be a conforming simplicial
triangulation
 of the domain $\Omega_m$.
 Let $\Eh$ be the collections of $(d-1)$-dimensional facets
 (edges for $d=2$, faces for $d=3$)
 of $\Omega_m$.
 Assume the mesh is fully fitted with respect to the conductive fractures, i.e.,
 $\Th^c:=\Omega_c\cap \Eh$ is a $(d-1)$-dimensional simplicial triangulation of
 the domain $\Omega_c$.
Here the mesh $\Th$ is allowed to be unfitted with respect to the blocking
 fractures.
Moreover, we denote $\Eh^c$ as the collection of
$(d-2)$-dimensional facets of $\Th^c$ (vertices for $d=2$, edges for $d=3$).

We use the lowest-order hybrid-mixed finite element methods to discretize
the model \eqref{model}. The following finite element spaces
will be needed:
\begin{subequations}
  \label{space}
 \begin{align}
   \label{space-u}
   \Vh :=&\; \{\bld v\in [L^2(\Th)]^d:\;
   \bld v|_K\in RT_0(K),\quad \forall K\in\Th\},\\
   \Wh :=&\; \{w\in L^2(\Th):\;
   w|_K\in P_0(K),\quad \forall K\in\Th\},\\
   \Mh :=&\; \{\mu\in L^2(\Eh):\;
   \mu|_F\in P_0(F),\quad \forall F\in\Eh\},\\
   \Vhc :=&\; \{\bld v_c\in [L^2(\Th^c)]^d:\;
   \bld v|_F\in RT_0(F),\quad \forall F\in\Th^c\},\\
   \Mhc :=&\; \{\mu\in L^2(\Eh^c):\;
   \mu|_E\in P_0(E),\quad \forall E\in\Eh^c\},
 \end{align}
\end{subequations}
where $RT_0(S)$ is the Raviart-Thomas space of lowest order
on a simplex $S$, and $P_0(S)$ is the space of constants.

We denote the following inner products:
\begin{alignat*}{2}
  (\phi, \psi)_{\Th}: =&\;\sum_{K\in \Th}\int_{K}\phi\,\psi\,\mathrm{dx},&&\quad
  \quad
  \langle\phi, \psi\rangle_{\partial\Th}: =\;\sum_{K\in
\Th}\int_{\partial K}\phi\,\psi\,\mathrm{ds},\\
  \langle\phi, \psi\rangle_{\Th^c}: =&\;\sum_{F\in
  \Th^c}\int_{F}\phi\,\psi\,\mathrm{ds},&&\quad
  \quad
  [\phi, \psi]_{\partial\Th^c}: =\;\sum_{F\in
\Th^c}\int_{\partial F}\phi\,\psi\,\mathrm{dr},
\end{alignat*}
where $\mathrm{dx}$ is for $d$-dimensional integration, $\mathrm{ds}$ is for $(d-1)$-dimensional
integration, and $\mathrm{dr}$ is for $(d-2)$-dimensional integration.
When $d=2$, $\int_{\partial F}{\phi\,\psi}\mathrm{dr}$ is simply
the sum of point
evaluations at the two end points of a line segment $F$.

\subsection{The hybrid-mixed method for the flow model}
The hybrid-mixed method for the hybrid-dimensional model \eqref{model}
is given as follows:
Find $(\bld u_h, p_h, \widehat{p}_h, \bld u_h^c, \widehat{p}_h^c)
\in \Vh\times \Wh\times \Mh\times \Vhc\times \Mhc$
with $\widehat{p}_h|_{\Gamma_D} =\mathbb{P}_0(p_D)$
and
$\widehat{p}_h^c|_{\Gamma_{cm}^D} =\mathbb{P}_0(p_D)$,
where $\mathbb{P}_0$ denotes the projection onto piecewise constants,
such that
\begin{subequations}
  \label{fem}
  \begin{align}
    \label{fem1}
    (\mathbb{K}_m^{-1}\bld u_h, \bld v_h)_{\Th}
    +\int_{\Omega_b}\frac{\epsilon}{K_b}(\bld u_h\cdot\bld n)
    (\bld v_h\cdot\bld n)\mathrm{ds}
    -(p_h, \nabla\cdot\bld v_h)_{\Th}
    +\langle\widehat{p}_h, \bld v_h\cdot\bld n\rangle_{\partial\Th} =&\; 0,\\
    \label{fem2}
    (\nabla\cdot\bld u_h, q_h)_{\Th}-(f, q_h)_{\Th} =&\; 0,\\
    \label{fem3}
    -\langle\bld u_h\cdot\bld n, \widehat{q}_h\rangle_{\partial\Th}
    + \langle\nabla_{\Gamma}\cdot\bld u_h^c, \widehat{q}_h\rangle_{\Th^c}
    +\int_{\Gamma^N}q_N\,\widehat q_h\,\mathrm{ds}
    =&\;0,\\
    \label{fem4}
    \langle(\epsilon\mathbb{K}_c)^{-1}\bld u_h^c, \bld
    v_h^c\rangle_{\Th^c}
    - \langle\widehat{p}_h, \nabla_{\Gamma}\cdot\bld v_h^c\rangle_{\Th^c}
    + [\widehat{p}_h^c, \bld v_h^c\cdot\bld \eta]_{\partial\Th^c}
    + \int_{\Gamma_{cb}}\alpha(\epsilon\mathbb{K}_c)^{-1}(\bld u_h^c\cdot\bld \eta)
    (\bld v_h^c\cdot\bld \eta)\,\mathrm{dr}
    =&\; 0,\\
    \label{fem5}
    -[\bld u_h^c\cdot\bld \eta, \widehat{q}_h^c]_{\partial\Th^c}
    =&\; 0,
  \end{align}
\end{subequations}
for all
$(\bld v_h, q_h, \widehat{q}_h, \bld v_h^c, \widehat{q}_h^c)
\in \Vh\times \Wh\times \Mh\times \Vhc\times \Mhc$
with $\widehat{q}_h|_{\Gamma_D} =\widehat{q}_h^c|_{\Gamma_{cm}^D}=0$,
where $\alpha>0$ is a penalty parameter for the implementation of the
no-flow boundary condition \eqref{bc5} on $\Gamma_{cb}$.
In our numerical implementation, we take $\alpha = 10^6$.

We show that the scheme \eqref{fem} is formally consistent
with the hybrid-dimensional model \eqref{model}:
\begin{itemize}
  \item [(1)] Equation \eqref{fem1} is
    a discretization of the Darcy's law \eqref{model1} in the bulk
    using integration-by-parts
    and the following property of Dirac-$\delta$ function:
    \[
      \int_{\Omega_m}\delta_{\Omega_b}\phi\mathrm{dx}
      = \int_{\Omega_b}\phi\,\mathrm{ds}.
    \]
  \item [(2)] Equation \eqref{fem2} is
    the discretization of mass conservation \eqref{model2} in the bulk.
  \item [(3)] Equation \eqref{fem3} simultaneously enforces (i) the continuity of
    normal velocity $\bld u_h\cdot\bld n$ across interior element boundaries
    $\Eh\backslash(\Th^c\cup\Gamma_N)$, (ii) the boundary condition
    \eqref{bc1} on $\Gamma_N$, and (iii) mass conservation \eqref{model4} within the
    conductive fractures in  $\Th^c$.
  \item [(4)] Equation \eqref{fem4} is a discretization of the Darcy's law
    \eqref{model3} on the conductive fractures $\Th^c$, where the pressure
    continuity condition \eqref{bc3} is also strongly enforced as
    $\widehat{p}_h$ both represents the bulk pressure on the element boundary
    $\Eh$ and the pressure within the conductive fracture $\Th^c$.
    Moreover, the last term in \eqref{fem4} is a penalty formulation of the
    no-flow boundary condition \eqref{bc5} on $\Gamma_{cb}$. Note that
    $\Gamma_{cb}$ is allowed to be not aligned with the facets of $\Th^c$.
  \item [(5)] Equation \eqref{fem5} is a
    transmission condition that simultaneously enforces (i)
    continuity of in-plane normal velocity $\bld u_h^c\cdot\bld \eta$
    on interior facets $E_h^c\backslash\{\Gamma_{cc}\cup
      \Gamma_{cm}^N\cup\Gamma_{ci}
    \}$,
    (ii) the mass conservation \eqref{bc3} on the intersections $\Gamma_{cc}$
    (iii) the no-flow boundary condition \eqref{bc5} on
    $\Gamma_{cm}^N$ and $\Gamma_{ci}$.
  \item [(6)] The Dirichlet boundary condition \eqref{bc0} and \eqref{bc4}
    are imposed strongly through the corresponding degrees of freedom (DOFs) on
    $\widehat p_h$ and $\widehat p_h^c$, respectively.
\end{itemize}

The following result further shows that the scheme \eqref{fem} is well-posed.
\begin{theorem}
  Assume the measure of the Dirichlet boundary $\Gamma_D$ is not empty, then
  the solution to the scheme \eqref{fem} exists and is unique.
\end{theorem}
\begin{proof}
  Since the equations in \eqref{fem} leads to a square linear system, we only
  need to show uniqueness.
  Now we assume the source terms in \eqref{fem} vanishes, i.e.,
  $f=p_D=g_N=0$.
  Taking test function to be the same as trial functions in \eqref{fem}
  and adding, we get
  \[
    (\mathbb{K}_m^{-1}\bld u_h, \bld u_h)_{\Th}
    +\int_{\Omega_b}\frac{\epsilon}{K_b}(\bld u_h\cdot\bld n)^2\,\mathrm{ds}
    +\langle(\epsilon\mathbb{K}_c)^{-1}\bld u_h^c, \bld u_h^c\rangle_{\Th^c}
  =0.
  \]
  Hence, $\bld u_h=\bld u_h^c=0$.
 Since $\bld u_h=0$,
  the inf-sup stability of the $RT_0$-$P_0$ finite element pair implies
  that $p_h= \widehat{p}_h = C$ from \eqref{fem1} where $C$ is a constant.
  Since $\Gamma_D$ is not empty and $p_D=0$, we get the constant $C=0$.
  Finally, restricting equation  \eqref{fem4} to a
  single element $F\in \Th^c$ and using the fact that
$\bld u_h^c=0$ and $\widehat p_h=0$, we get
\[
  \int_{\partial F}\widehat p_h^c\bld v_h^c\cdot\bld \eta\,\mathrm{ds} = 0,
  \quad \forall \bld v_h^c\in RT_0(F),
\]
  which then implies that $\widehat p_h^c=0$.
  This completes the proof.
\end{proof}

\subsection{Static condensation and linear system solver}
The linear system \eqref{fem} can be efficiently solved via static condensation,
where the DOFs for $\bld u_h$, $p_h$, and
$\bld u_h^c$ can be locally eliminated, resulting in a coupled global linear system for
the DOFs for $\widehat {p}_h$ and $\widehat {p}_h^c$, which is
symmetric and positive definite.
Efficient linear system solvers for the resulting condensed system is an interesting topic
where one could design efficient decoupling algorithms or robust
monolithic preconditioners.
Here we simply use a sparse direct solver in the computation and postpone the
detailed study
of linear system solvers to our future work.
\subsection{Local pressure postprocessing}
We use the following well-known local (piecewise linear) pressure postprocessing  to improve the accuracy of pressure approximation in the bulk:
find
\[
p_h^*\in W_h^* :=\; \{w\in L^2(\Th):\;
   w|_K\in P_1(K),\quad \forall K\in\Th\},
\]
where $P^1(K)$ is the space of linear polynomials on element $K$,
such that
\begin{subequations}
  \label{postprocess}
\begin{align}
  \label{postprocess1}
(\nabla p_h^*, \nabla q_h^*)_{\Th}
= &\;- (\mathbb{K}_m^{-1}\bld u_h, \nabla q_h^*)_{\Th},\\
(p_h^*, 1)_{\Th}
= &\;
(p_h, 1)_{\Th},
\end{align}
\end{subequations}
for all $q_h^*\in W_h^*$.

\subsection{The hybrizied finite volume method for the transport model}
We consider a standard cell-centered, first-order upwinding finite volume scheme for the transport model \eqref{transport}, coupled with
the implicit Euler method for the temporal discretization.
We hybridize the cell-centered finite volume scheme so that the coupled unknowns live on the mesh skeletons, which simplifies the definition of upwinding fluxes on the conductive fracture interactions (e.g. point $H$ in Figure \ref{fig:geo}).
Hence we use piecewise constant spaces to
approximate the matrix concentration $c_h \in \Wh$ on the mesh $\Th$,
the matrix concentration $\widehat{c}_h \in \Mh$ on the matrix mesh skeleton $\Eh$, and the fracture concentration $\widehat{c}_{c,h}\in M_h^c$
on the fracture mesh skeleton $\Eh^c$.

The hybridized finite volume scheme with implicit Euler temporal discretization is given as follows:
given data $(c_h^{n-1}, \widehat{c}_h^{n-1})\in W_h\times \Mh$ at time $t^{n-1}$,
find $(c_h^{n}, \widehat{c}_h^{n}, \widehat{c}_{c,h}^{n})\in W_h\times \Mh\times M_h^c$
at time $t^n:=t^{n-1}+\Delta t$
with $\widehat{c}_h^n|_{\partial\Omega_{in}} = P_0(c_B(t^n))$ and
$\widehat{c}_{c,h}^n|_{\Gamma_{in}} = P_0(c_{c,B}(t^n))$
such that
\begin{subequations}
\label{transport-eq}
\begin{alignat}{2}
\left(\phi_m\frac{c_h^n-c_h^{n-1}}{\Delta t}, d_h\right)_{\Th}
+\langle\bld u_h\cdot\bld n \widehat{c}_h^{n,*},d_h
\rangle_{\partial\Th}
=&\;
\left(c_h^n f, d_h\right)_{\Th},
\\
-\langle\bld u_h\cdot\bld n \widehat{c}_h^{n,*},\widehat{d}_h
\rangle_{\partial\Th}
+
\left\langle\epsilon\phi_c\frac{\widehat{c}_{h}^n-\widehat{c}_h^{n-1}}{\Delta t}, \widehat{d}_h\right\rangle_{\Th^c}
+[\bld u_h^c\cdot\bld \eta \widehat{c}_{c,h}^{n,*},\widehat{d}_h
]_{\partial\Th^c} = &\;0,\\
[\bld u_h^c\cdot\bld \eta \widehat{c}_{c,h}^{n,*},\widehat{d}_{c,h}
]_{\partial\Th^c} = &\;0,
\end{alignat}
for all
$(d_h, \widehat{d}_h, \widehat{d}_{c,h})\in W_h\times \Mh\times M_h^c$
with $\widehat{d}_h|_{\partial\Omega_{in}} = 0$ and
$\widehat{d}_{c,h}|_{\Gamma_{in}} = 0$,
where the upwinding fluxes are given as follows:
\begin{align}
\widehat{c}_h^{n,*}|_{\partial K} = \left\{
\begin{tabular}{ll}
$c_h^n$     & if $\bld u_h\cdot\bld n_K >0$, \\[1ex]
$\widehat c_h^n$     & if $\bld u_h\cdot\bld n_K \le0$,
\end{tabular}
\right. \\
\widehat{c}_{c,h}^{n,*}|_{\partial F} = \left\{
\begin{tabular}{ll}
$\widehat c_h^n$     & if $\bld u_h^c\cdot\bld \eta_F >0$, \\[1ex]
$\widehat c_{c,h}^n$     & if $\bld u_h^c\cdot\bld \eta_F \le0$.
\end{tabular}
\right.
\end{align}
\end{subequations}

\subsection{Remarks on the mesh restrictions and comparison with existing methods}
The proposed flow and transport solvers \eqref{fem}, \eqref{transport-eq} require the mesh to be fitted to the conductive fractures, while allowing for an unfitted treatment of the blocking fractures.
While the derivation of numerical schemes that work on fully unfitted meshes is beyond the scope of this paper, here we propose a simple mesh postprocessing technique to convert a general unfitted background matrix mesh to an {\it immersed} mesh that is fitted to all the fractures.
Similar immersing mesh techniques were used for interface problems \cite{IH11,FT14,ABLR15,CWW17}.
Below we illustrate the procedure of immersing a single fracture to an unfitted tetrahedral mesh in 3D:
\begin{itemize}
    \item [(i)] Represent the fracture geometry as the zero level set of a continuous piecewise linear function $\phi_h$ on the background mesh.
    Perturb $\phi_h$ slightly if necessary to avoid fracture pass through the background mesh nodes.
    \item [(ii)] Loop over the background mesh edges, find the cut edges where $\phi_h$ has opposite sign on the two edge endpoints. For each cut edge, compute the coordinates of the cut vertex $v_c$ where $\phi_h(v_c) = 0$, and add $v_c$ to the mesh nodes.
\item [(iii)] Loop over the background mesh faces, find the cut faces
which contains the cut vertices.
Order the cut vertices based on their vertex label number.
Loop over the cut vertices, for each (sub-)face that contains the
cut vertex, split the (sub-)face by 2 by connecting the cut vertex with the opposite (sub-)face node.
\item [(iv)] Loop over the background mesh elements, find the cut elements
which contains the cut vertices.
Order the cut vertices based on their vertex label number.
Loop over the cut vertices, for each (sub-)element that contains the
cut vertex, split the (sub-)element by 2 by connecting the cut vertex with the opposite two (sub-)element nodes that are not aligned with the cut edge.
\end{itemize}
The above recursive bisection procedure guarantees that the fracture lies on the boundary of the generated immersed mesh.
The case with multiply intersecting fractures can be treated by recursion.
Here we note that the generated immersed mesh is usually highly anisotropic since the background mesh is completely independent of the fracture configurations. Our numerical results in the next section suggest that the hybrid-mixed method \eqref{fem} works well on these anisotropic immersed meshes. Typical 2D immersed meshes for complex fracture configurations are given in Figure~\ref{fig:complex0} and  Figure~\ref{fig:realX} below.

We now briefly compare our proposed fractured flow solver \eqref{fem}
with some existing schemes in \cite{Berre_2021}, which were used to solve a series of 4 benchmark problems in 3D fractured porous media flow.
Among the 17 schemes in \cite[Table 1]{Berre_2021}, 7  were shown to yield no significant deviations for all the tests, see \cite[Figure 18]{Berre_2021}, which include the
multi-point flux approximation (UiB-MPFA),  the lowest order mixed virtual element method (UiB-MVEM), and the lowest order Raviart-Thomas mixed finite element method (UiB-RT0) mainly developed by the research group in the
University of Bergen \cite{UiB0, UiB1, Boon2018},
the MPFA scheme (USTUTT-MPFA) and the two-point flux approximation scheme (USTUTT-TPFA\_Circ) developed by Flemisch et al. \cite{USTUTT},
the mimetic finite difference method (LANL-MFD) \cite{MFD14},
and the hybrid finite volumes discontinuous hydraulic head method
(UNICE\_UNIGE-HFV\_Disc) developed by Brenner et al. \cite{HFV}.
Among these 7 schemes, the first three schemes use a mixed dimensional interface model that require the modeling of co-dimension 1-3 fractured flows, where the mesh can be non-matching across subdomains, but needs to be geometrically conforming to the fractures.
On the other hand, the last four schemes work on a mixed dimensional interface model where only fractured flow in co-dimension 1 were modeled, which require the mesh to be completely conforming to the fractures.
All of these schemes yield a locally conservative velocity approximation.
We further note that the two methods in \cite{Berre_2021} that allow for general nonconforming meshes, namely the Lagrange multiplier method \cite{LMFEM2D, LMFEM2D2, LMFEM3D} and the EDFM method \cite{EDFM20}, cannot handle blocking fractures and do not provide a locally conservative velocity approximation.

Numerical results of our proposed scheme \eqref{fem} for the benchmark problems in \cite{Berre_2021} indicate that our results yield no significant deviations with the above mentioned 7 schemes, see details in the next section.  Our scheme also produce a locally conservative velocity approximation, and the resulting linear system after static condensation is a symmetric positive definite (SPD) problem with global unknowns involve pressure DOFs on the mesh skeleton only. The number of the global unknowns of our scheme is roughly $N_F$, which is the total number of mesh faces, and the average nonzero entries per row in the system matrix is 7 (a pressure DOF on an interior tetrahedral face is connected to  6 neighboring face pressure DOFs).
Concerning the computational cost of our scheme, it is more expensive than
the TPFA scheme (USTUTT-TPFA\_Circ) which lead to an SPD system with roughly $N_C$ cell-wise pressure DOFs and about 5 nonzero entries per row in the system matrix, is slightly less expensive than
the cell-based MPFA schemes (UiB-MPFA, USTUTT-MPFA), which lead to SPD systems with roughly $N_C$ cell-wise pressure DOFs
and about 20-50 nonzero entries per row in the system matrix, and is
significantly cheaper than the schemes UiB-MVEM, UiB-RT0, LANL-MFD, and UNICE\_UNIGE-HFV\_Disc, which lead to saddle point systems with total number of roughly $N_F$ velocity DOFs and $N_C$ pressure DOFs. Note that $N_F\approx 2 N_C$. Hence, our proposed scheme is also highly competitive in terms of computational costs.
Another distinctive advantage of our scheme over these 7 schemes is that
the mesh can be completely nonconforming to the blocking fractures.

\section{Numerics}
\label{sec:num}
In this section,
we present
detailed numerical results for the proposed hybrid-mixed method
for the four 2D benchmark test cases in \cite{FLEMISCH2018239}
and the four 3D benchmark test cases in
\cite{Berre_2021}.
We name the method \eqref{fem} as {\sf HM-DFM} since it is a hybrid mixed method for
a discrete fracture model.
When plotting the pressure or hydraulic head disctribution over line segments, we evaluate
the second-order postprocessed solution in \eqref{postprocess} for the proposed
method.
The focus of the numerical experiments is on the verification of the accuracy of our proposed
flow model \eqref{model} and the associated method \eqref{fem}.
Hence, we test the flow solver \eqref{fem} for all the 8 benchmark cases.
Meanwhile, we also test the accuracy of velocity approximation by feeding them to the transport problem \eqref{transport}, which is solved using the scheme \eqref{transport-eq} for three cases, namely Benchmark 2 in 2D, and Benchmark 5/6 in 3D. Furthermore, convergence study via mesh refinements was conducted for  Benchmark 2 and Benchmark 6 below.

Our numerical simulations are performed using the open-source finite-element software
{\sf NGSolve} \cite{Schoberl16}, \url{https://ngsolve.org/}.
Jupyter notebooks for reproducing all numerical examples in this section
can be found in the git repository
\url{https://github.com/gridfunction/fracturedPorousMedia}.
Visualization of meshes for the 3D benchmark examples and interactive
contour plots of the pressure/hydraulic head can also be found therein.


\subsection{Benchmark 1: Hydrocoin (2D)}
This example is originally a benchmark for heterogeneous groundwater flow presented in the
international Hydrocoin project \cite{Swedish}.
A slight modification for the geometry was made in
\cite[Section 4.1]{FLEMISCH2018239}, and we follow the settings therein.
In particular,
the bulk domain is a polygon with vertices
$A=(0, 150), B=(400, 100), C=(800, 150), D=(1200,100), E=(1600, 150), F=(1600, -1000),
G=(1500, -1000), H=(1000, -1000)$ and $I=(0, -1000)$ measured in meters.
There are two conductive fractures in the domain $\{BG\}$ and $\{DH\}$.
The fracture $\{BG\}$ has thickness $\epsilon=5\sqrt{2} m$ and
the fracture $\{DH\}$ has thickness $\epsilon=33/\sqrt{5} m$.
The permeability (hydraulic conductivity) is $\mathbb
K_m=10^{-8}m/s$ in the bulk and $\mathbb K_c=10^{-6}m/s$ in the fractures.
Dirichlet boundary condition $p=\text{height}$ is imposed on the top boundary,
and homogeneous Neumann boundary condition is imposed on the rest of the
boundary.
Here the unknown variable $p$ is termed as the piezometric head according to \cite{Swedish}.
The quantity of interest is the distribution of the piezometric head $p$
along the horizontal line at a depth of $200 m$.

We apply the method \eqref{fem} on a uniform triangular mesh with mesh size
$h=60$, see the left panel of Figure \ref{fig:hydro}, which leads to $1,115$ matrix elements and $44$ fracture elements.
On this mesh, the number of the globally coupled DOFs is $1,779$, in which
$1,691$ DOFs are associated with the bulk hybrid variable $\widehat p_h$,
and
$43$ DOFs are associated with the fracture hybrid variable $\widehat p_h^c$.
In the right panel of Figure \ref{fig:hydro},
we record the postprocessed piezometric head $p_h^*$ in \eqref{postprocess}
along the line segment  $z=-200 m$, where $z$ is the horizontal direction, along
with the reference data obtained from a mimetic finite difference method on
a very fine mesh (with $889,233$ DOFs). We observe that the results for the  proposed method on such a coarse mesh
already shows a good agreement
with the reference data.

\begin{figure}[ht]
\centering
\includegraphics[width=0.55\textwidth]{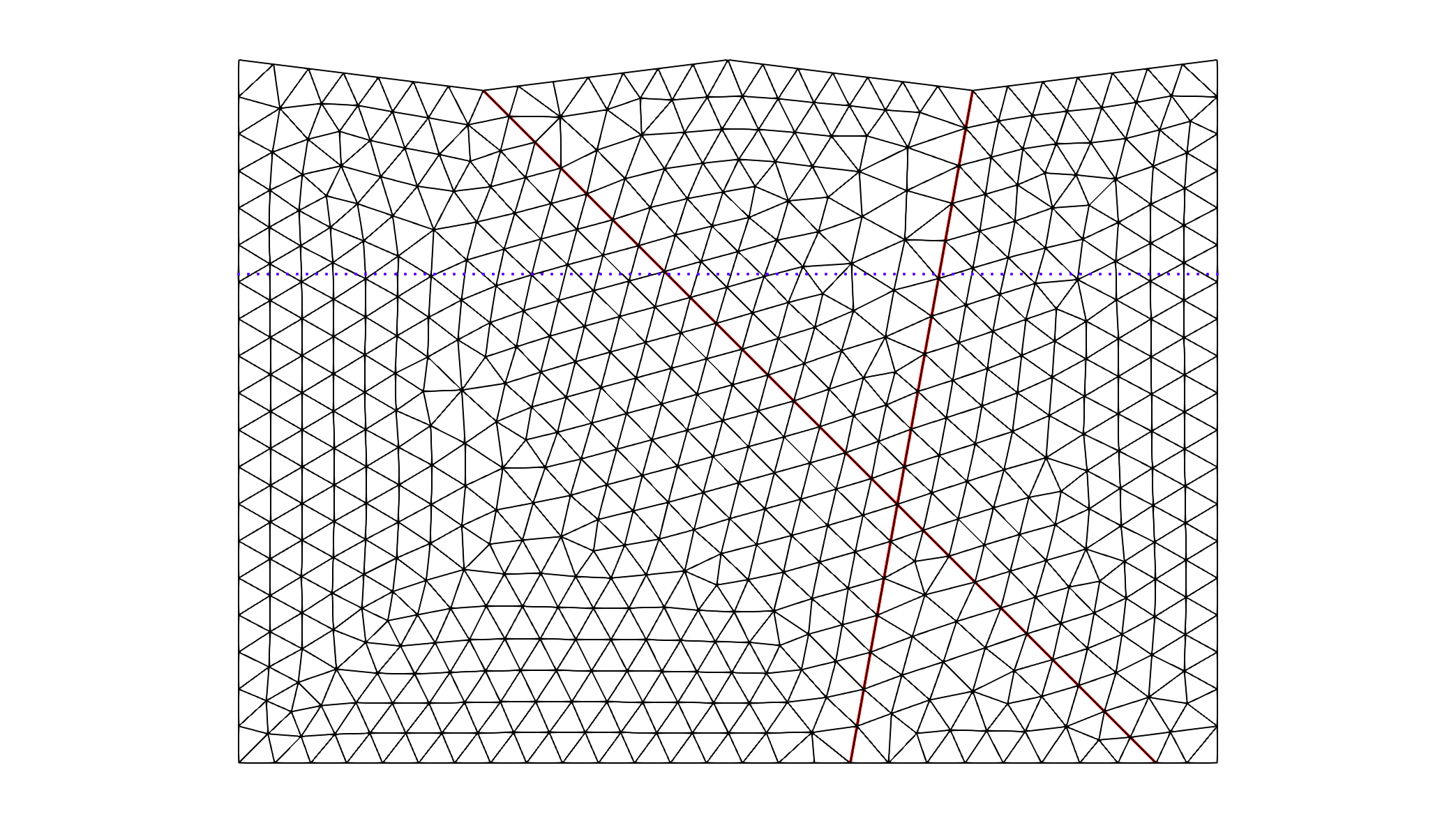}
\includegraphics[width=0.43\textwidth]{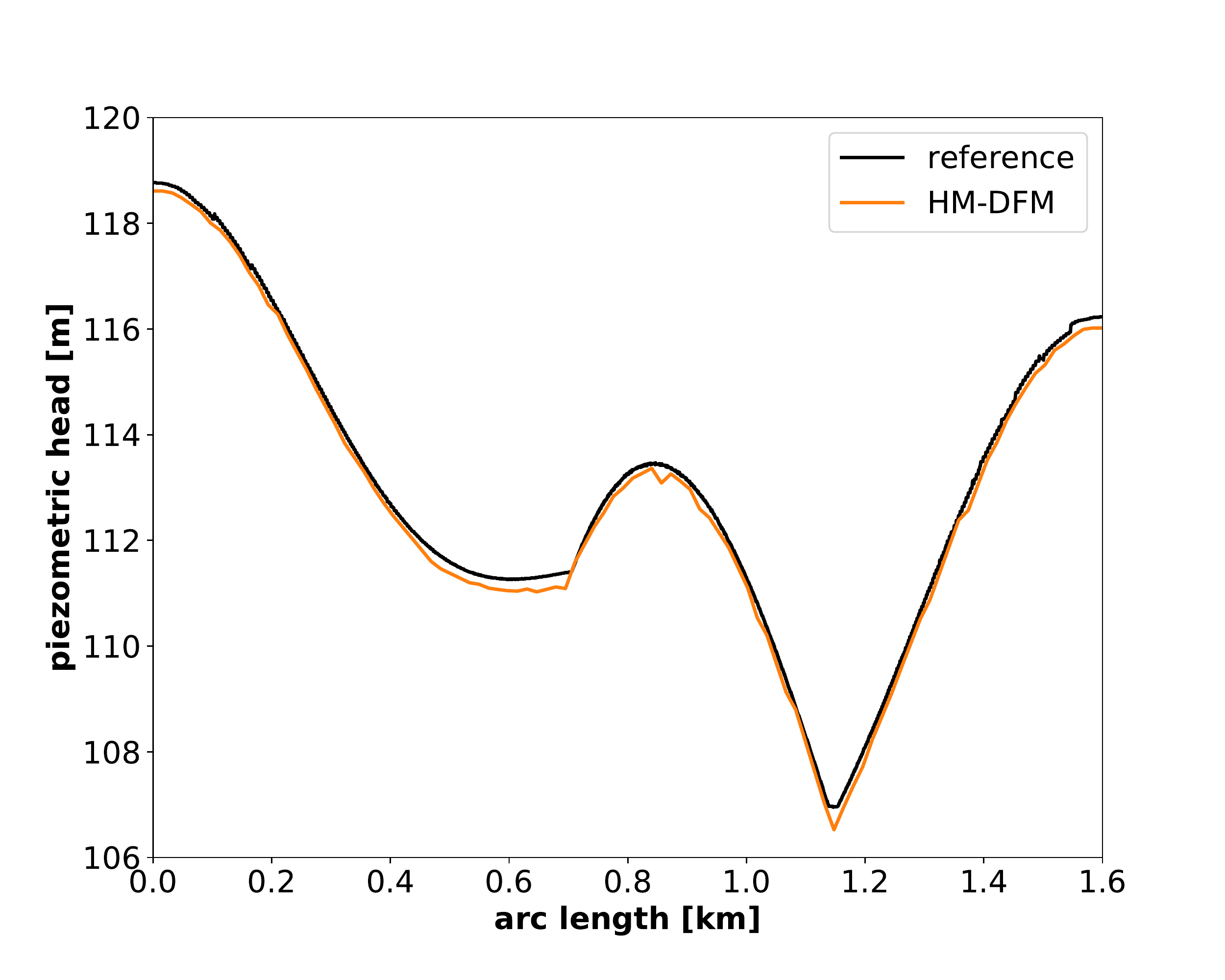}
\caption{Benchmark 1. Left: computational mesh. Right:
  piezometric head along the line $z=-200 m$ (dotted blue line on the left
  figure).
}
\label{fig:hydro}
\end{figure}

\subsection{Benchmark 2: Regular Fracture Network (2D)}
This test case is originally from \cite{geiger} and is modified by \cite{FLEMISCH2018239},
which simulates a regular fracture network in a square porous media.
The computational domain including the fracture network and boundary conditions
is shown in Figure \ref{fig:geiger}.
\begin{figure}[ht]
  \centering
  \includegraphics[width=.6\textwidth]{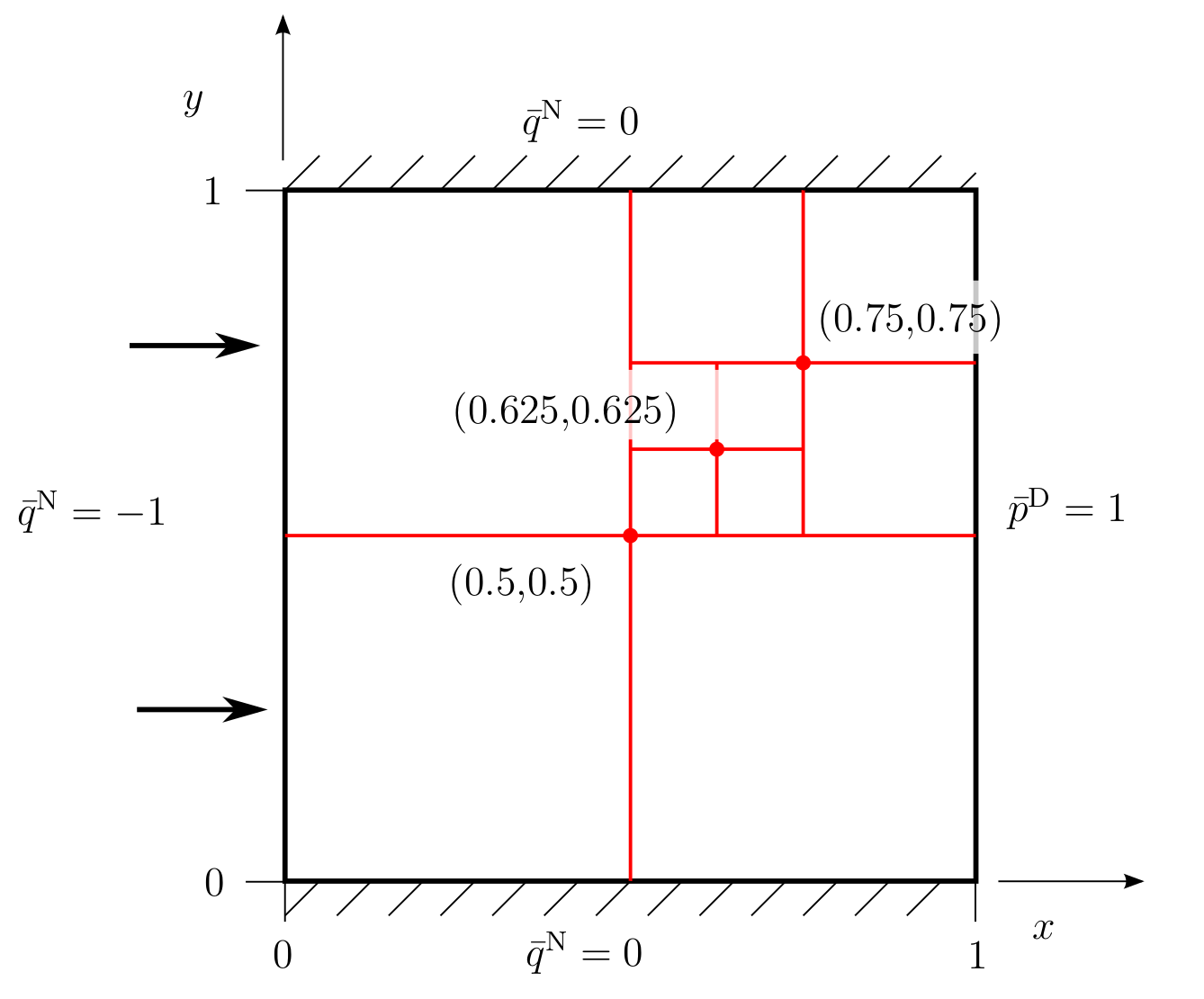}
  \caption{Benchmark 2. Domain and boundary conditions.}
  \label{fig:geiger}
\end{figure}
The matrix permeability is set to $\mathbb K_m=\mathbb I$, and fracture thickness is
$\epsilon = 10^{-4}$. Two cases of fracture permeability was considered:
(i) a highly conductive network with $\mathbb K_c = 10^4\mathbb I$, (ii) a blocking
fracture with $K_b=10^{-4}$.

We apply the method \eqref{fem} on a
triangular mesh with $1,348$ matrix elements and
$91$ fracture elements, see the left panel of Figure \ref{fig:geiger0}.
For the blocking fracture case, we also present the result on a unfitted
triangular mesh with $1,442$ matrix elements.

For the conductive fracture case, the number of the globally coupled DOFs is $2,127$, in which
$2,041$ DOFs are associated with the bulk hybrid variable $\widehat p_h$,
and
$86$ DOFs are associated with the fracture hybrid variable $\widehat p_h^c$.
The pressure distributions  along two lines, one horizontal at $y = 0.7$ and
one vertical at $x = 0.5$ are shown in Figure \ref{fig:geiger1}, along with
the reference data obtained from a mimetic finite difference method on
a very fine mesh (with $1,175,056$ DOFs).
Similar to the previous example,
we observe that the results for the proposed method  show a good agreement with the reference data.

For the blocking fracture case, the number of the globally coupled DOFs is $2,041$ on the fitted mesh and
is $2,188$ on the unfitted mesh.
The pressure distribution  along the lines $(0,0.1)$--$(0.9,1.0)$
is  shown in Figure \ref{fig:geiger2}.
Again, we observe a very good agreement with reference data for the results on
the fitted mesh. The result on the unfitted mesh case is slightly off due to mesh nonconformity,
which is expected as it could not capture the pressure discontinuity across the
blocking fractures.

\begin{figure}[ht]
  \centering
  \begin{tabular}{cc}
    \includegraphics[width=0.48\textwidth]{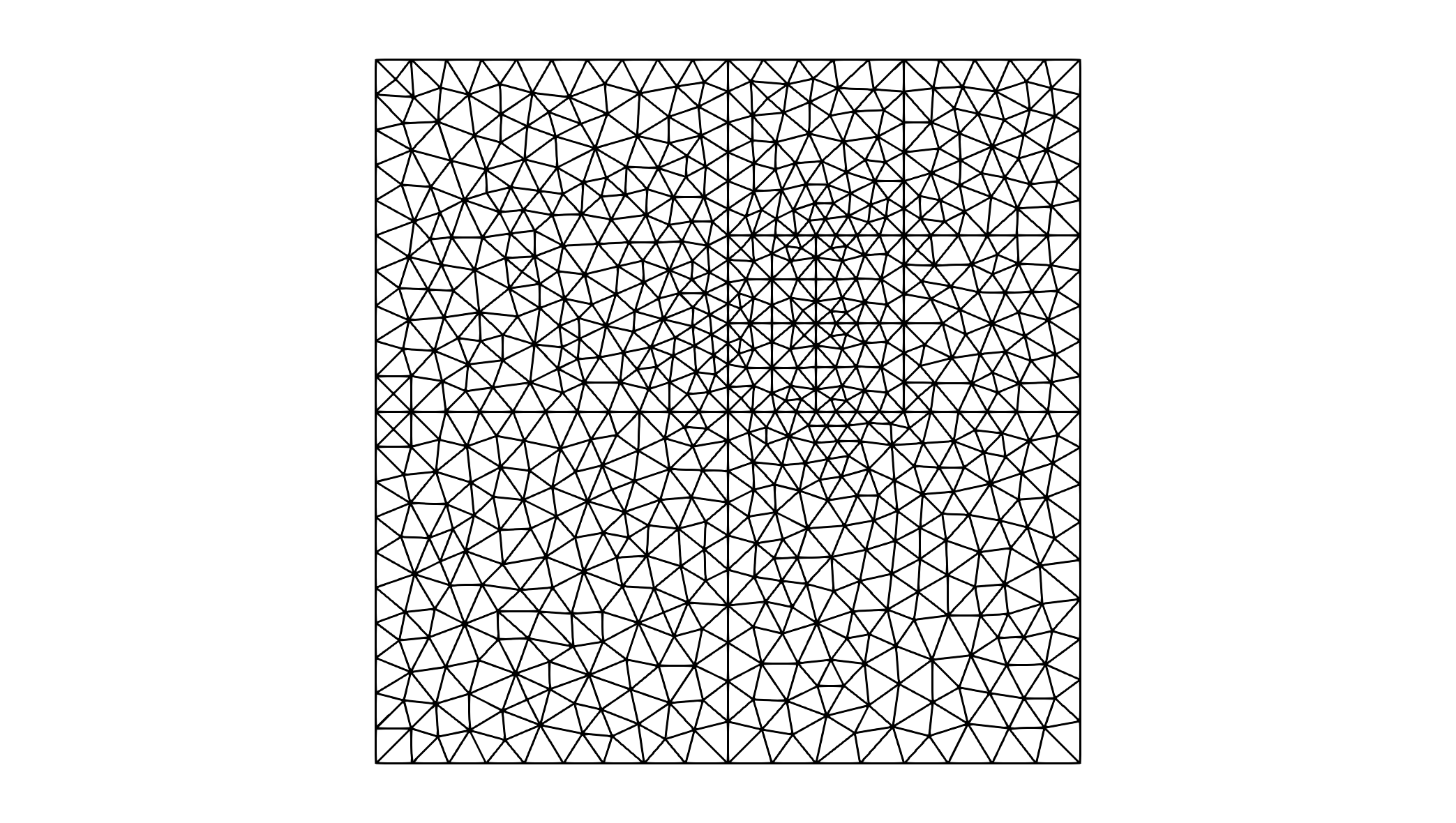}&
    \includegraphics[width=.48\textwidth]{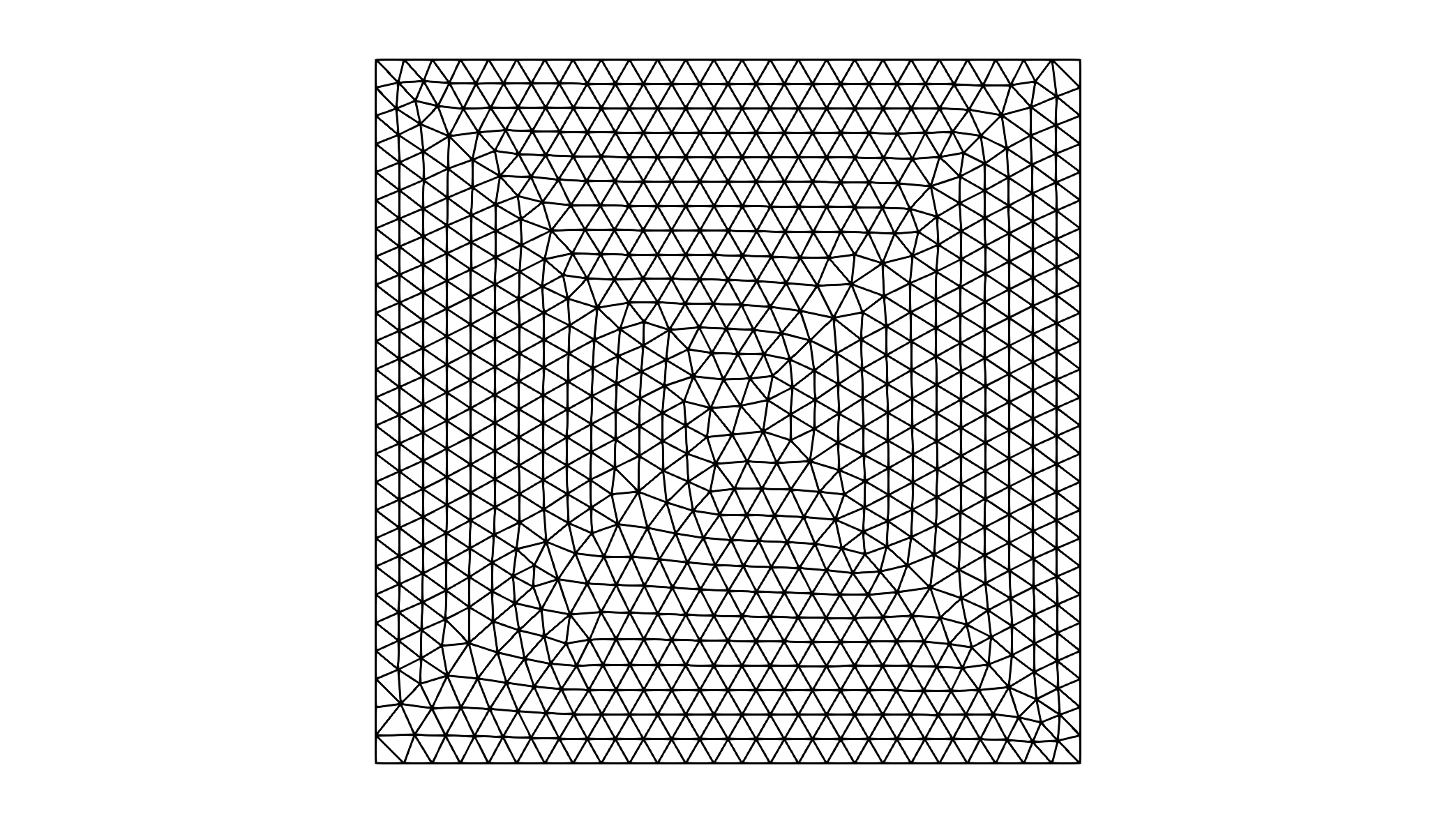}\\
    (a) a fitted mesh. &
   (b) a unfitted mesh.
  \end{tabular}
  \caption{Benchmark 2: computational meshes.
    The fitted mesh on the left panel is used for both conductive and blocking
    fracture cases.
    The unfitted mesh on the right panel is used only for the
    blocking fracture case.
}
  \label{fig:geiger0}
\end{figure}

\begin{figure}[ht]
  \centering
  \begin{tabular}{cc}
    \includegraphics[width=.43\textwidth]{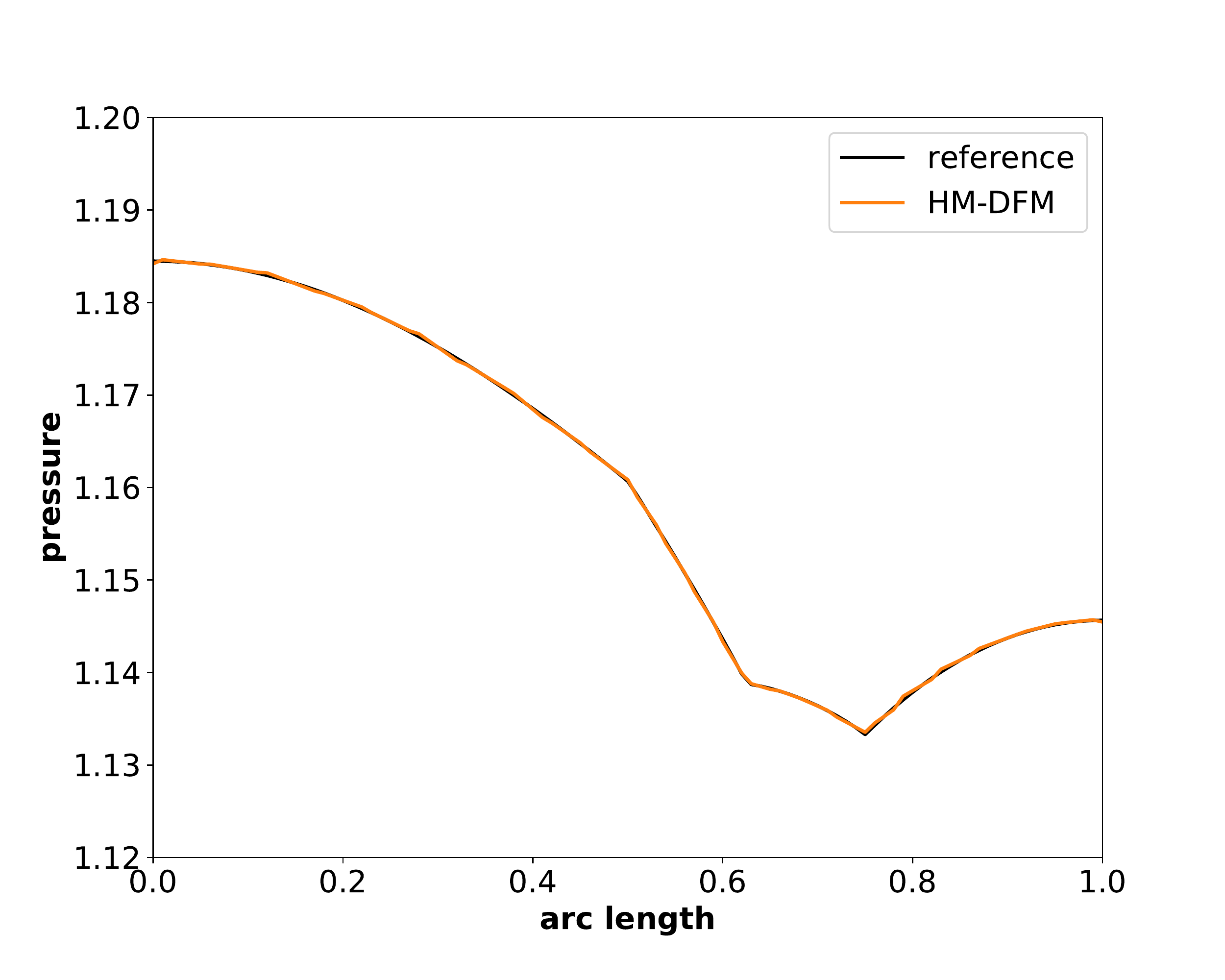}&
    \includegraphics[width=.43\textwidth]{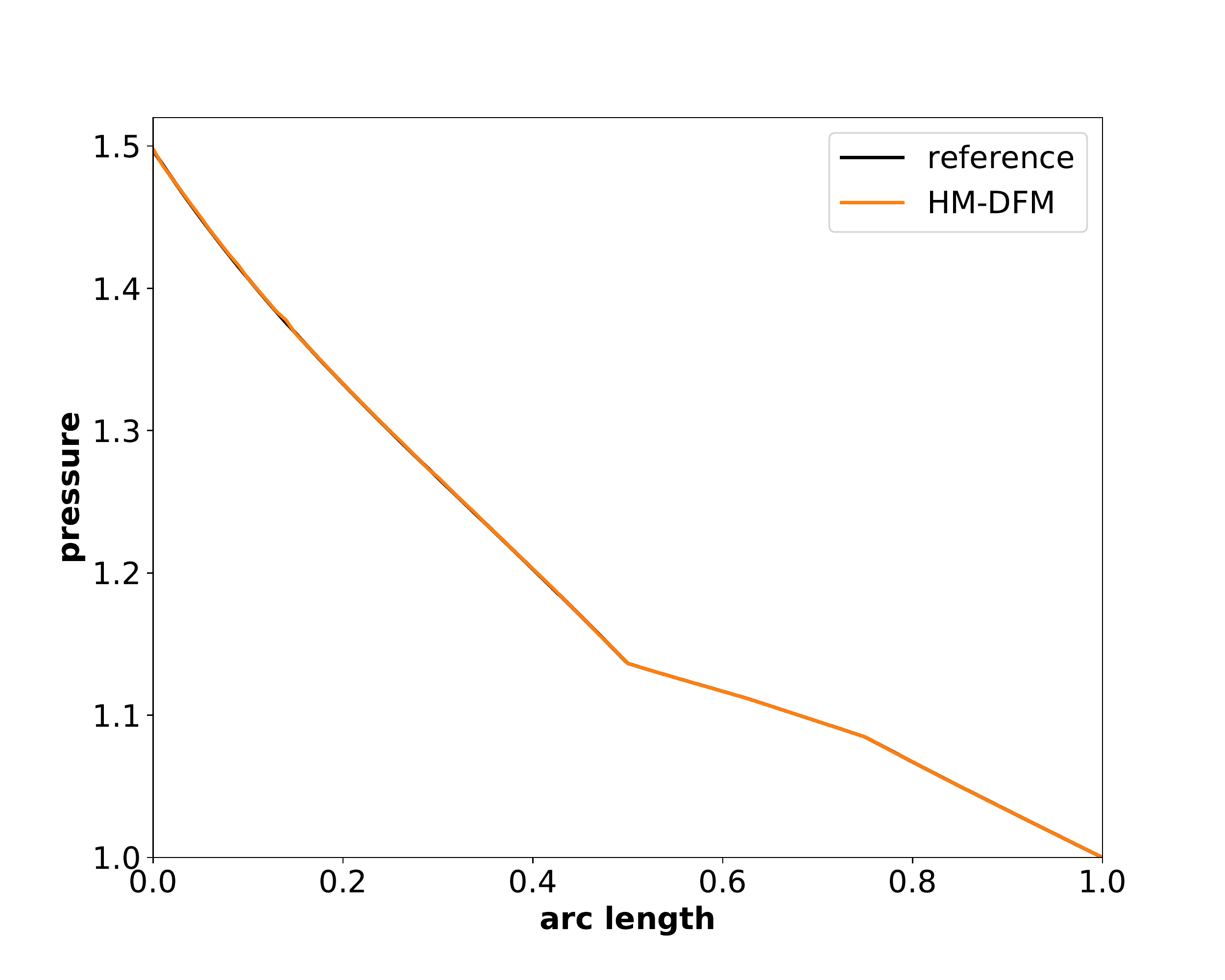}\\
    (a) Horizontal line at $y=0.7$. &
   (b) Vertical line at $x=0.5$.
  \end{tabular}
  \caption{Benchmark 2 with conductive fractures: pressure distribution along two
  lines. }
  \label{fig:geiger1}
\end{figure}

\begin{figure}[ht]
  \centering
    \includegraphics[width=.63\textwidth]{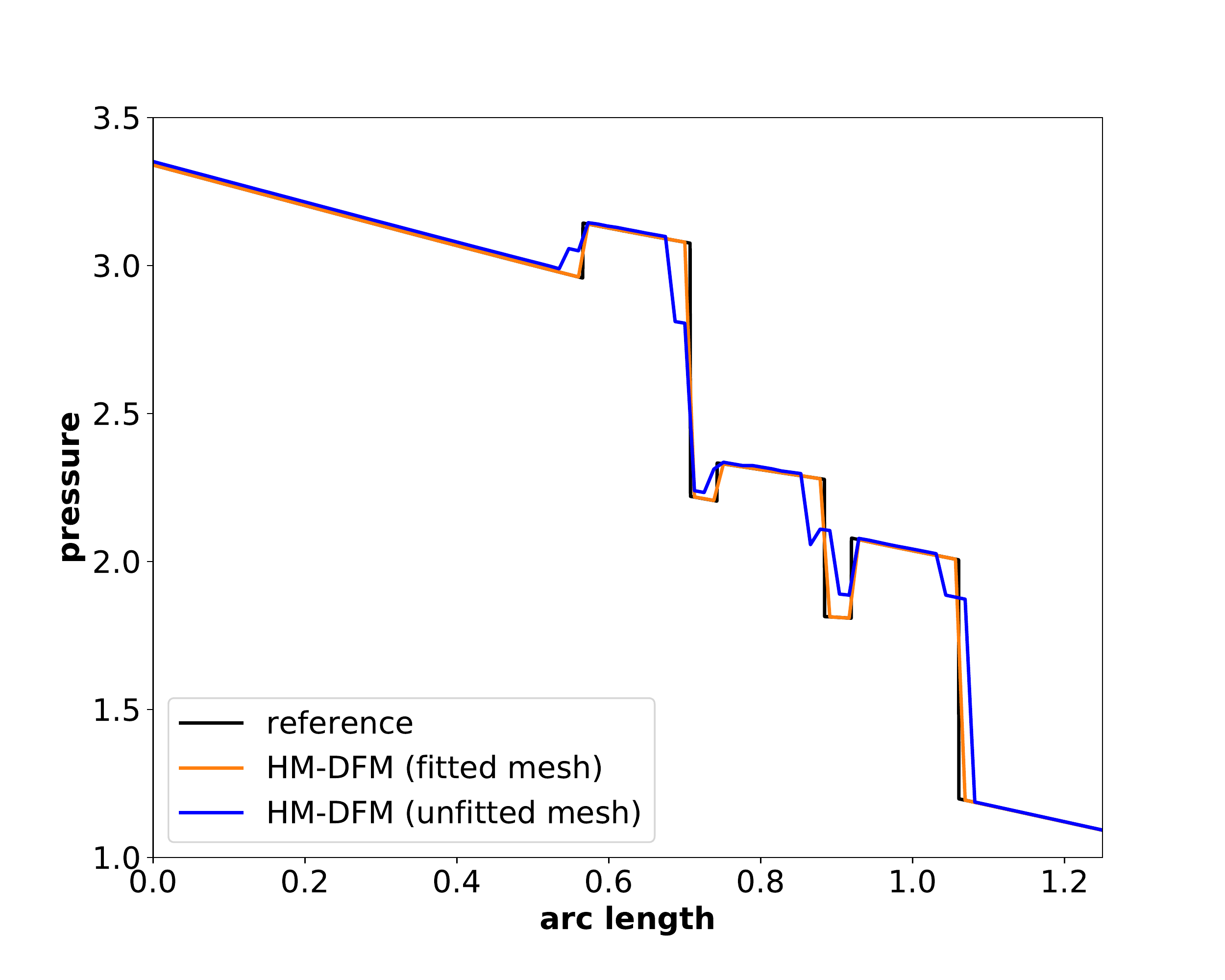}
  \caption{Benchmark 2 with blocking fractures:
    values along the line (0,0.1)--(0.9,1.0). }
  \label{fig:geiger2}
\end{figure}

\subsubsection{Coupling with transport and convergence study with mesh refinements}
After the velocity fields are computed from the scheme \eqref{fem},
we feed them to the transport model \eqref{transport}, and solve it by using the hybrid finite volume scheme \eqref{transport-eq}.
We take the porosities $\phi_m=0.1$, $\phi_c=0.9$ in the model \eqref{transport}, with the initial concentrations $c_0=c_{c,0}=0$,
and set the left boundary as the
inflow boundary for the concentrations, with $c_B=c_{c,B}=1$.
The final time of simulation is $T=0.1$.
Convergence of our coupled scheme \eqref{fem} and \eqref{transport-eq} is checked via a mesh refinement study, where the initial meshes are given in Figure~\ref{fig:geiger0}, and three level of uniform mesh refinements are applied afterwards.
The constant time step size is taken to be $\Delta t = 2^{-l}\times 5\times 10^{-3}$, where $l$ is the mesh refinement level.
Since there is no analytic solution to the problem, we provide a reference solution using the coupled scheme \eqref{fem} and \eqref{transport-eq}
on the fourth level refined fitted mesh (with about  345k elements) with a small time step size $\Delta t = 3.125\times 10^{-5}$.
Contour of matrix concentrations of the reference solution at time $t=0.05$ and $t=0.1$ are presented in Figure~\ref{fig:geigerX}, where we clearly observe the conducting and blocking effects of the respective fractures.
\begin{figure}[ht]
  \centering
  \subfigure[Conductive fractures,  $t=0.05$]{\includegraphics[width = 3.0in]{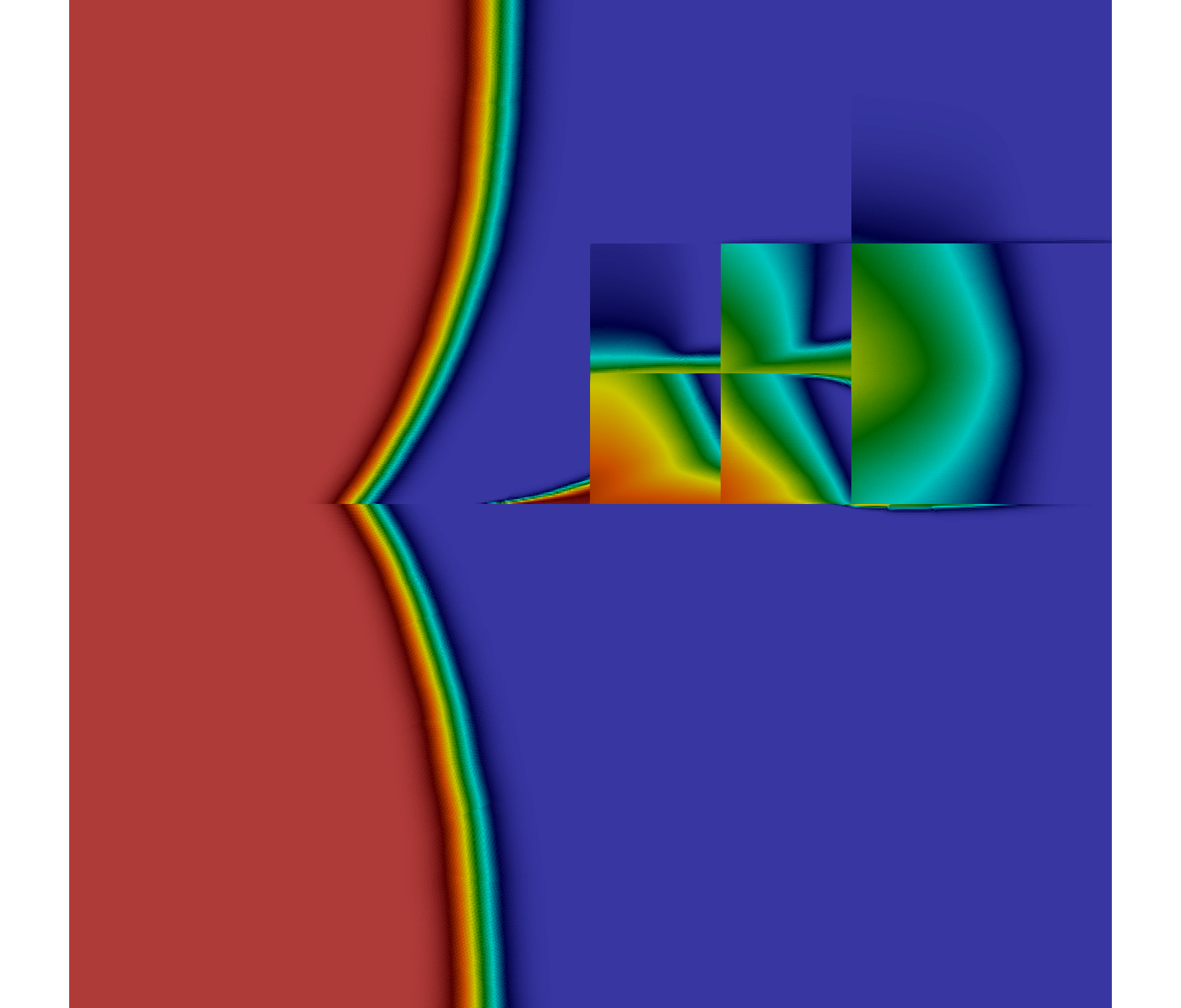}}
    \subfigure[Conductive fractures,  $t=0.1$]{\includegraphics[width = 3.0in]{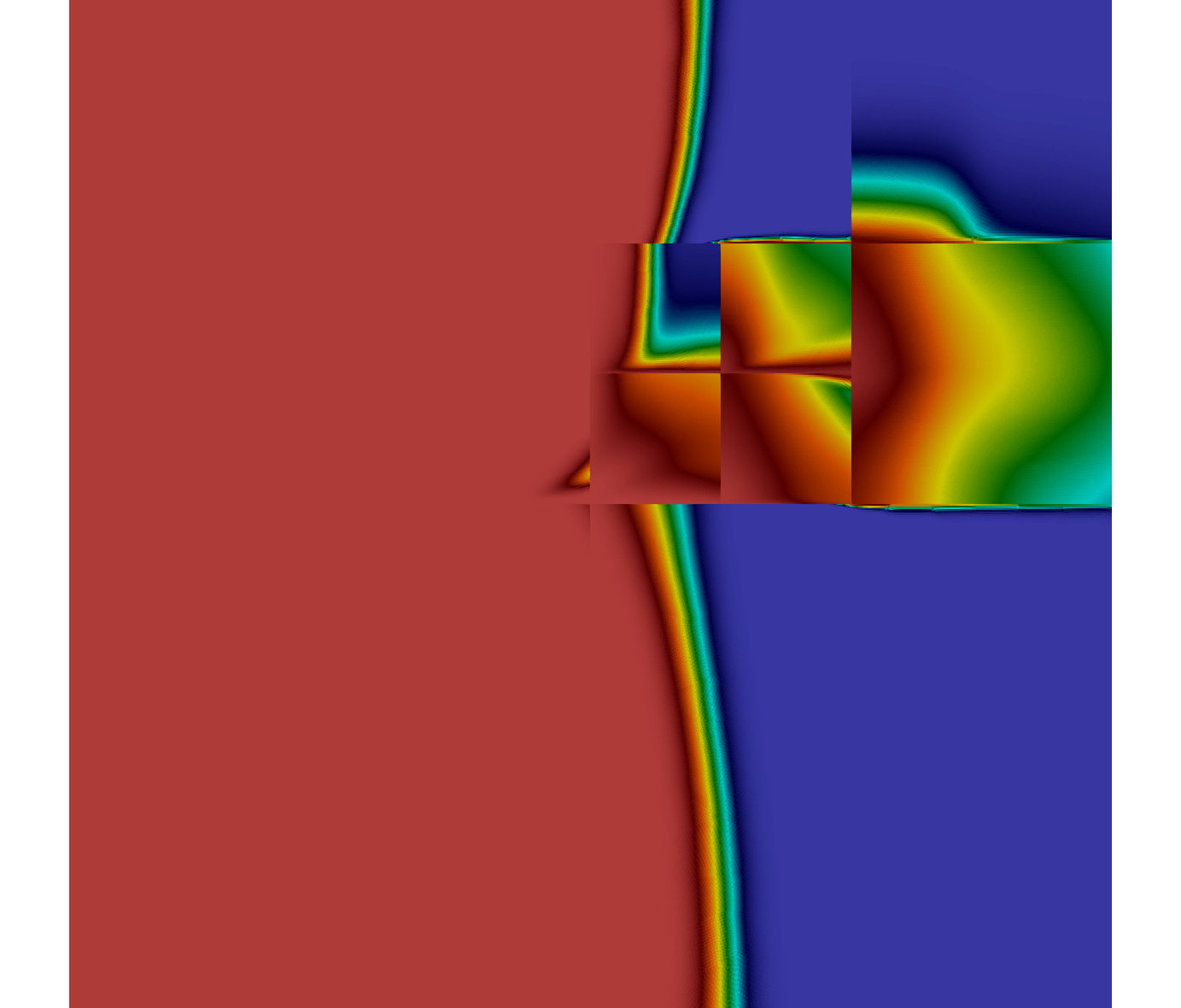}}
    \subfigure[Blocking fractures,  $t=0.05$]{\includegraphics[width = 3.0in]{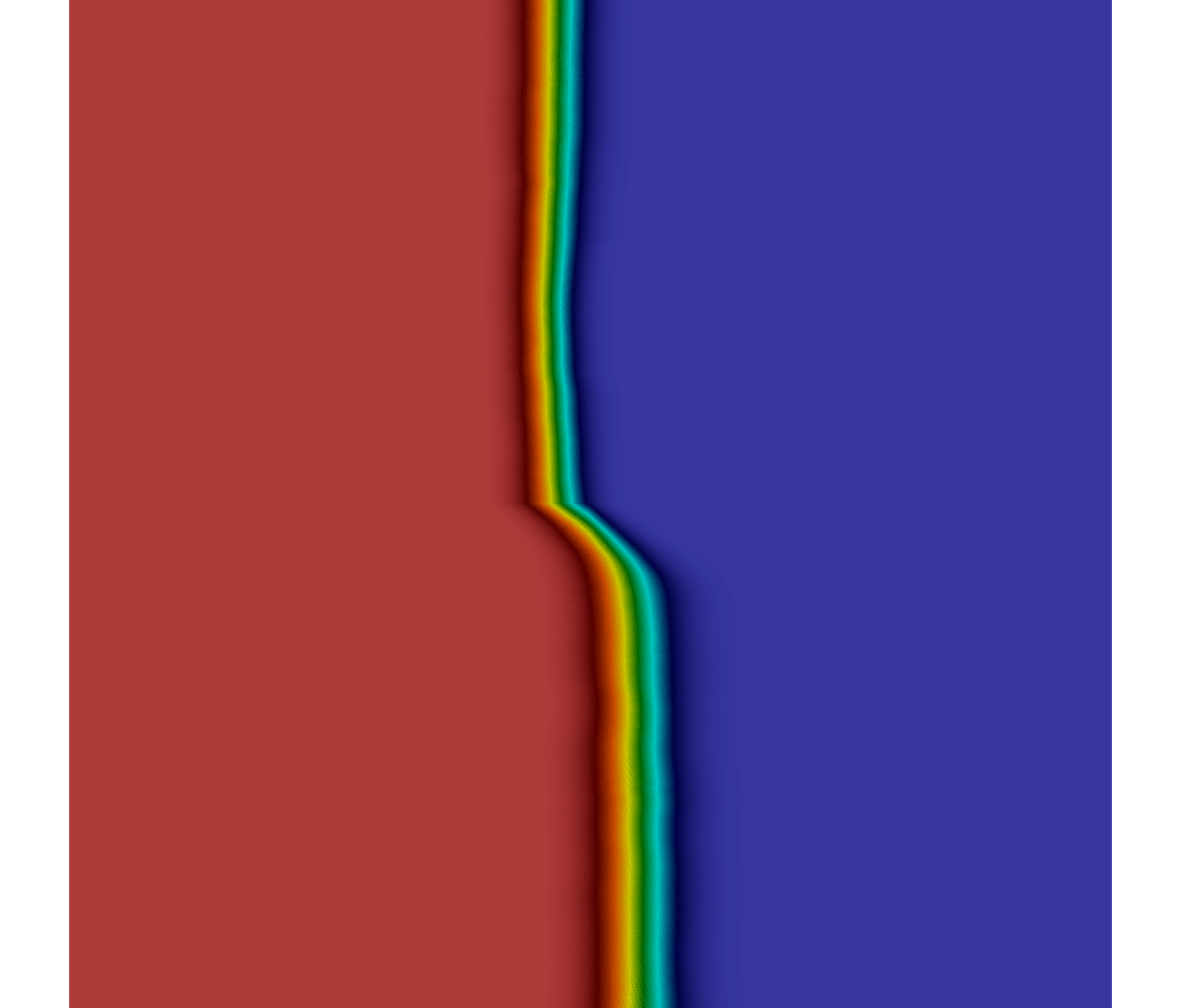}}
    \subfigure[Blocking fractures,  $t=0.1$]{\includegraphics[width = 3.0in]{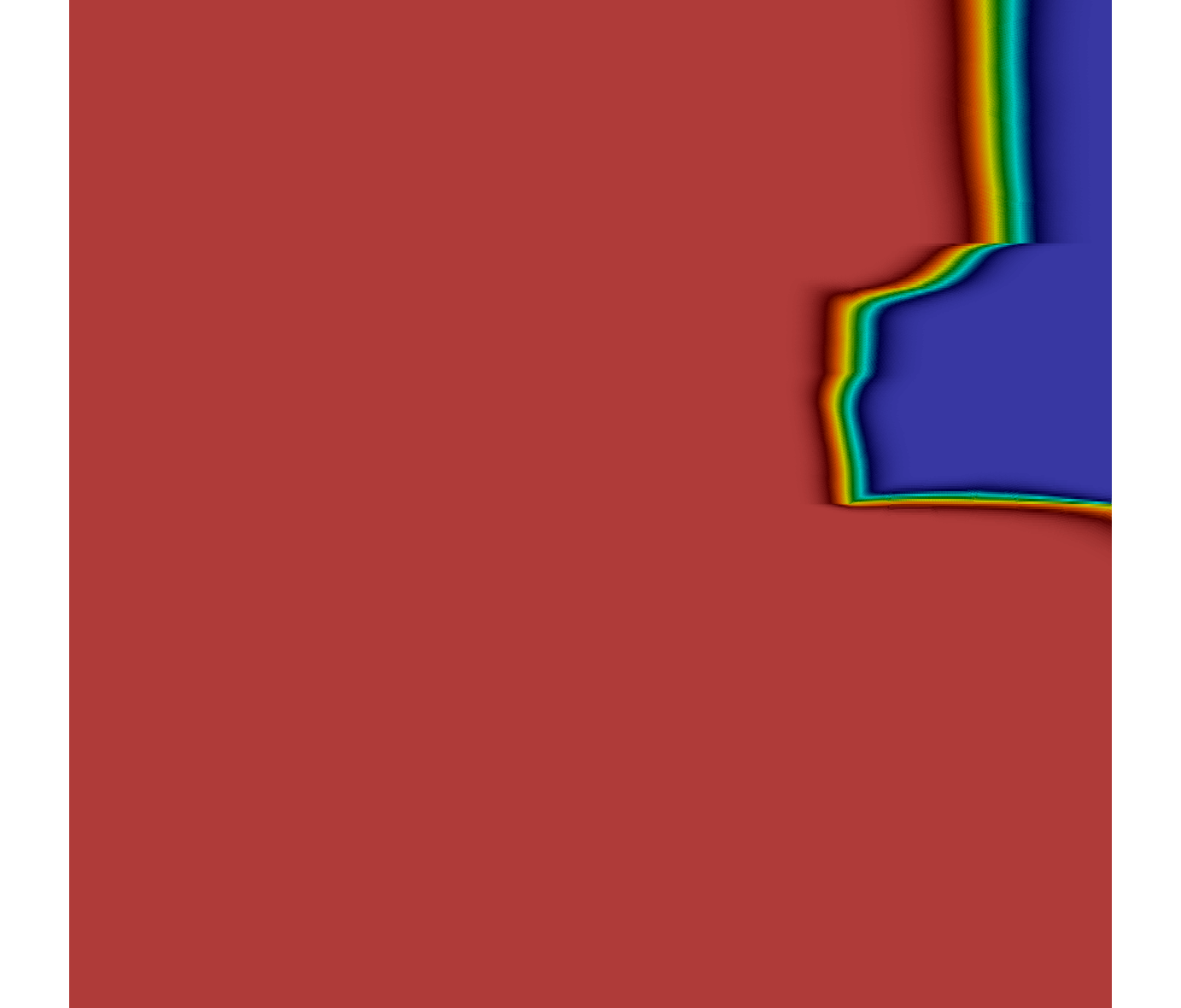}}
  \caption{Benchmark 2: Matrix concentration at time $t=0.05$ (left) and
  $t=0.1$ (right). Top row: conductive fractures. Bottom row: blocking fractures. Color range: 0(blue)--1(red).
  Solution obtained on the fourth level refined mesh with
a  small time step size $\Delta t = 3.125\times 10^{-5}$.
  }
  \label{fig:geigerX}
\end{figure}
Moreover, we plot the computed matrix concentrations along the cut line $y=0.7$ in Figure~\ref{fig:geigerY}, where we observe convergence as the mesh refines.
\begin{figure}[ht]
  \centering
  \subfigure[Conductive fractures,  fitted mesh]{\includegraphics[width = 3.0in]{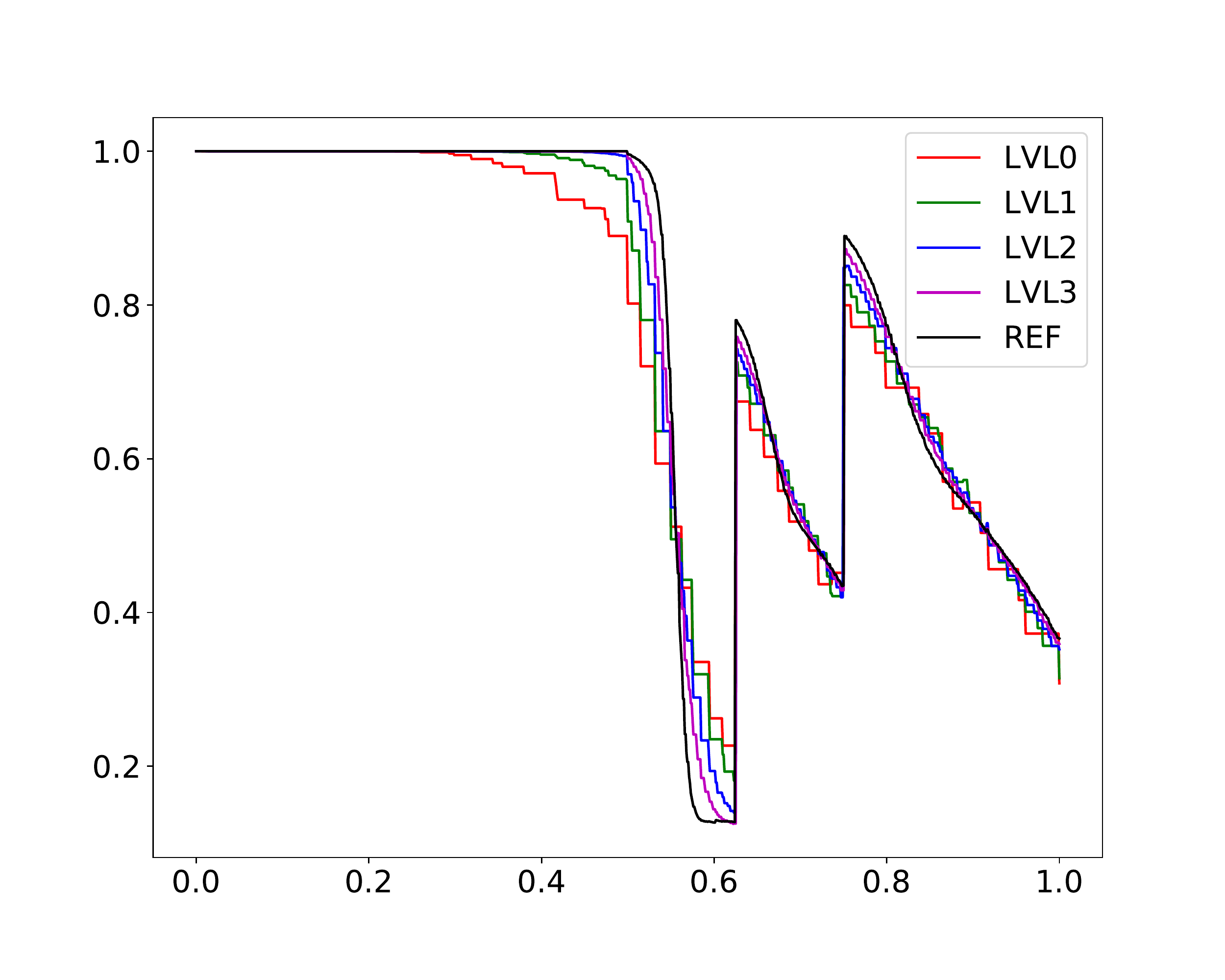}}\\
    \subfigure[Blocking fractures,  fitted mesh]{\includegraphics[width = 3.0in]{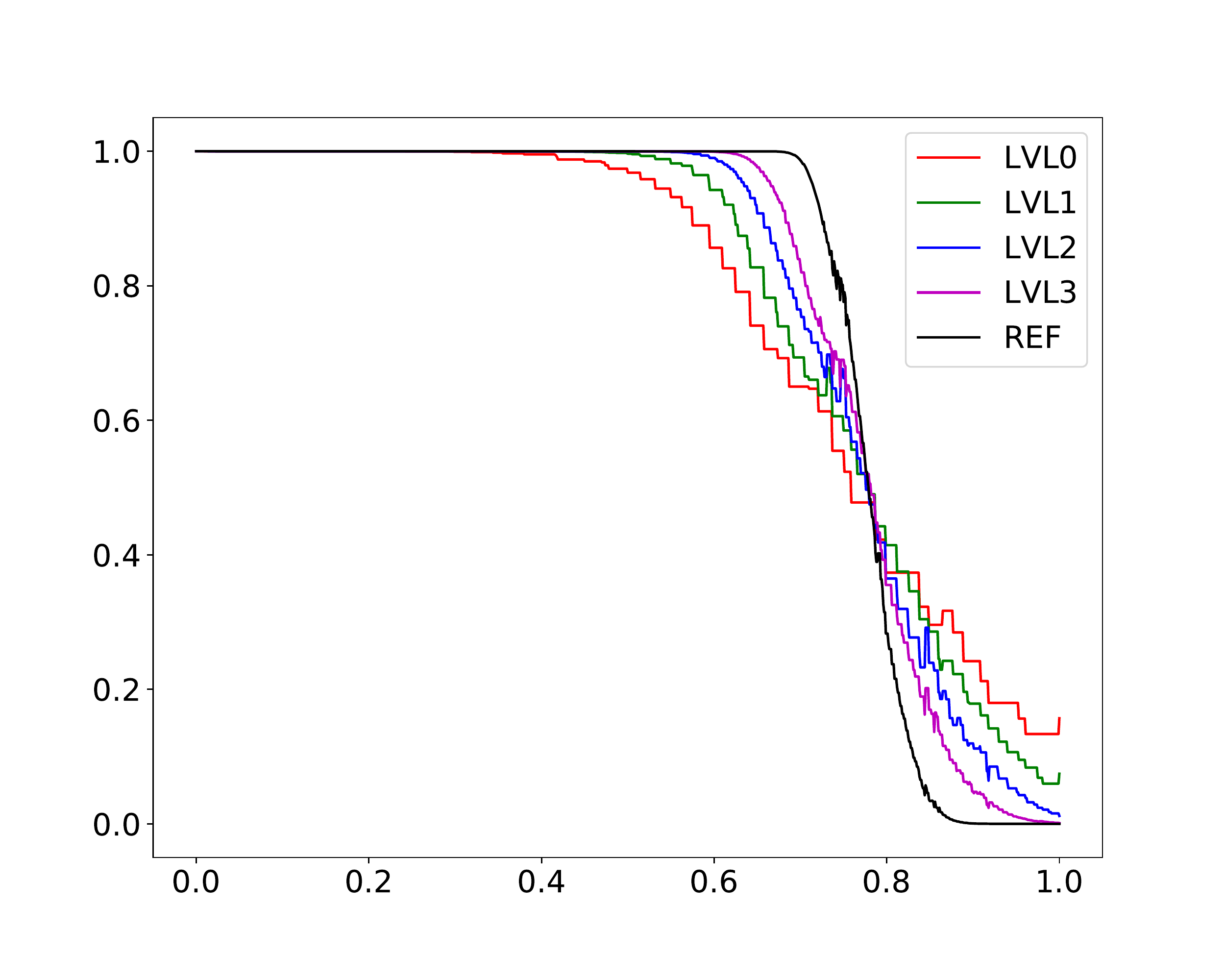}}
    \subfigure[Blocking fractures,  unfitted mesh]{\includegraphics[width = 3.0in]{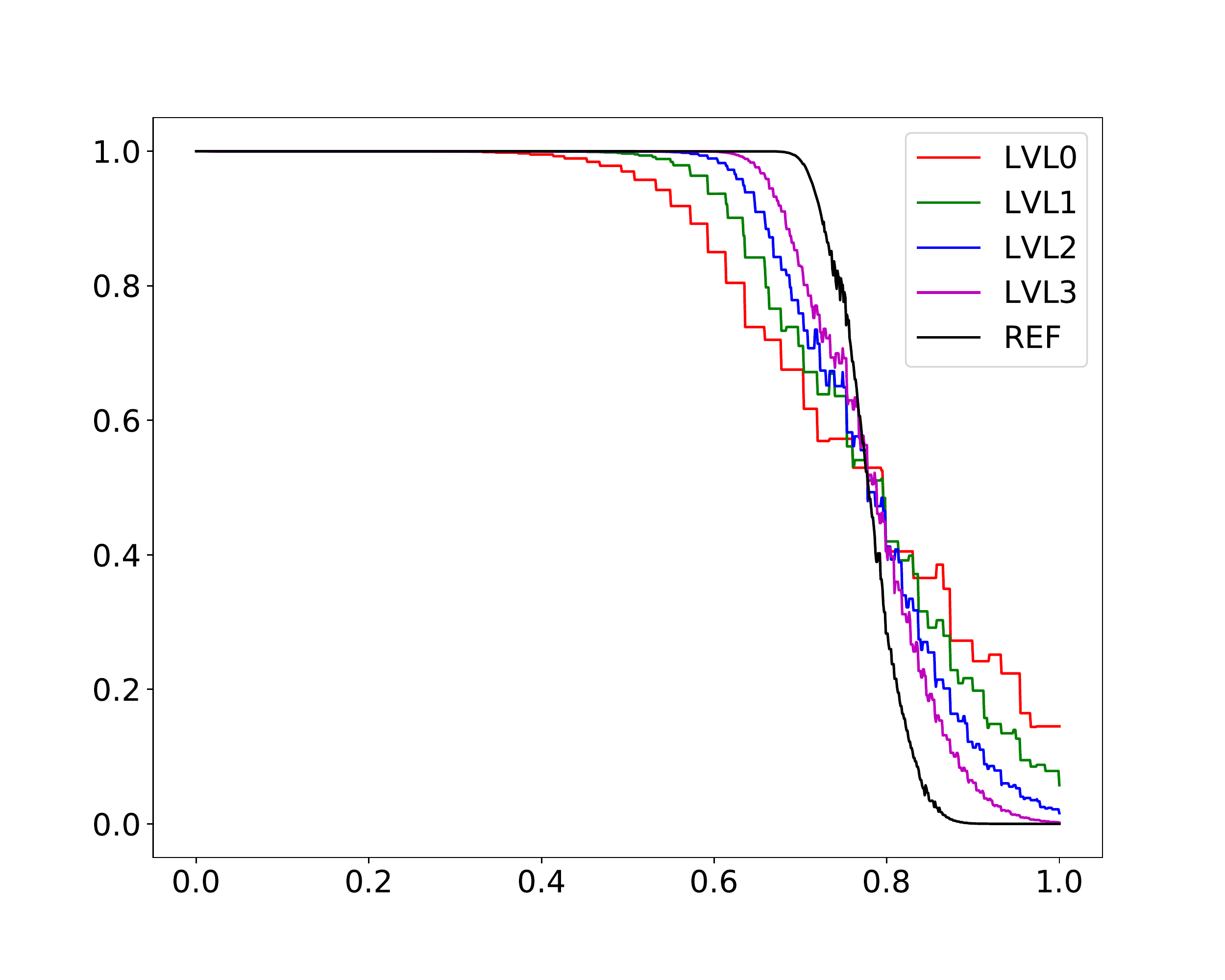}}
  \caption{Benchmark 2: Matrix concentration along the line  $y=0.7$ at time $t=0.1$ for the solution on different meshes. LVL stands for the number of  mesh refinement levels.
  Reference solution is obtained on
    the fourth level refined fitted mesh with a small time step size
    $\Delta t = 3.125\times 10^{-5}$.
  }
  \label{fig:geigerY}
\end{figure}

Finally, the $L^2$-errors in the matrix velocity and postprocessed matrix pressure, and the $L^2$-errors in  the matrix concentration at final time $T=0.1$
are recorded in Table~\ref{tab:geiger1} for the conductive fracture case, in Table~\ref{tab:geiger2} for the blocking fracture case on fitted meshes and
in Table~\ref{tab:geiger3} for the blocking fracture case on unfitted meshes.
From Table~\ref{tab:geiger1} for the conductive fracture case, we observe that the convergence rate in the velocity approximation is first order and that in the postprocessed pressure approximation is second order, which is consistent with the expected convergence behavior of the hybrid-mixed method for the equi-dimensional case \cite{RT77,AB85}, and the convergence rate for the
concentration is about $1/2$, which is also expected
for the hybridized finite volume scheme
due to the concentration discontinuities in the domain.
Similar convergence behavior was observed in Table~\ref{tab:geiger2}
for the blocking fracture case on fitted meshes.
From Table~\ref{tab:geiger3} we observe $1/2$ order convergence for all three variables, where the degraded velocity and pressure convergence is due to nonconformity of the mesh with the fractures.
\begin{table}[ht!]
    \centering
    \begin{tabular}{c|cc|cc|cc}
    mesh ref. lvl.     & $L^2$-err in $\bld u_h$ & rate
    & $L^2$-err in $p_h^*$ & rate
&     $L^2$-err in $c_h(T)$ & rate\\
\hline
    0     & 3.567e-02 &--& 3.786e-04 &--&1.177e-01 &--\\
1&1.954e-02 & 0.87&1.061e-04 & 1.84&8.587e-02 & 0.45\\
2&1.029e-02 & 0.92&7.146e-06 & 2.00&5.883e-02 & 0.55\\
3&4.881e-03 & 1.08&2.863e-05 & 1.89&3.541e-02 & 0.73\\
\hline
    \end{tabular}
    \vspace{1ex}
    \caption{Benchmark 2 with conductive fractures (fitted mesh): history of convergence for the $L^2$-errors in $\bld u_h$, $p_h^*$, and
    $c_h(T)$ along mesh refinements. Reference solution is obtained on
    the fourth level refined fitted mesh with a small time step size
    $\Delta t = 3.125\times 10^{-5}$.}
    \label{tab:geiger1}
\end{table}

\begin{table}[ht!]
    \centering
    \begin{tabular}{c|cc|cc|cc}
    mesh ref. lvl.     & $L^2$-err in $\bld u_h$ & rate
    & $L^2$-err in $p_h^*$ & rate
&     $L^2$-err in $c_h(T)$ & rate\\
\hline
0 & 1.358e-02 & -- & 2.406e-04 & -- &1.396e-01 &-- \\
1 & 7.098e-03 & 0.94 & 6.402e-05 & 1.91 &1.025e-01 & 0.44 \\
2 & 3.607e-03 & 0.98 & 1.630e-05 & 1.97 &7.149e-02 & 0.52 \\
3 & 1.666e-03 & 1.11 & 3.566e-06 & 2.19 &4.572e-02 & 0.64 \\
\hline
    \end{tabular}
    \vspace{1ex}
    \caption{Benchmark 2 with blocking fractures (fitted mesh): history of convergence for the $L^2$-errors in $\bld u_h$, $p_h^*$, and
    $c_h(T)$ along mesh refinements. Reference solution is obtained on
    the fourth level refined fitted mesh with a small time step size
    $\Delta t = 3.125\times 10^{-5}$.}
    \label{tab:geiger2}
\end{table}

\begin{table}[ht!]
    \centering
    \begin{tabular}{c|cc|cc|cc}
    mesh ref. lvl.     & $L^2$-err in $\bld u_h$ & rate
    & $L^2$-err in $p_h^*$ & rate
&     $L^2$-err in $c_h(T)$ & rate\\
\hline
0 & 7.611e-02 & -- & 8.295e-02 & -- &1.424e-01 &-- \\
1 & 5.357e-02 & 0.51 & 5.890e-02 & 0.49 &1.050e-01 & 0.44 \\
2 & 3.991e-02 & 0.42 & 4.088e-02 & 0.53 &7.418e-02 & 0.50 \\
3 & 2.634e-02 & 0.60 & 2.899e-02 & 0.50 &4.882e-02 & 0.60 \\
\hline
    \end{tabular}
    \vspace{1ex}
    \caption{Benchmark 2 with blocking fractures (unfitted mesh): history of convergence for the $L^2$-errors in $\bld u_h$, $p_h^*$, and
    $c_h(T)$ along mesh refinements. Reference solution is obtained on
    the fourth level refined fitted mesh with a small time step size
    $\Delta t = 3.125\times 10^{-5}$.}
    \label{tab:geiger3}
\end{table}

\subsection{Benchmark 3: Complex Fracture Network (2D)}
This test case considers a small but complex fracture network that includes permeable
and blocking fractures. The domain and boundary conditions are shown in Figure
\ref{fig:complex}.
The exact coordinates for the fracture positions are provided in
\cite[Appendix C]{FLEMISCH2018239}.
The fracture network contains ten straight immersed fractures. The fracture
thickness
is $\epsilon=10^{-4}$ for all fractures, and permeability is
$\mathbb K_c=10^4$ for all fractures except
for fractures 4 and 5 which are blocking fractures with $K_b=10^{-4}$ .
Note that we are considering two subcases a) and b) with a pressure gradient which is predominantly vertical and horizontal respectively.
\begin{figure}[ht]
  \centering
    \includegraphics[width=.86\textwidth]{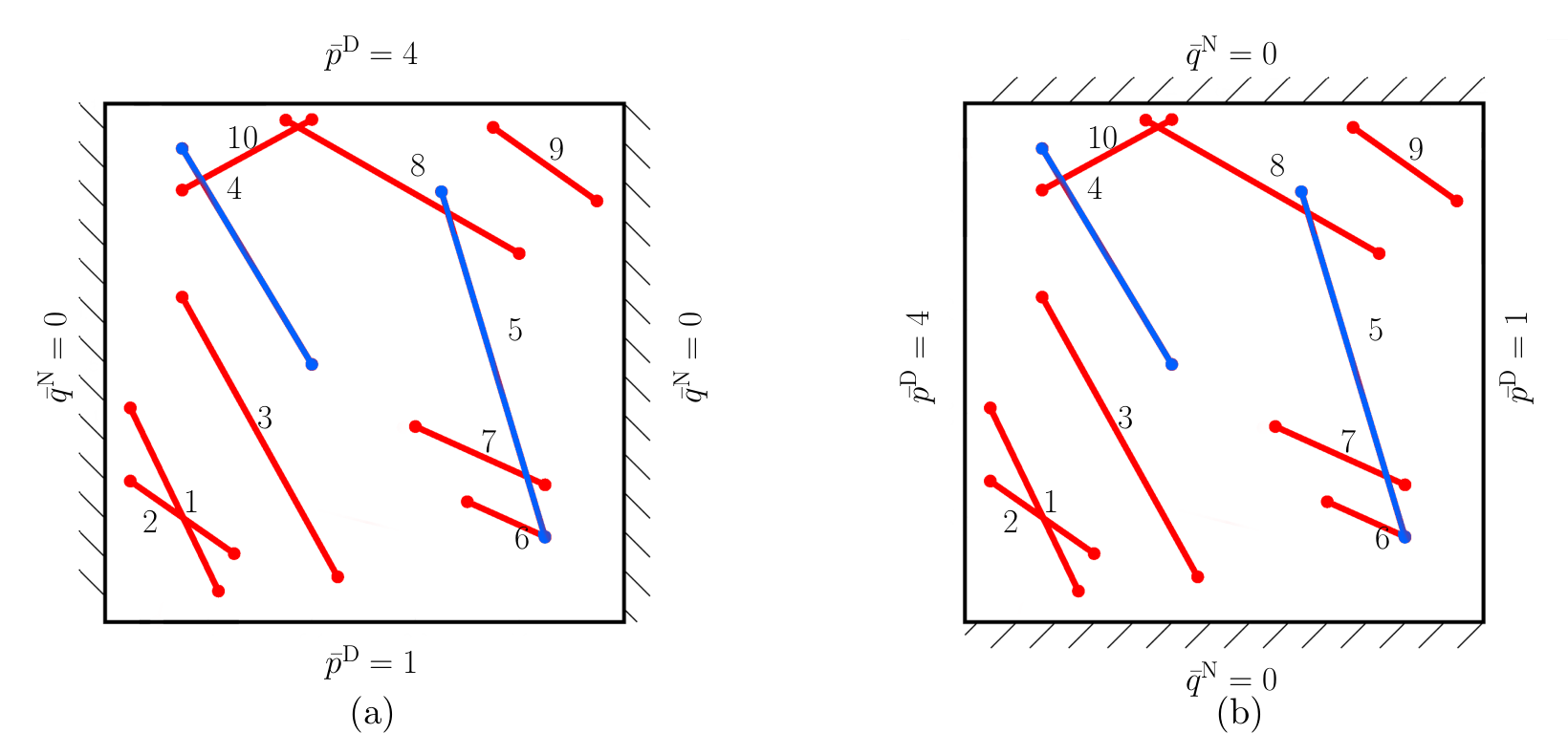}
  \caption{Benchmark 3: computational domain and boundary conditions.}
  \label{fig:complex}
\end{figure}

We apply the method \eqref{fem} on two set of meshes:
a triangular fitted mesh with $1,332$ matrix elements and
$88$ fracture elements which was provided in the git repository
\url{https://git.iws.uni-stuttgart.de/benchmarks/fracture-flow}, see left of Figure \ref{fig:complex0},
and a triangular immersed fitted mesh with $1,370$ matrix elements and
$211$ fracture elements obtained from a background unfitted mesh using the
immersing mesh technique introduced in Section 3.6, see right of Figure \ref{fig:complex0}.
\begin{figure}[ht]
  \centering
    \subfigure[A fitted mesh]{\includegraphics[width = 3.0in]{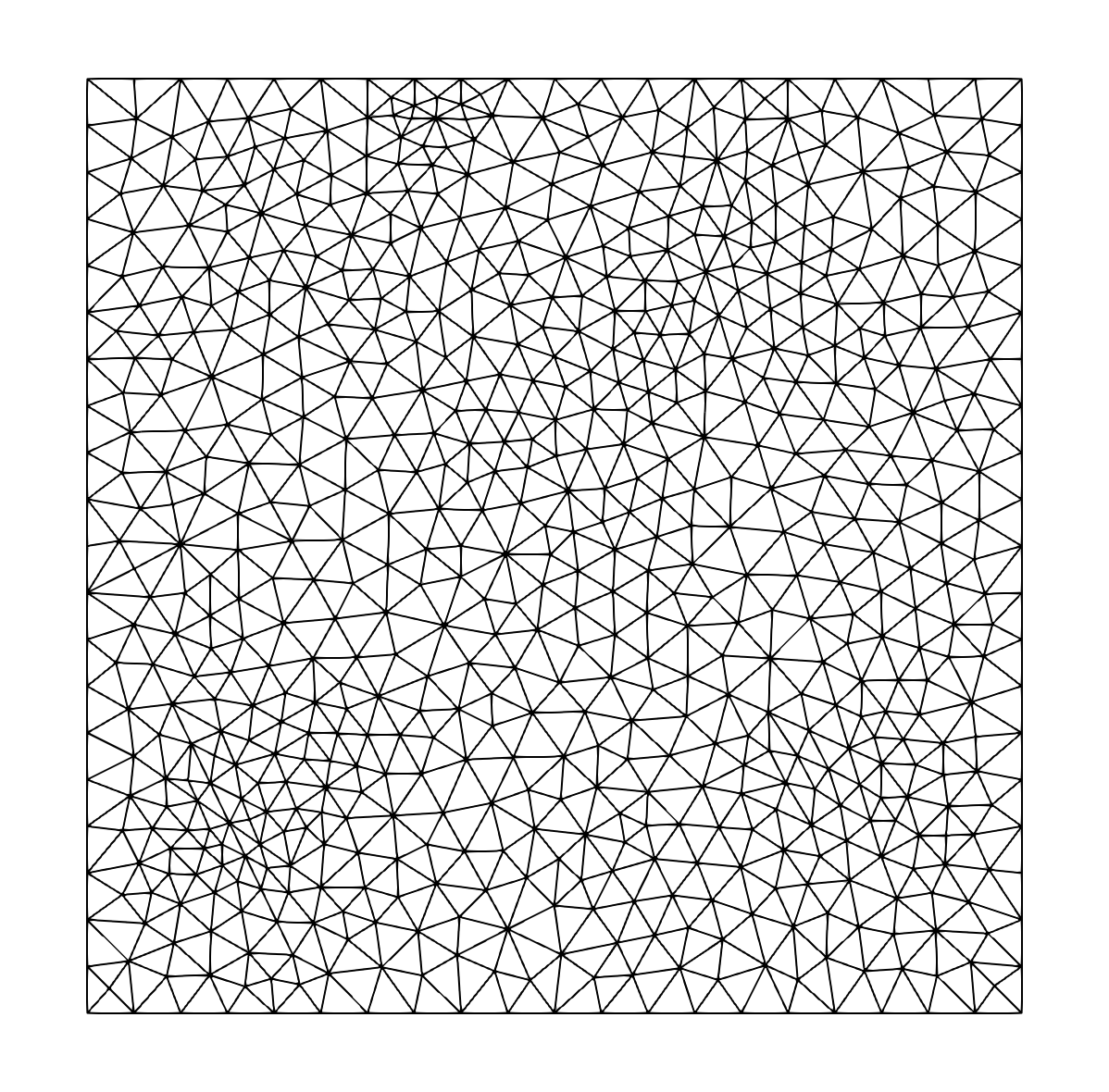}}
      \subfigure[An immersed fitted mesh]{\includegraphics[width = 3.0in]{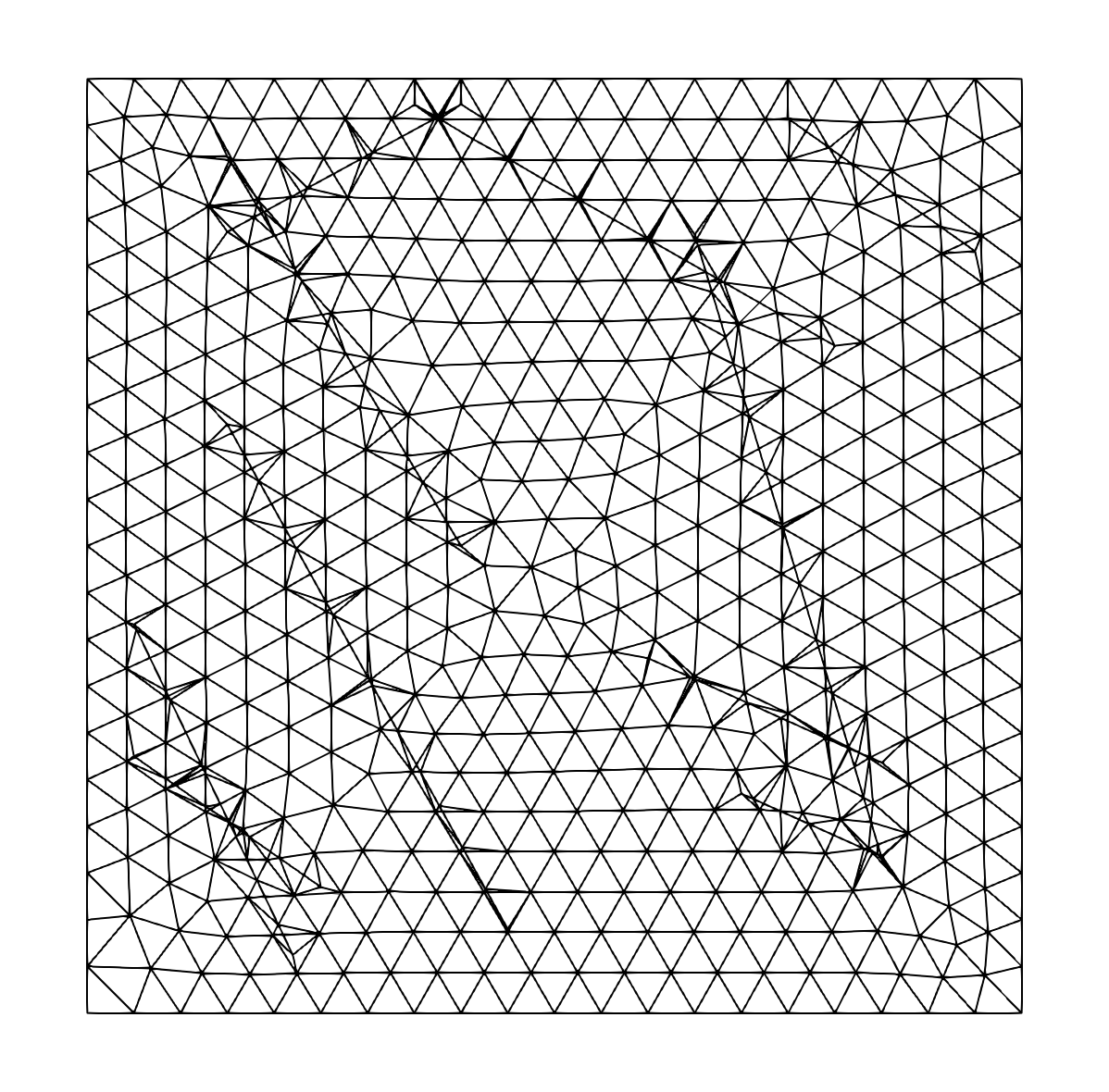}}
  \caption{Benchmark 3: computational meshes.}
  \label{fig:complex0}
\end{figure}
The globally coupled DOFs is $2,066$ for the fitted mesh, and is
$2,211$ for the immersed mesh.
The pressure distributions  along the lines $(0,0.5)$--$(1.0,0.9)$
are  shown in Figure \ref{fig:complex1}.
We observe that the results on the two meshes are very close to each other, and they are in good agreements with the reference data obtained from a mimetic finite difference method on a very fine mesh with $1.8$ million DOFs.

\begin{figure}[ht]
  \centering
  \subfigure[Vertical flow]{\includegraphics[width = 3.0in]{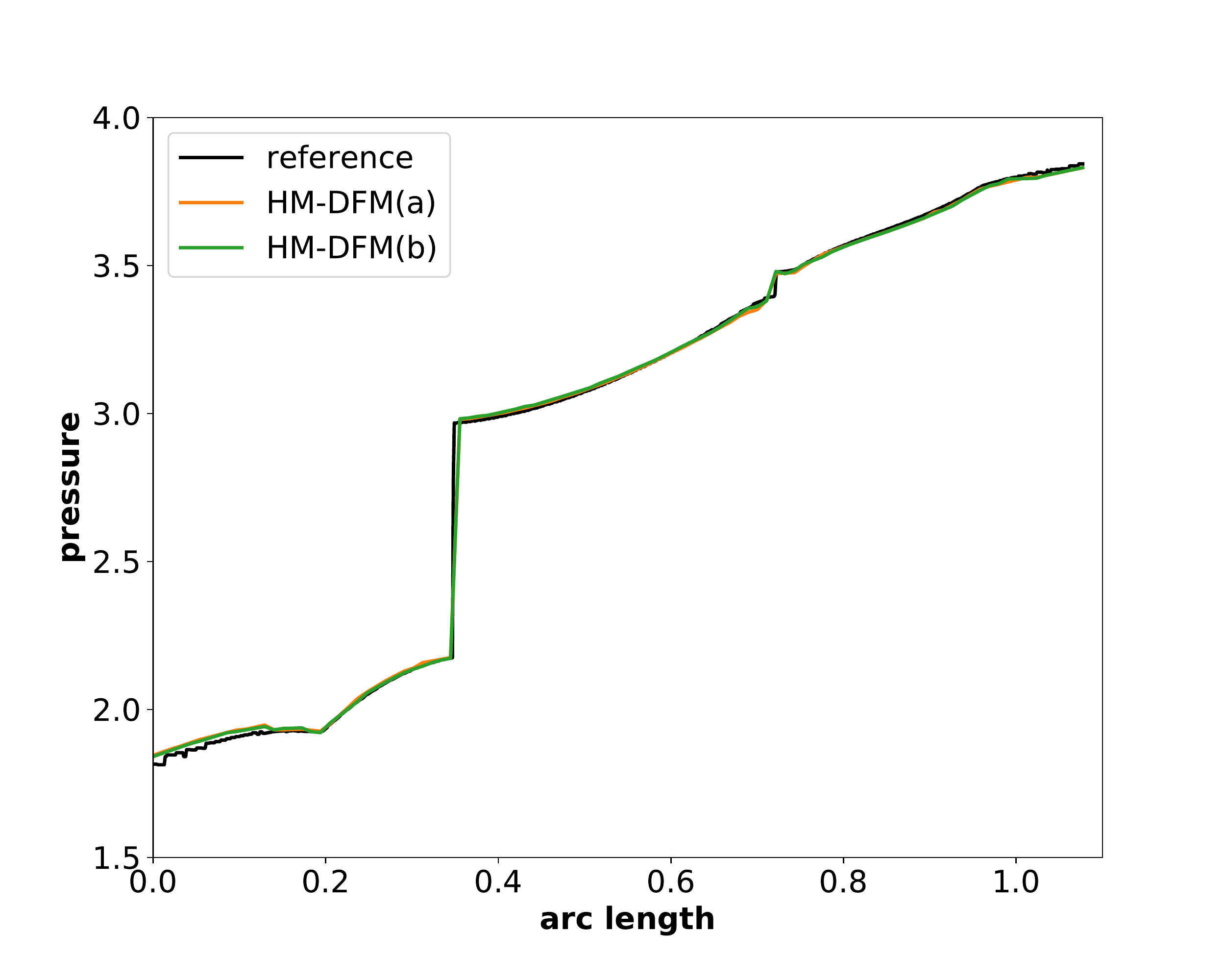}}
\subfigure[Horizontal flow]{\includegraphics[width = 3.0in]{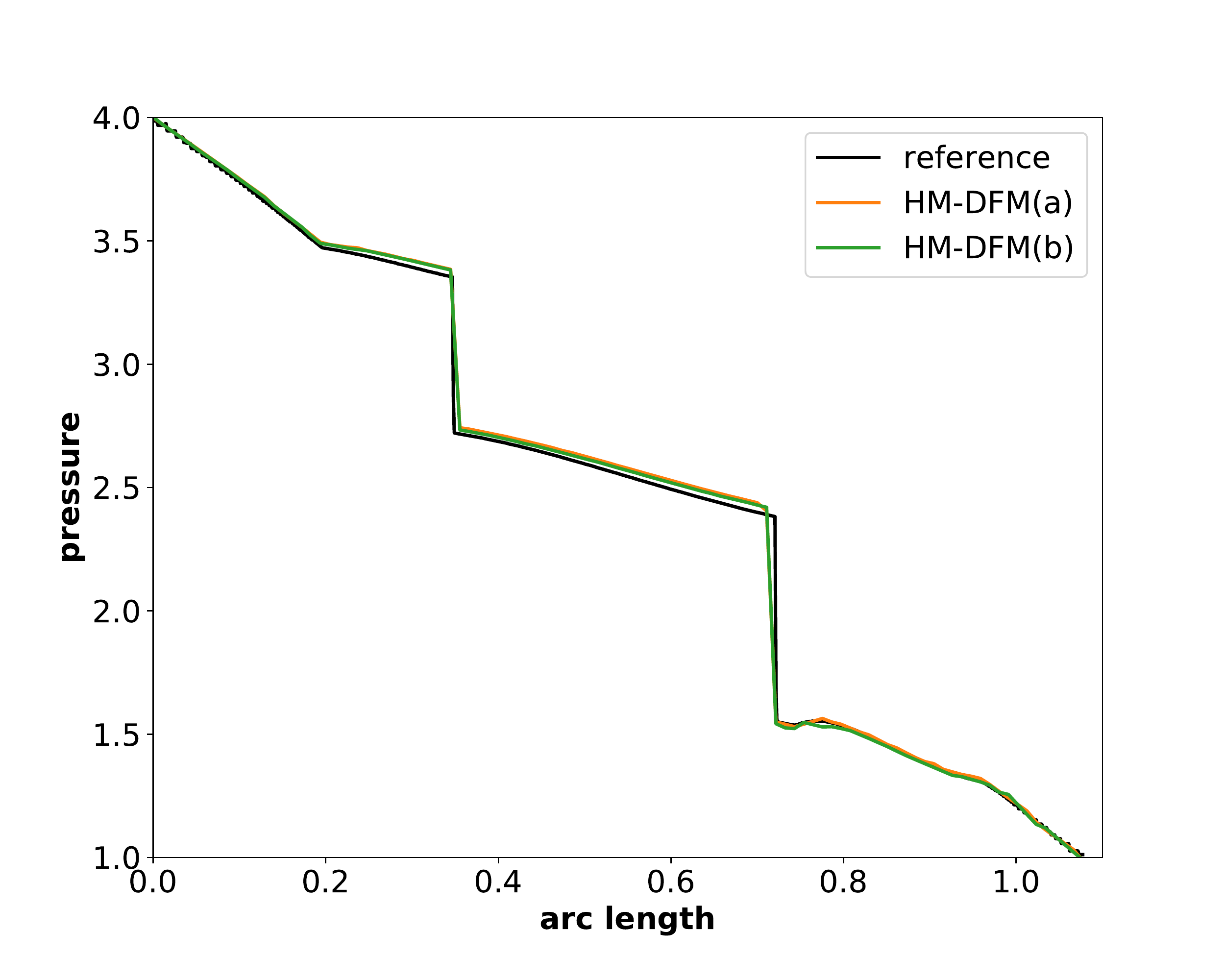}}
  \caption{Benchmark 3: pressure distribution along
  line $(0,0.5)$--$(1-0.9)$.
  HDG-DFM(a) is the numerical solution on the fitted mesh in Figure~\ref{fig:complex0} (a),
    HDG-DFM(b) is the numerical solution on the immersed fitted mesh in Figure~\ref{fig:complex0} (b).
  }
  \label{fig:complex1}
\end{figure}

\subsection{Benchmark 4: a Realistic Case (2D)}
We consider a real set of fractures from an interpreted outcrop in
the Sotra island, near Bergen in Norway.
The size of the domain is 700 $m$ $\times$ 600 $m$ with uniform scalar
permeability $\mathbb K_m=10^{-14}m^2$. The set of fractures is composed of
64 line segments, in which the permeability is $\mathbb K_c=10^{-8}m^2$.
The fracture thickness is $\epsilon=10^{-2}m$. The exact coordinates for the
fracture positions are provided in the above mentioned git repository.
The domain along with boundary conditions is given in  Figure~\ref{fig:real}.
Similar to the previous example, we apply the method \eqref{fem}
on two set of conforming meshes:
a fitted mesh consists of $10,807$ matrix elements and
$1,047$ fracture elements provided in \url{https://git.iws.uni-stuttgart.de/benchmarks/fracture-flow}, see left of Figure~\ref{fig:realX}, and an immersed
fitted mesh consists of $5,473$ matrix elements and
$1541$ fracture elements obtained from a background unfitted mesh using the immersing mesh technique introduced in Section 3.6, see right of Figure~\ref{fig:realX}.
The number of the globally coupled DOFs is $17,253$ for the fitted mesh (a), and $9,753$ for the immersed mesh (b).
\begin{figure}[ht]
  \centering
    \includegraphics[width=.6\textwidth]{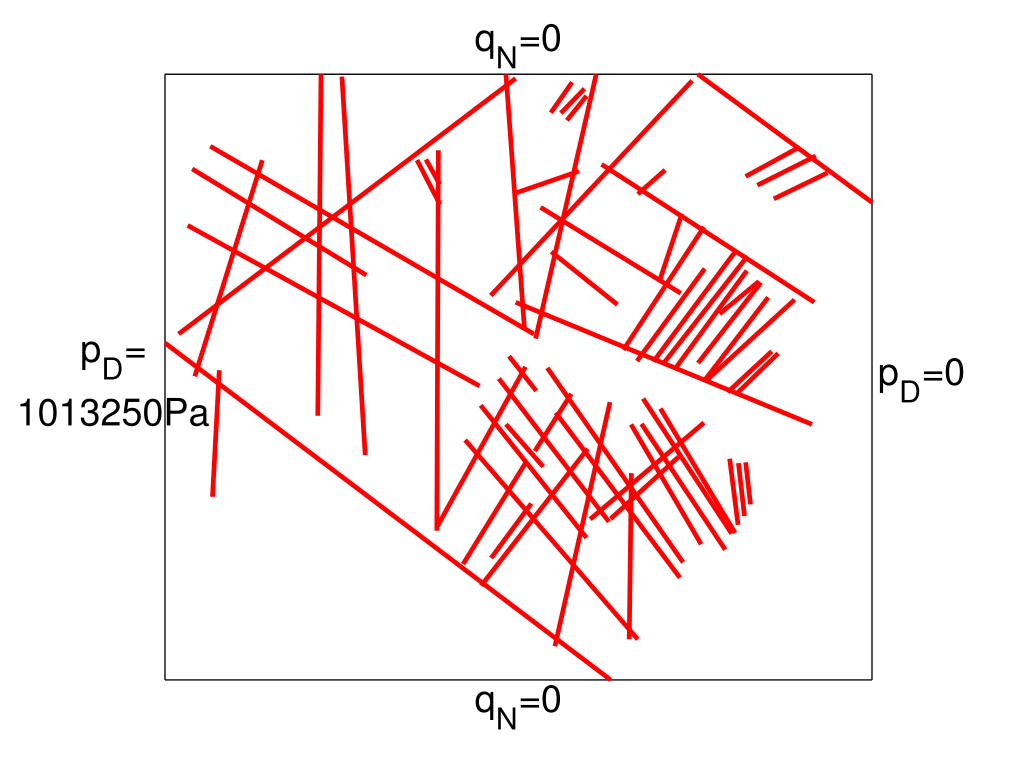}
  \caption{Benchmark 4: Computational domain and boundary conditions.
  }
  \label{fig:real}
\end{figure}

\begin{figure}[ht]
  \centering
    \subfigure[A fitted mesh]{\includegraphics[width = 3.0in]{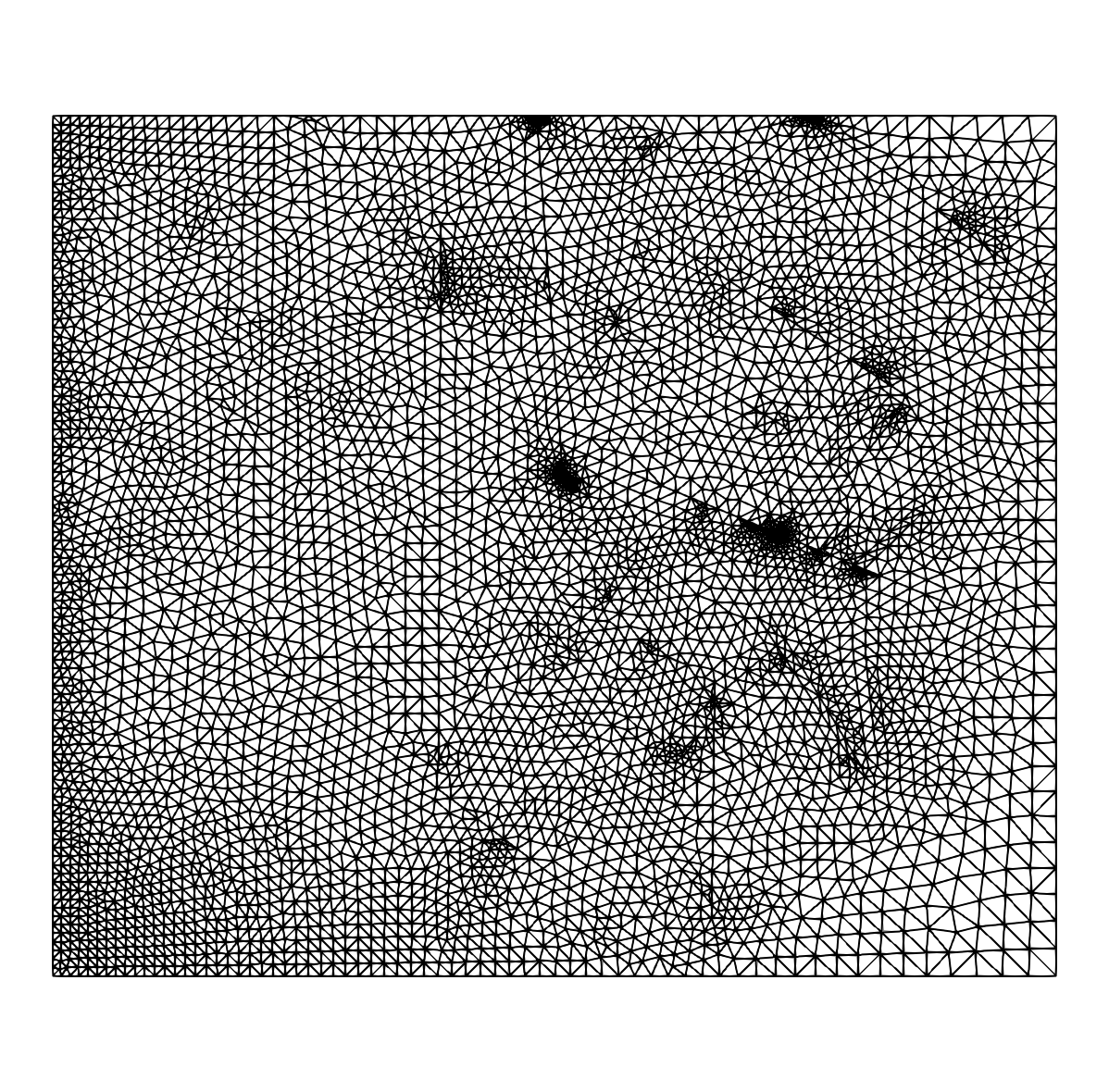}}
      \subfigure[An immersed fitted mesh]{\includegraphics[width = 3.0in]{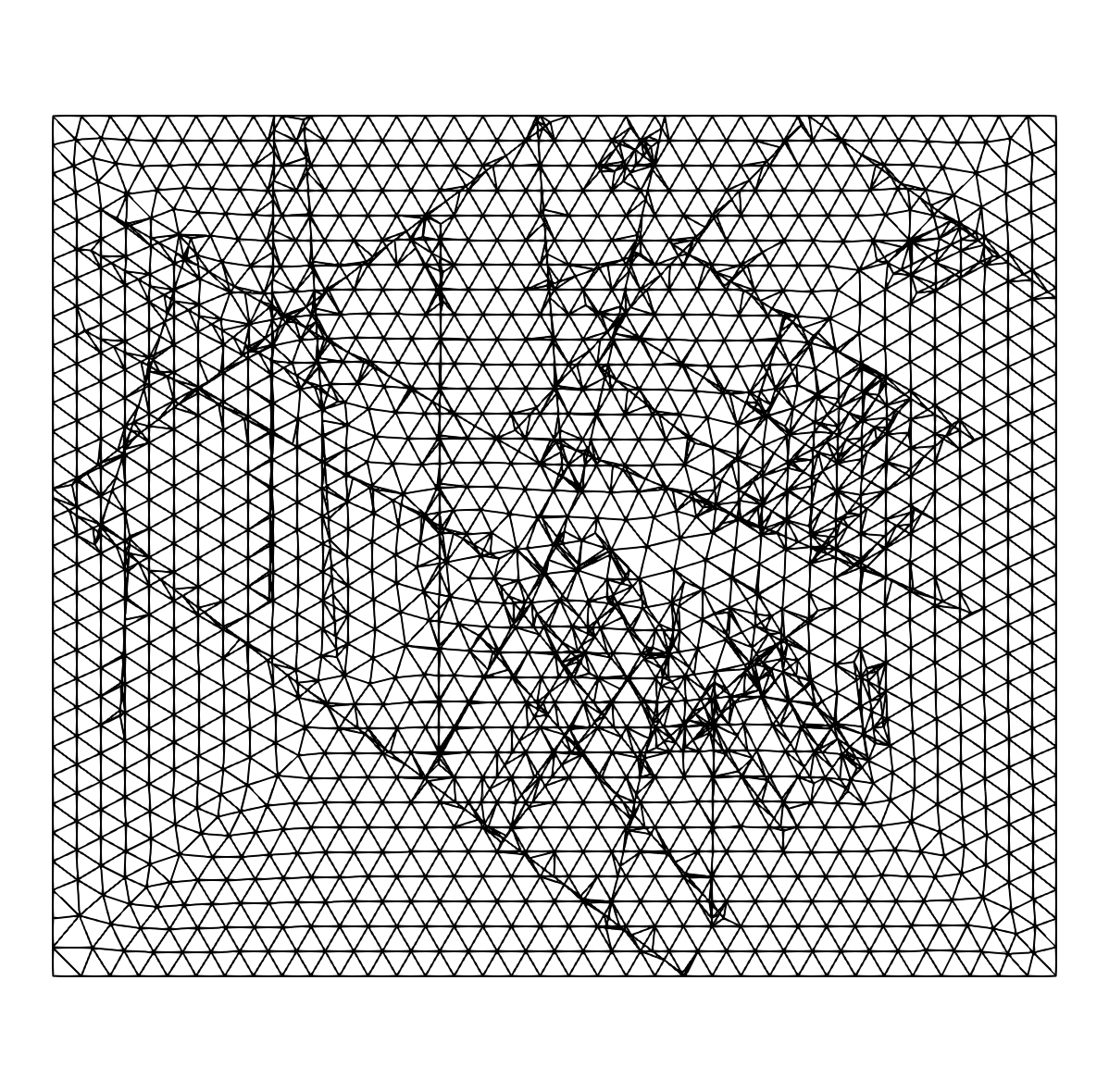}}
  \caption{Benchmark 4: computational meshes.}
  \label{fig:realX}
\end{figure}

The pressure distribution  along the two lines $y=500 m$
and $x=625 m$
are  shown in Figure \ref{fig:real1}, along with the results
for the mortar-DFM method with $25,258$ DOFs from \cite{FLEMISCH2018239}.
We observe that the three results are in good agreements with each other, with the HDG-DFM(b) using the least amount of DOFs.
\begin{figure}[ht]
  \centering
  \begin{tabular}{cc}
    \includegraphics[width=.43\textwidth]{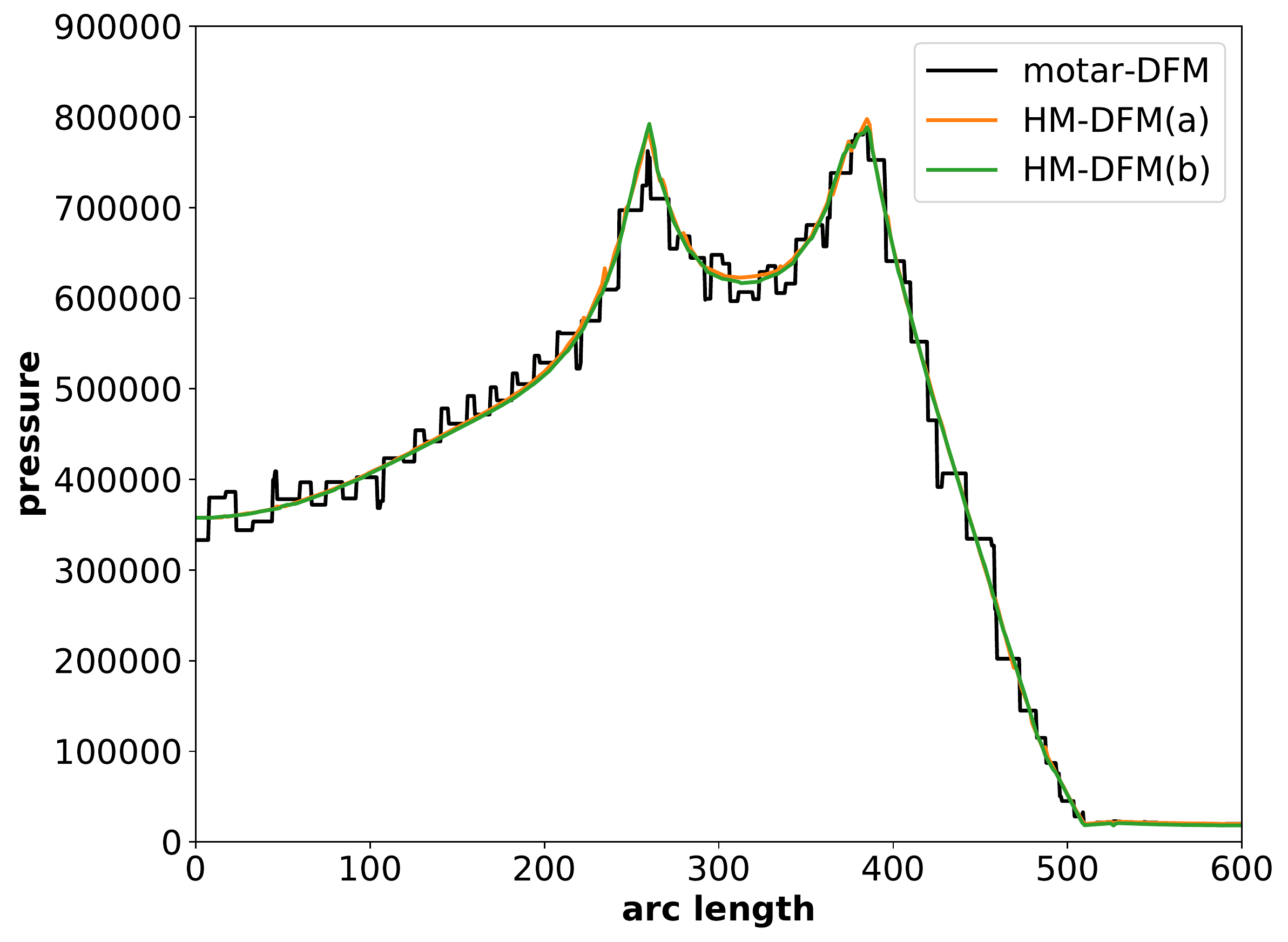}&
    \includegraphics[width=.43\textwidth]{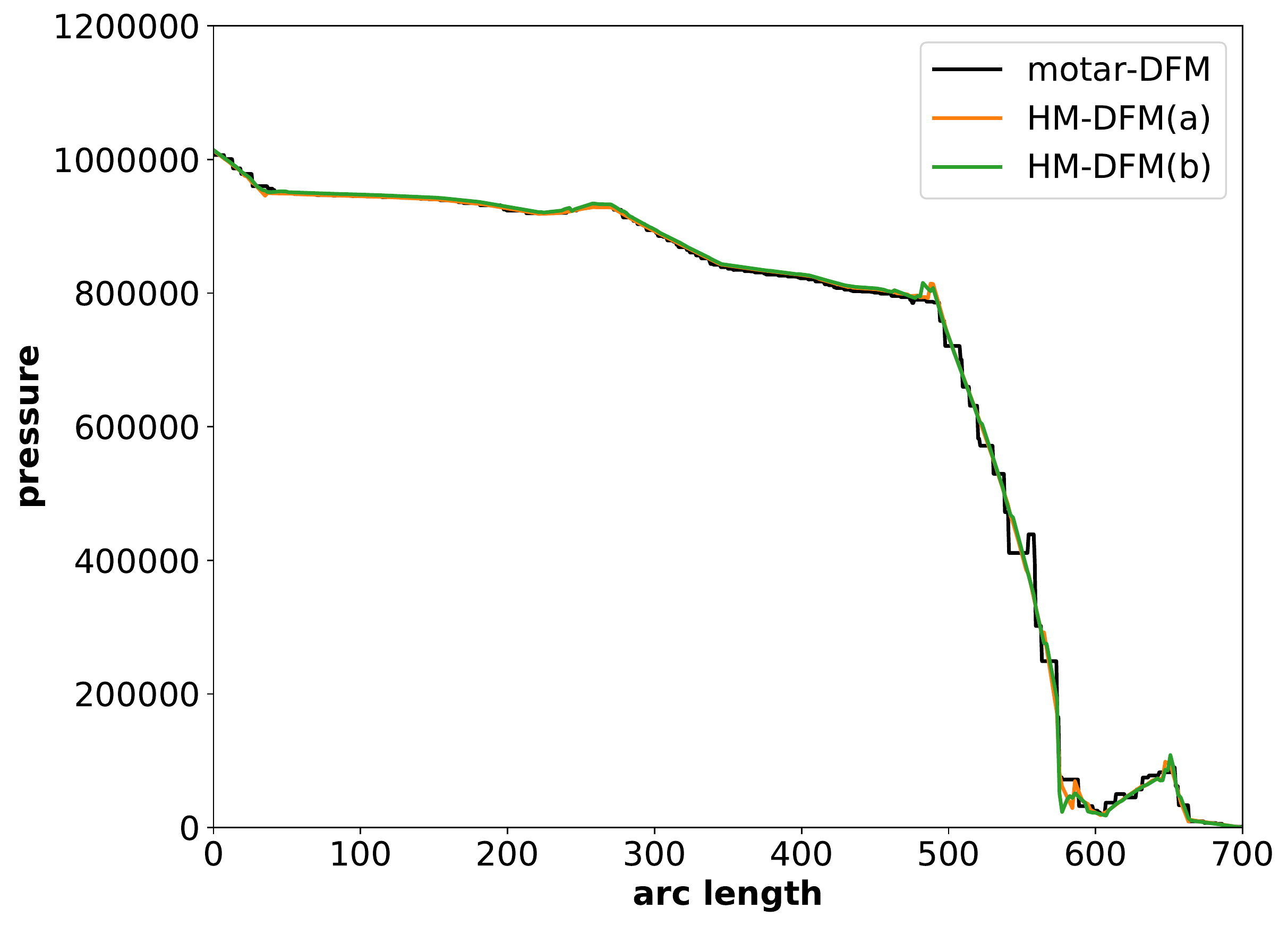}\\
    (a) $y=500 m$
                                                &(b) $x=625 m$
  \end{tabular}
  \caption{Benchmark 4: Pressure distribution along lines $y=500 m$ (left) and
  $x=625 m$ (right).
  HDG-DFM(a) is the numerical solution on the fitted mesh in Figure~\ref{fig:realX}(a), HDG-DFM(b) is the numerical solution on the immersed fitted mesh in Figure~\ref{fig:realX}(b).
  }
  \label{fig:real1}
\end{figure}

\subsection{Benchmark 5: Single Fracture (3D)}
This is the first benchmark case proposed in \cite{Berre_2021}.
To be consistent with the notation in \cite{Berre_2021},
the pressure and permeabilities are renamed as hydraulic head and hydraulic
conductivities, respectively for this test case and the three examples following.
Figure \ref{fig:single} illustrates the geometrical description.
Here the domain $\Omega$ is a cube-shaped region $(0\mathrm{m}, 100\mathrm{m})\times (0\mathrm{m},
100\mathrm{m})\times (0\mathrm{m}, 100\mathrm{m})$ which is crossed by a conductive planar fracture,
$\Omega_2$, with a thickness
of  $\epsilon=10^{-2}m$.
The matrix domain consists of subdomains $\Omega_{3,1}$, above the fracture, and
$\Omega_{3,2}$ and $\Omega_{3,3}$ below. The subdomain $\Omega_{3,3}$
represents a heterogeneity within the rock matrix.
The matrix conductivities are given in Figure \ref{fig:single}, and the
fracture conductivity is $\mathbb{K}_c = 0.1$ so that
$\epsilon \mathbb K_c= 10^{-3}$.
Inflow into the system occurs through a narrow band defined by $\{0 m\}\times(0 m,
100 m)\times(90 m, 100 m)$.
Similarly, the outlet is a narrow band defined by $(0 m, 100 m) \times \{0 m\} \times (0 m,
10 m).$
At the inlet and outlet bands, we impose the hydraulic head
$h_{in}=4 m$ and $h_{out}=1 m$ respectively. The remaining parts of the boundary
are assigned no-flow conditions.
Following the setup in \cite{Berre_2021},
we set $c_B=0.01 m^{-3}$ at the inlet boundary
for the transport problem. The matrix porosity $\phi$ is taken to be $0.2$ on $\Omega_{3,1}\cup \Omega_{3,2}$ and
$0.25$ on $\Omega_{3,3}$, and the fracture porosity $\phi_c$ is taken to be $0.4$.
The final time of simulation is $T=10^9 s$, and the time step size is
$\Delta t = 10^7 s$.
\begin{figure}[ht]
  \centering
    \includegraphics[width=.7\textwidth]{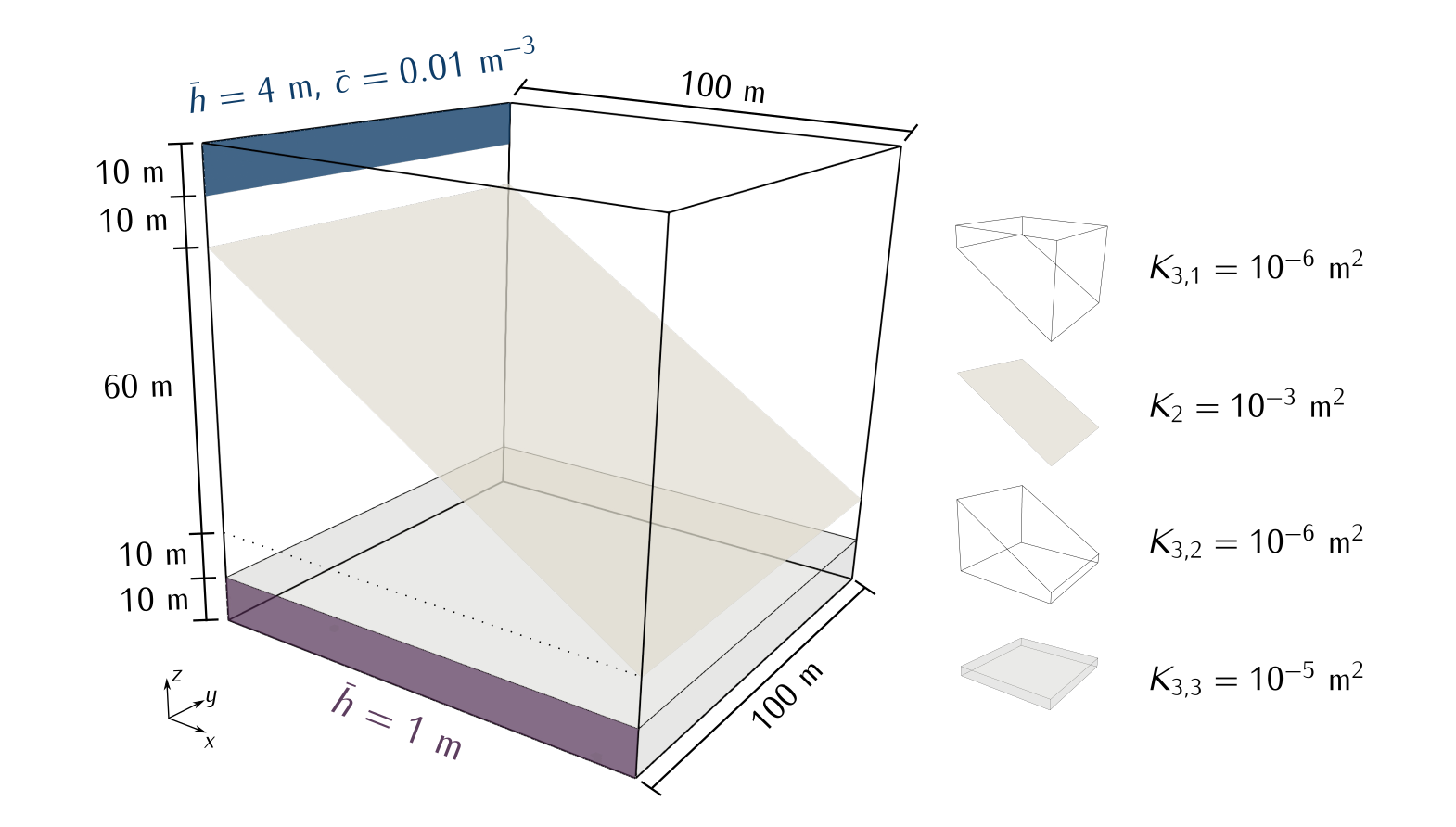}
    \caption{Benchmark 5: Conceptual model and geometrical description
    of the domain.}
    \label{fig:single}
  \end{figure}

  We perform the method \eqref{fem} and \eqref{transport-eq} on a coarse tetrahedral mesh with
  $10,232$ matrix elements and $448$ fracture elements and a
  fine tetrahedral mesh with $111,795$ matrix elements and
  $1,758$ fracture elements.
  The number of the globally coupled DOFs on the coarse mesh is $23,377$,
  while that on the fine mesh is $235,619$.
The hydraulic head along the line $(0\mathrm{m}, 100\mathrm{m}, 100\mathrm{m})$--$(100\mathrm{m}, 0\mathrm{m}, 0\mathrm{m})$
is shown in Figure \ref{fig:single1}, along with
reference data and published spread provided in the git repository
\url{https://git.iws.uni-stuttgart.de/benchmarks/fracture-flow-3d}.
The reference data in Figure \ref{fig:single1} is obtained from
the {\sf USTUTT-MPFA} method on a mesh with approximately
1 million matrix elements,
while the shaded region depicts the area between the 10th and the 90th
percentile of the published results in \cite{Berre_2021}
on mesh refinement level 1 (left, $\sim 10k$ cells) and refinement level
2 (right, $\sim 100k$ cells).
The match number results from evaluating at 100 evenly distributed evaluation
points if the value for the {\sf HM-DFM} method
is between the respective lower and upper value.
We observe that our result agrees with the reference
values quite well, especially on the fine mesh.
\begin{figure}[ht]
  \centering
  \begin{tabular}{cc}
    \includegraphics[width=0.48\textwidth]{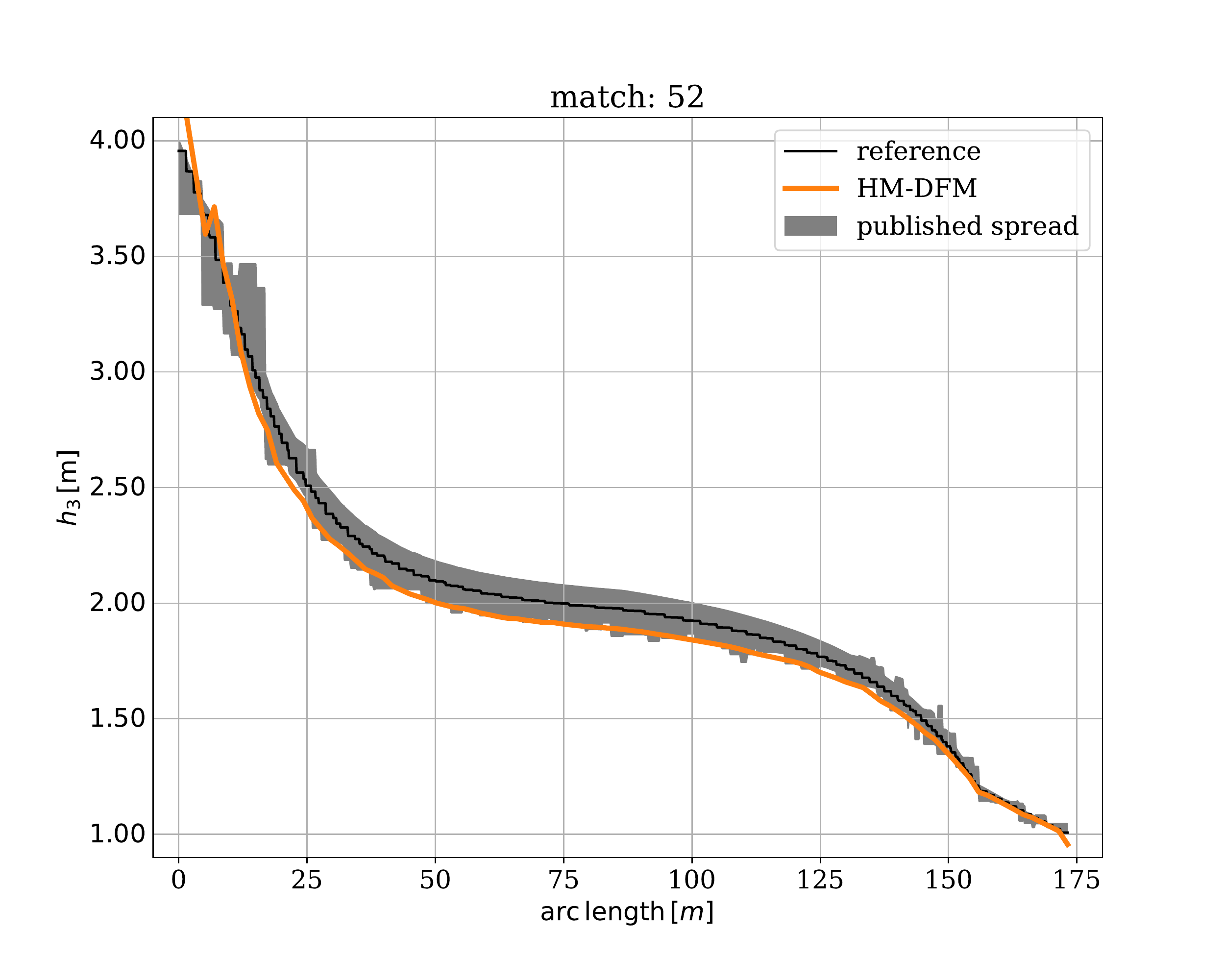}&
    \includegraphics[width=.48\textwidth]{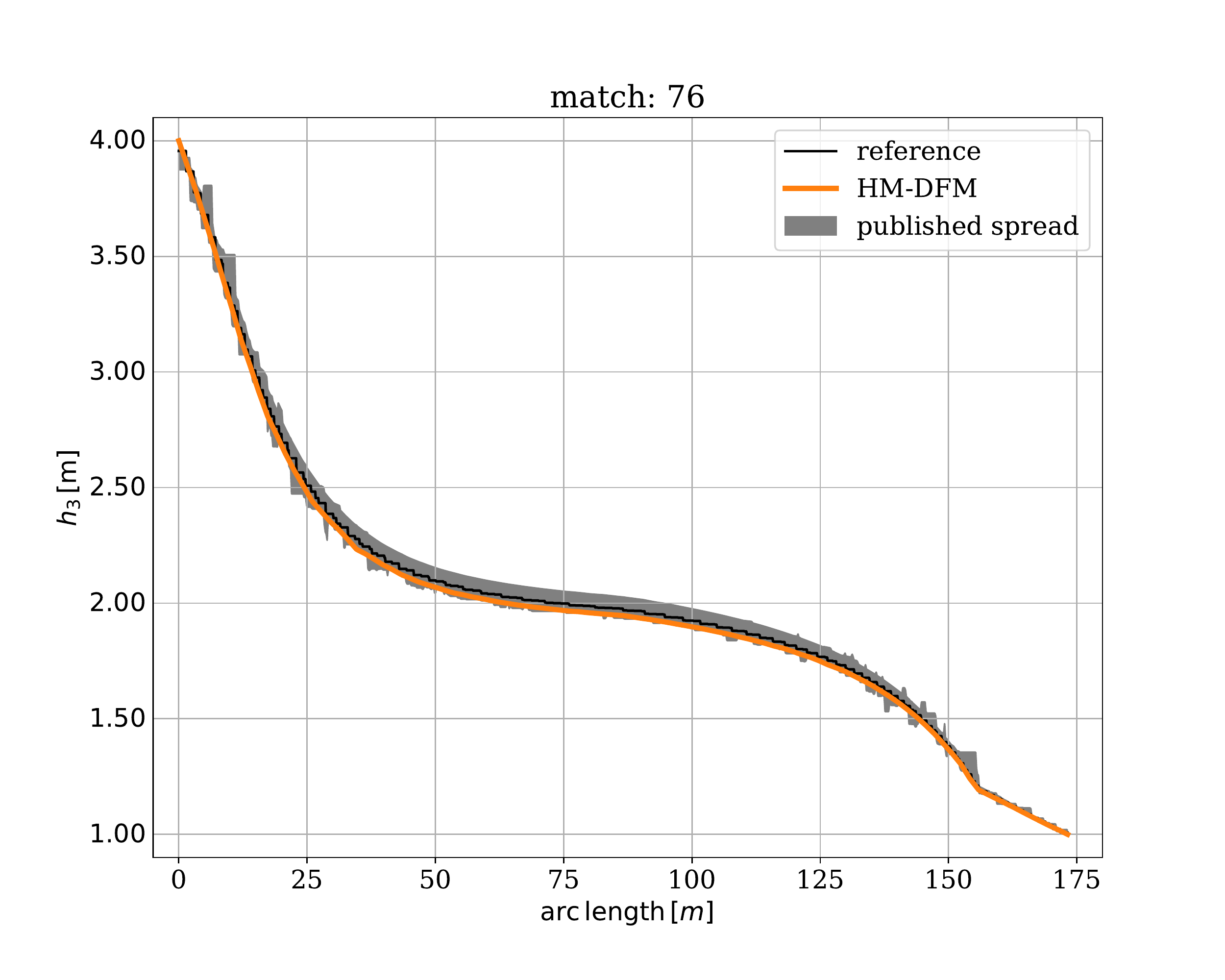}\\
    (a) $\sim 10k$ cells. &
    (b) $\sim 100k$ cells.
  \end{tabular}
    \caption{Benchmark 5: Hydraulic head in the matrix over the line
    $(0\mathrm{m}, 100\mathrm{m}, 100\mathrm{m})$--$(100\mathrm{m}, 0\mathrm{m}, 0\mathrm{m})$.
Left: results on a coarse mesh with about $10k$ cells.
Right: results on a fine mesh with about $100k$ cells.
  }
    \label{fig:single1}
  \end{figure}

Moreover, we plot the
matrix concentration along the line $(0\mathrm{m}, 100\mathrm{m}, 100\mathrm{m})$--$(100\mathrm{m}, 0\mathrm{m}, 0\mathrm{m})$
 in Figure \ref{fig:single1X}, and the
fracture concentration along the line $(0\mathrm{m}, 100\mathrm{m}, 80\mathrm{m})$--$(100\mathrm{m}, 0\mathrm{m}, 20\mathrm{m})$
 at final time $T=10^9 s$ in Figure \ref{fig:single1Y}, together with
 the published spread provided in the git repository,
which depicts the area between the 10th and the 90th percentile of the published results in \cite{Berre_2021} using
similar first order finite volume schemes with implicit Euler time stepping and $\Delta t =10^7 s$.
We observe that our results agree quite well with the provided data.
 \begin{figure}[ht]
  \centering
  \begin{tabular}{cc}
    \includegraphics[width=0.48\textwidth]{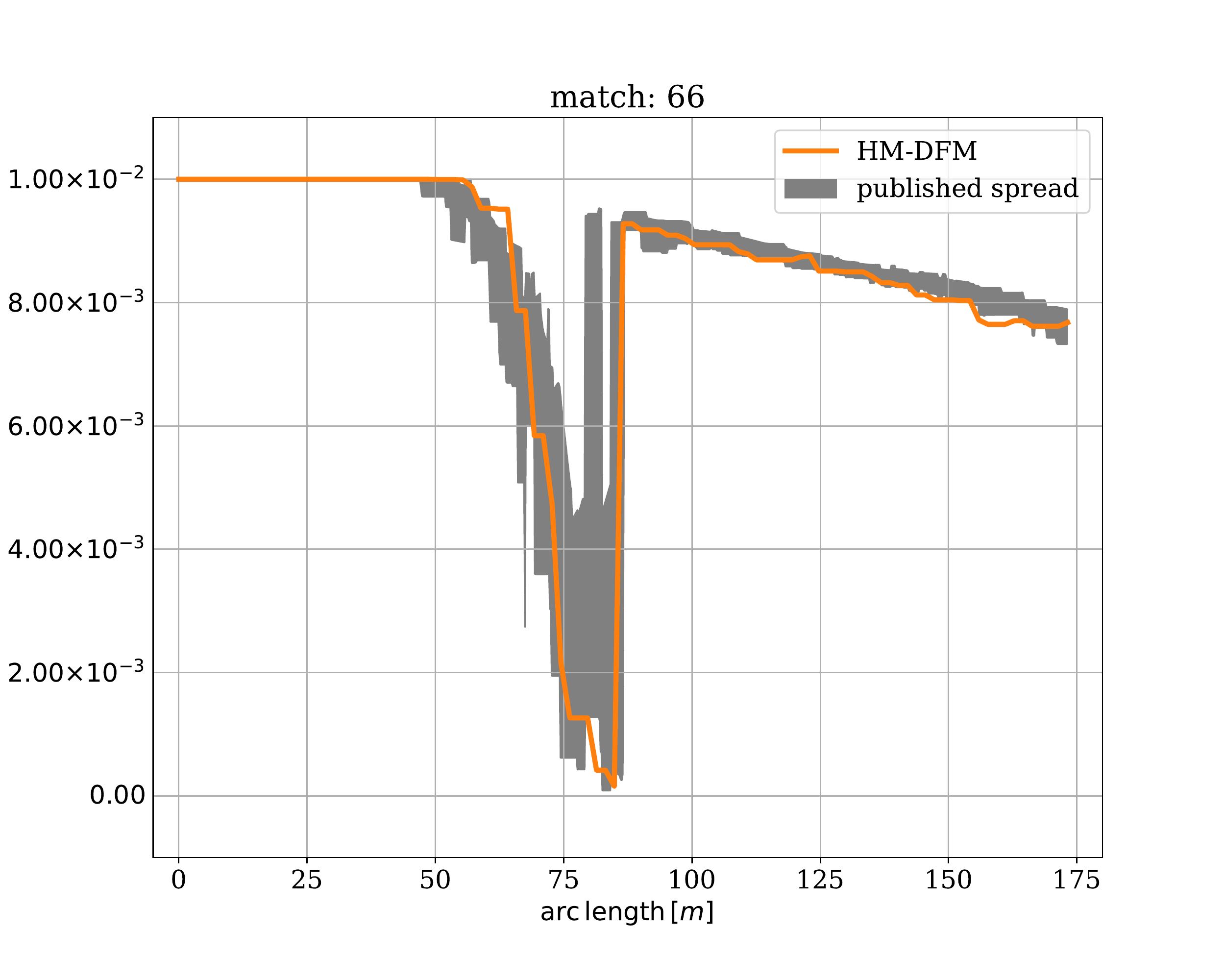}&
    \includegraphics[width=.48\textwidth]{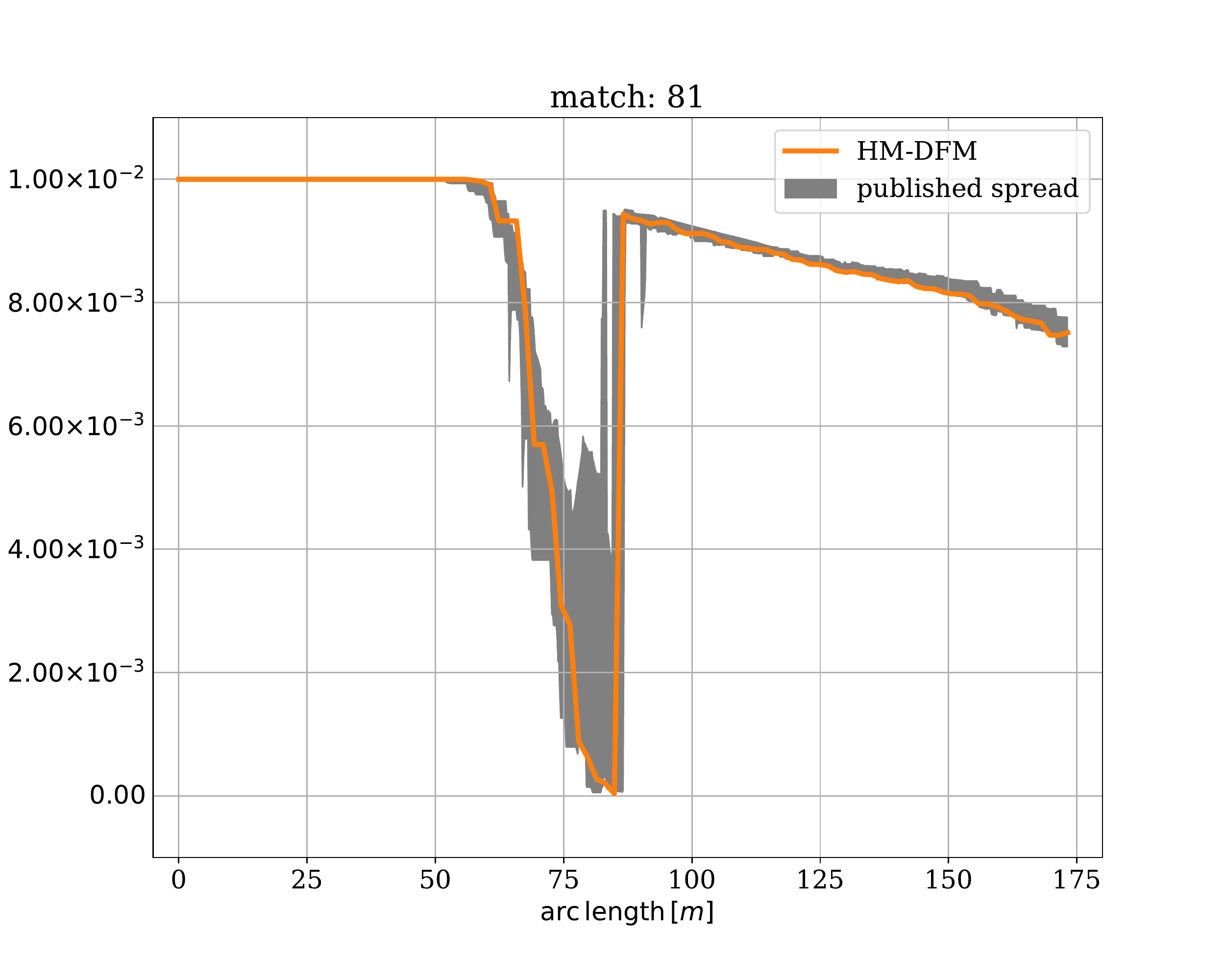}\\
    (a) $\sim 10k$ cells. &
    (b) $\sim 100k$ cells.
  \end{tabular}
    \caption{Benchmark 5: Hydraulic head in the matrix over the line
    $(0\mathrm{m}, 100\mathrm{m}, 100\mathrm{m})$--$(100\mathrm{m}, 0\mathrm{m}, 0\mathrm{m})$.
Left: results on a coarse mesh with about $10k$ cells.
Right: results on a fine mesh with about $100k$ cells.
  }
    \label{fig:single1X}
  \end{figure}
  \begin{figure}[ht]
  \centering
  \begin{tabular}{cc}
    \includegraphics[width=0.48\textwidth]{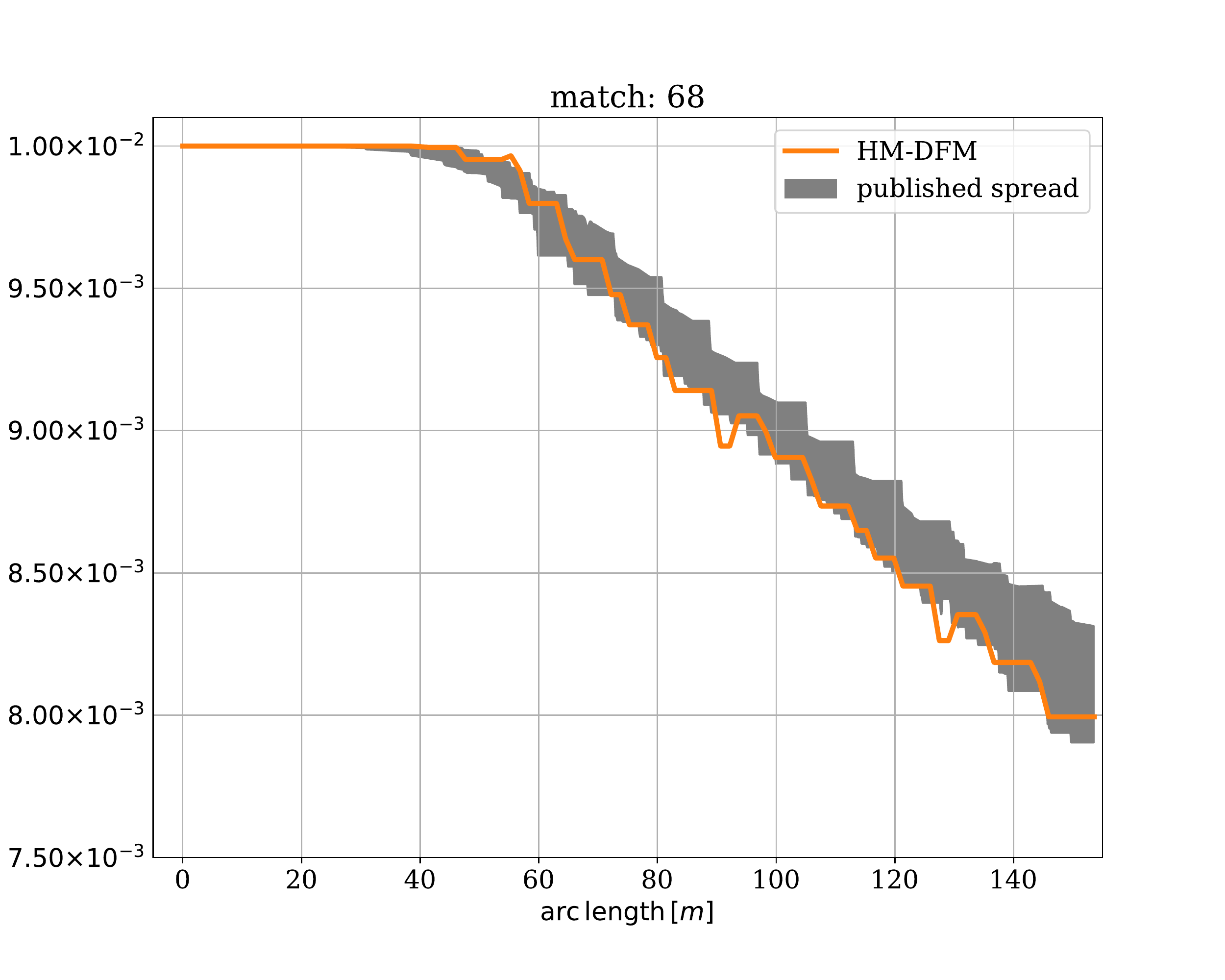}&
    \includegraphics[width=.48\textwidth]{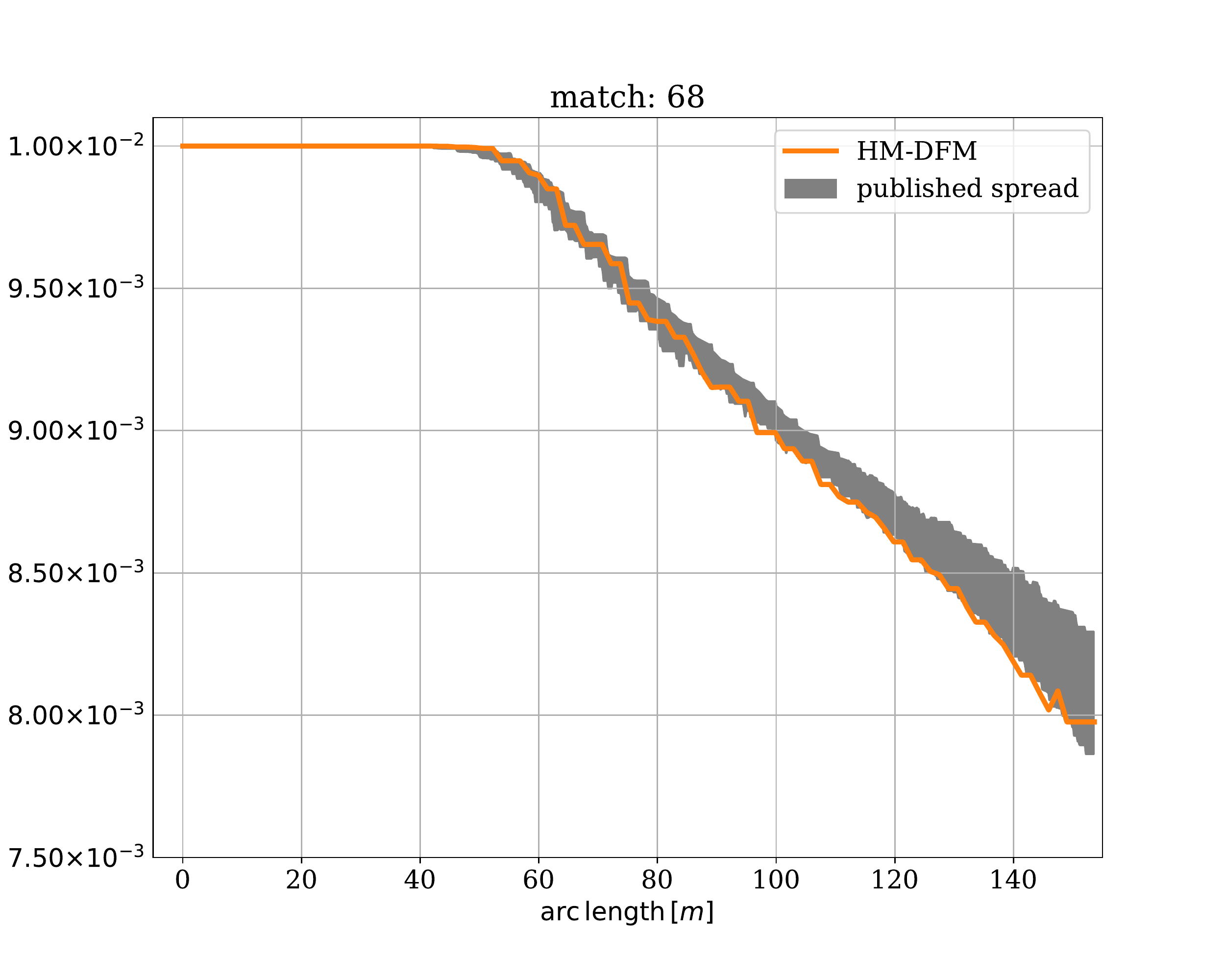}\\
    (a) $\sim 10k$ cells. &
    (b) $\sim 100k$ cells.
  \end{tabular}
    \caption{Benchmark 5: Fracture concentration over the line
    $(0\mathrm{m}, 100\mathrm{m}, 80\mathrm{m})$--$(100\mathrm{m}, 0\mathrm{m}, 20\mathrm{m})$.
Left: results on a coarse mesh with about $10k$ cells.
Right: results on a fine mesh with about $100k$ cells.
  }
    \label{fig:single1Y}
  \end{figure}

\subsection{Benchmark 6: Regular Fracture Network (3D)}
This is the second benchmark case proposed in \cite{Berre_2021}, which is a 3D analog of Benchmark 2.
The domain is given by the unit cube $\Omega=(0\mathrm{m}, 1\mathrm{m})^3$ and contains
9 regularly oriented fractures, as illustrated in
Figure \ref{fig:regular3D}.
Dirichlet boundary condition $p=\bar h =1 \mathrm{m}$
is imposed on the boundary
$\Gamma_D = \{(x,y,z)\in \partial \Omega: x, y, z > 0.875\mathrm{m}\}$,
Neumann boundary condition
$\bld u\cdot\bld n = -1\mathrm{m/s}$ is imposed on the boundary
$\partial\Omega_{in} = \{(x,y,z)\in \partial \Omega: x, y, z < 0.25\mathrm{m}\}$,
and no-flow boundary condition is imposed on the remaining boundaries.
The heterogeneous matrix conductivity is illustrated in Figure \ref{fig:regular3D}, and the
fracture conductivity is either $\mathbb{K}_c = 10^4 \mathrm{m^2}$, which represents
a conductive fracture or
$K_b=10^{-4}\mathrm{m^2}$ which represents a blocking fracture. The fracture thickness is
$\epsilon = 10^{-4}\mathrm{m}$.
For the transport equation,
matrix porosity is taken to be $\phi=0.1$,
conductive fracture concentration is $\phi_c=0.9$, and
the inflow boundary condition $c_B=1 m^{-3}$ is set on the inlet boundary $\partial\Omega_{in}$.
Final time of the simulation is $T=0.25 s$.

  \begin{figure}[ht]
  \centering
    \includegraphics[width=.7\textwidth]{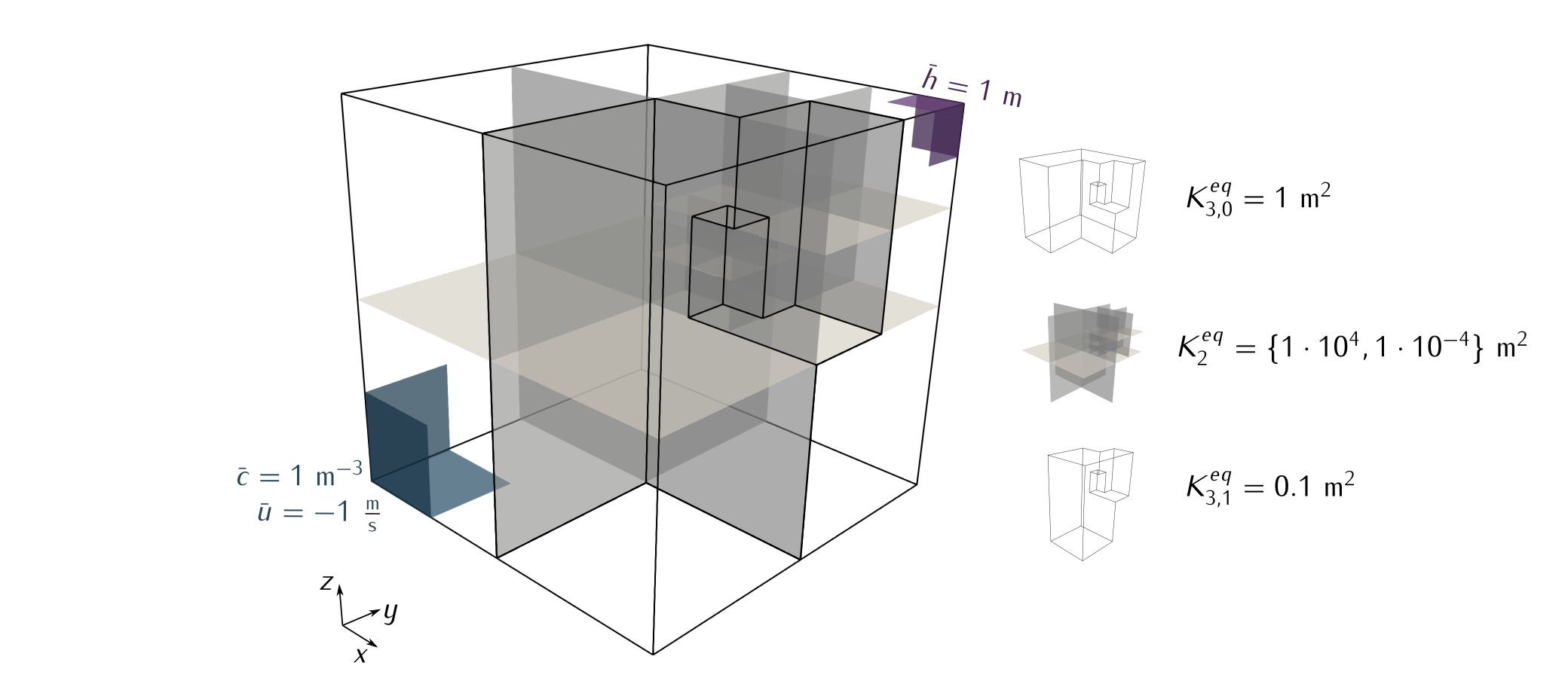}
    \caption{Benchmark 6: Conceptual model and geometrical description
    of the domain.}
    \label{fig:regular3D}
  \end{figure}

  We perform the method \eqref{fem} 
  on a coarse {\it fitted} tetrahedral mesh with
  $4,375$ matrix elements and $944$ fracture elements and a
  fine tetrahedral mesh with $36,336$ matrix elements and
  $4,524$ fracture elements.
  The number of the globally coupled DOFs on the coarse mesh is $13,373$ for the conductive
  fracture case and $8,334$ for the blocking fracture case (only DOFs for
  $\widehat p_h$ are global DOFs in this case),
  while that on the fine mesh is $94,738$ for the conductive fracture case and
  $70,881$ for the blocking fracture case.
The hydraulic head along the diagonal line $(0\mathrm{m}, 0\mathrm{m},
0\mathrm{m})$--$(1\mathrm{m}, 1\mathrm{m}, 1\mathrm{m})$
is shown in Figure \ref{fig:regular3DC} for the conductive fracture case and in Figure \ref{fig:regular3DB} for the blocking fracture case.
We observe that our results agree with the reference values very well,
which were obtained from the {\sf USTUTT-MPFA} method on a
mesh with approximately 1 million matrix elements.
The small derivation of our result on the left panel of Figure~\ref{fig:regular3DC} with the reference data is acceptable due to the use of a very coarse mesh.

\begin{figure}[ht]
  \centering
  \begin{tabular}{cc}
    \includegraphics[width=0.48\textwidth]{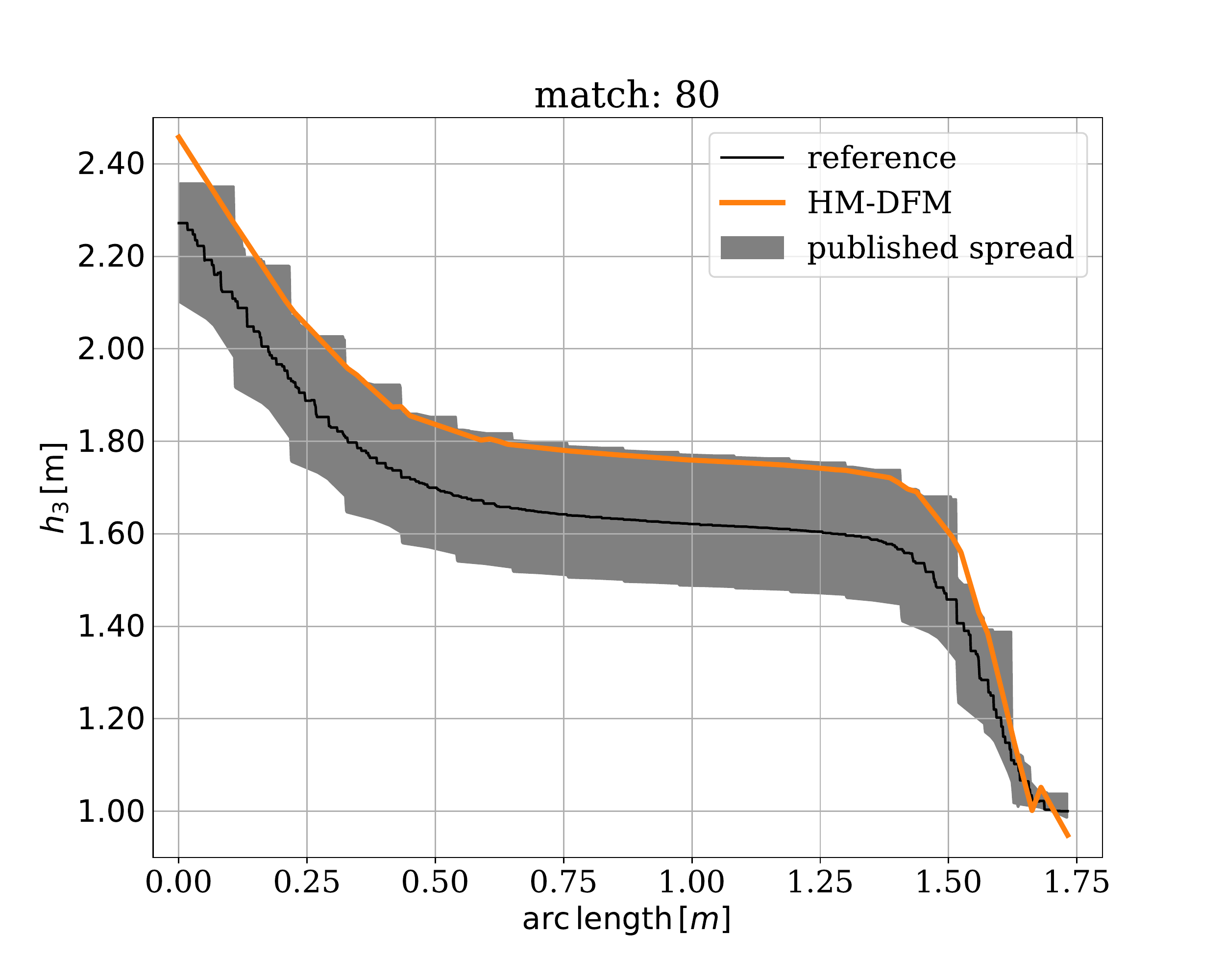}&
    \includegraphics[width=.48\textwidth]{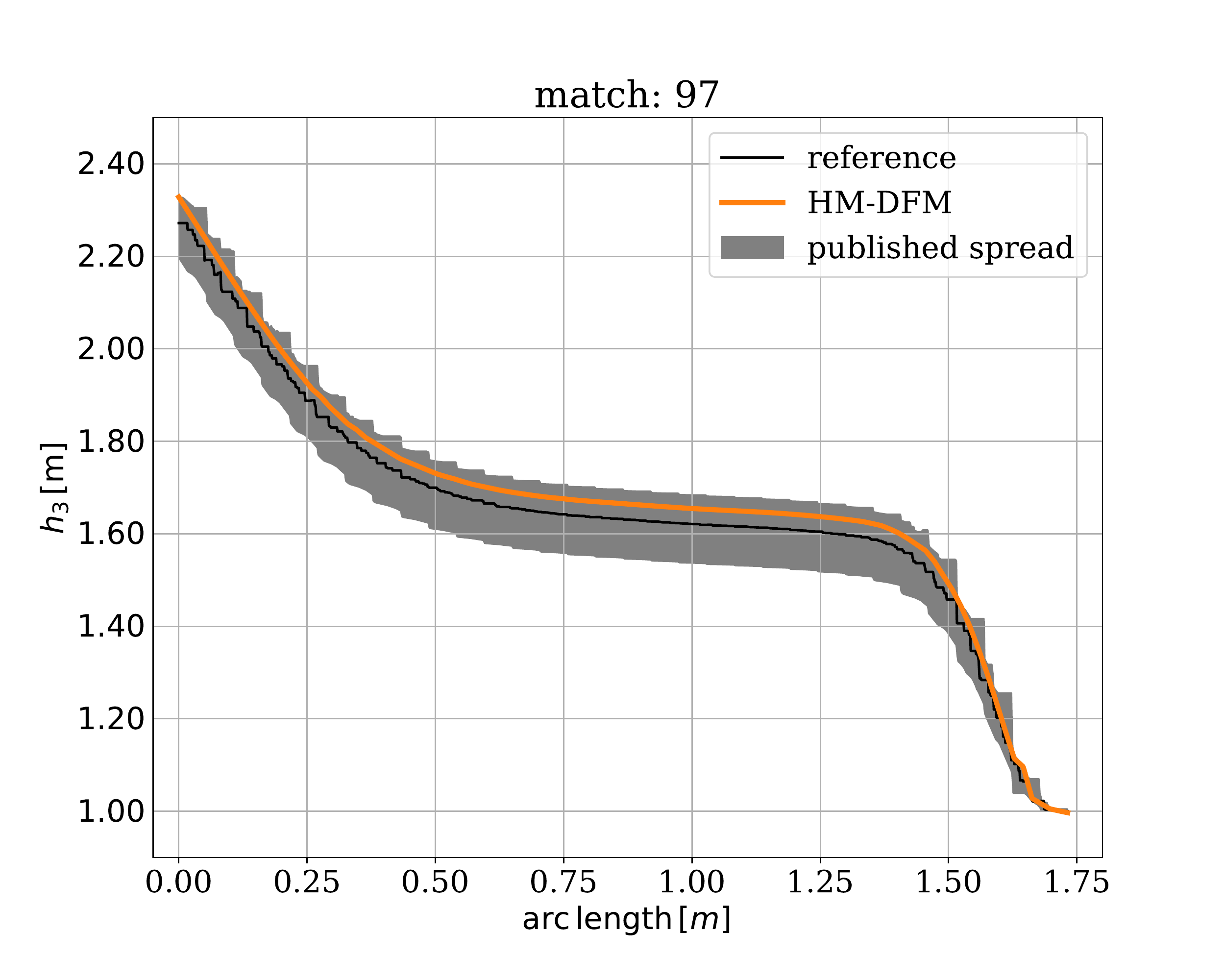}\\
    (a) $\sim 4k$ cells. &
    (b) $\sim 36k$ cells.
  \end{tabular}
    \caption{Benchmark 6 (conductive fractures): Hydraulic head in the matrix over the line
    $(0\mathrm{m}, 0\mathrm{m}, 0\mathrm{m})$--$(1\mathrm{m}, 1\mathrm{m}, 1\mathrm{m})$.
Left: results on a coarse mesh with about $4k$ cells.
Right: results on a fine mesh with about $36k$ cells.
  }
    \label{fig:regular3DC}
  \end{figure}

  \begin{figure}[ht]
  \centering
  \begin{tabular}{cc}
    \includegraphics[width=0.48\textwidth]{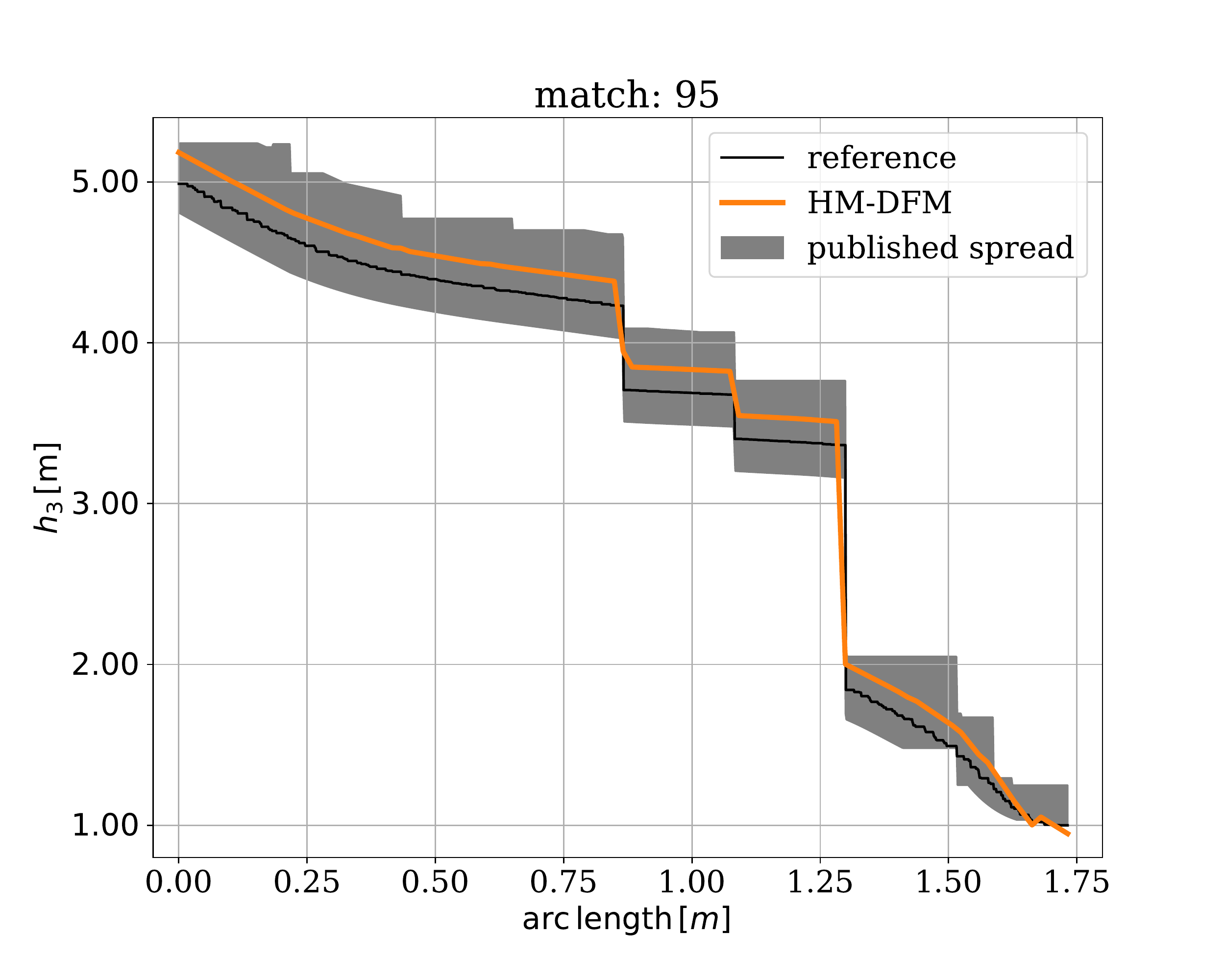}&
    \includegraphics[width=.48\textwidth]{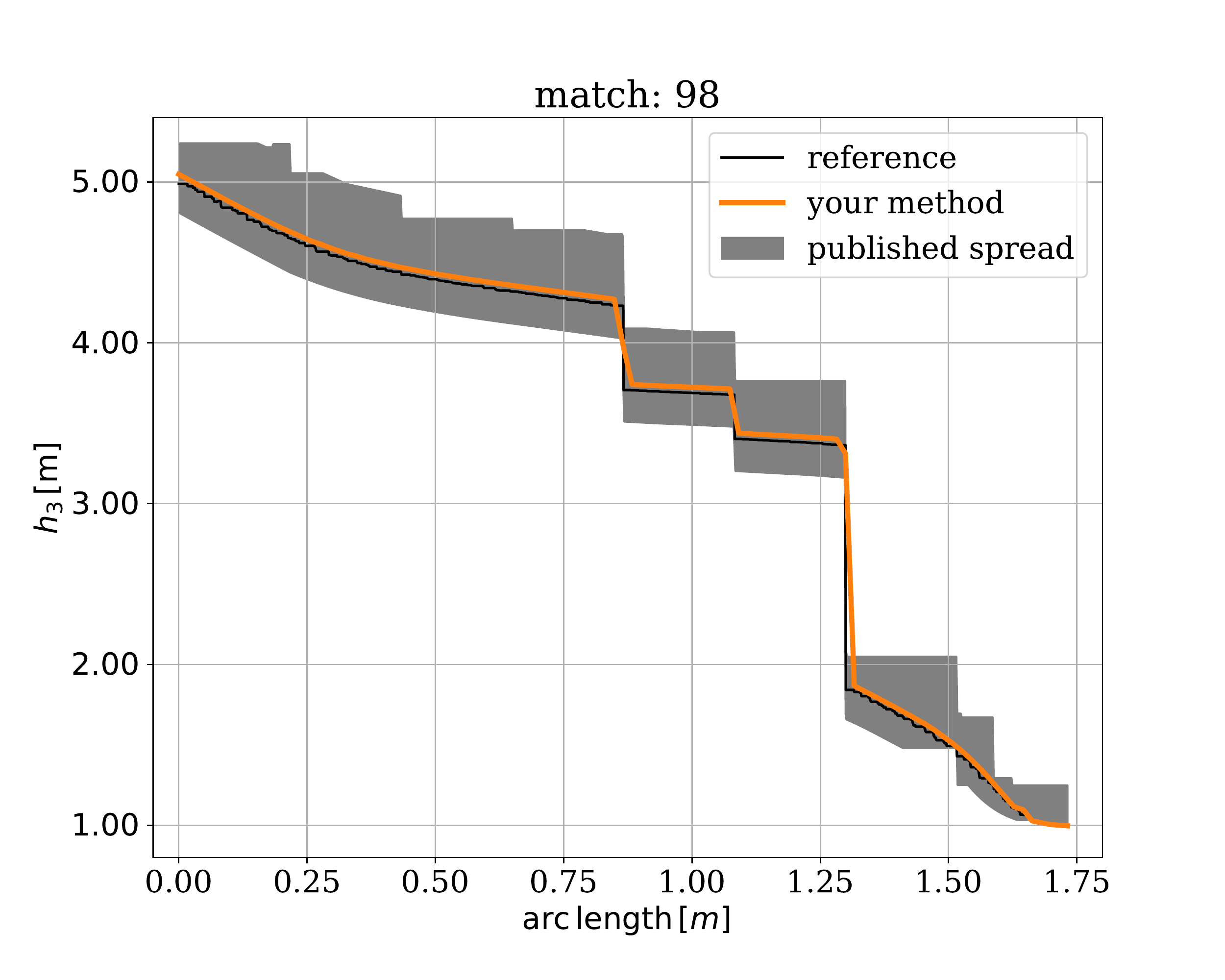}\\
    (a) $\sim 4k$ cells. &
    (b) $\sim 36k$ cells.
  \end{tabular}
  \caption{Benchmark 6 (blocking fractures): Hydraulic head in the matrix over the line
    $(0\mathrm{m}, 0\mathrm{m}, 0\mathrm{m})$--$(1\mathrm{m}, 1\mathrm{m}, 1\mathrm{m})$.
Left: results on a coarse mesh with about $4k$ cells.
Right: results on a fine mesh with about $36k$ cells.
  }
    \label{fig:regular3DB}
  \end{figure}

We further performed a convergence study of the flow and transport solvers \eqref{fem} and \eqref{transport-eq} via mesh refinements,
and record the $L^2$-errors in matrix velocity and postprocessed pressure,
and the $L^2$-errors in matrix concentration at final time $t=0.25$
in Table~\ref{tab:regularX} for the conductive fracture case and
in Table~\ref{tab:regularY} for the blocking fracture case, where the initial mesh is the coarse one with $4,375$ tetrahedral elements.
A total of three uniform mesh refinements was performed, and the solution on the third level mesh was used as the reference solution to calculate the associated errors.
The time step size is taken to be  $\Delta t = 2^{-l}\times2.5\times 10^{-3} s$, where $l$ is the mesh refinement level.
On the finest mesh, there are about  $2.25$ million tetrahedral elements and $4.5$ million globally coupled DOFs.
From both tables, we observe convergence of our schemes, and in particular the convergence rate for the velocity is approaching first order, that for the postprocessed pressure is approaching second order, and for the concentration is about first order.
\begin{table}[ht!]
    \centering
    \begin{tabular}{c|cc|cc|cc}
    mesh ref. lvl.     & $L^2$-err in $\bld u_h$ & rate
    & $L^2$-err in $p_h^*$ & rate    & $L^2$-err in $c_h(T)$ & rate\\
\hline
0 & 1.789e-01 & --& 1.456e-01 & --& 1.496e-01 &--\\
1 & 1.120e-01 & 0.68& 5.886e-02 & 1.31& 9.645e-02 & 0.63\\
2 & 6.181e-02 & 0.86& 1.852e-02 & 1.67& 5.102e-02 & 0.92\\
\hline
    \end{tabular}
    \vspace{1ex}
    \caption{Benchmark 6 with conductive fractures (fitted mesh): history of convergence for the $L^2$-errors in $\bld u_h$, $p_h^*$,  and $c_h(T)$ along mesh refinements. Reference solution is obtained on
    the third level refined fitted mesh with about $2.25$ million matrix elements and time step size $\Delta t = 3.125\times 10^{-4}$.}
    \label{tab:regularX}
\end{table}

\begin{table}[ht!]
    \centering
    \begin{tabular}{c|cc|cc|cc}
    mesh ref. lvl.     & $L^2$-err in $\bld u_h$ & rate
    & $L^2$-err in $p_h^*$ & rate
        & $L^2$-err in $c_h(T)$ & rate\\
\hline
0 & 1.791e-01 & --& 1.533e-01 & --& 1.288e-01 &--\\
1 & 1.118e-01 & 0.68& 6.080e-02 & 1.33& 8.139e-02 & 0.66\\
2 & 6.172e-02 & 0.86& 1.891e-02 & 1.68& 3.939e-02 & 1.05\\
\hline
    \end{tabular}
    \vspace{1ex}
    \caption{Benchmark 6 with blocking fractures (fitted mesh): history of convergence for the $L^2$-errors in $\bld u_h$, $p_h^*$,  and $c_h(T)$ along mesh refinements. Reference solution is obtained on
    the third level refined fitted mesh with about $2.25$ million matrix elements and time step size $\Delta t = 3.125\times 10^{-4}$.}
    \label{tab:regularY}
\end{table}

Finally, in  Figure~\ref{fig:regularY} we plot slices of concentrations
computed on the 3rd refined mesh  at final time
$t=0.25$ along the five vertical planes $x=0.1, x=0.3, x=0.5, x=0.7$ and $x=0.9$, and  in
Figure~\ref{fig:regularX}
we plot the evolution of mean concentration over time on the following three regions:
\begin{align*}
    \Omega_A&:= (0.5m, 1m)\times (0m, 0.5m)\times (0m,0.5m), \\
    \Omega_B&:= (0.5m, 0.75m)\times (0.5m, 0.75m)\times (0.75m,1m), \\
    \Omega_C&:= (0.75m, 1m)\times (0.75m, 1m)\times (0.5m,0.75m).
\end{align*}
\begin{figure}[ht]
  \centering
  \subfigure[C. $c_h$ on $x=0.1$]{\includegraphics[width = 1.2in]{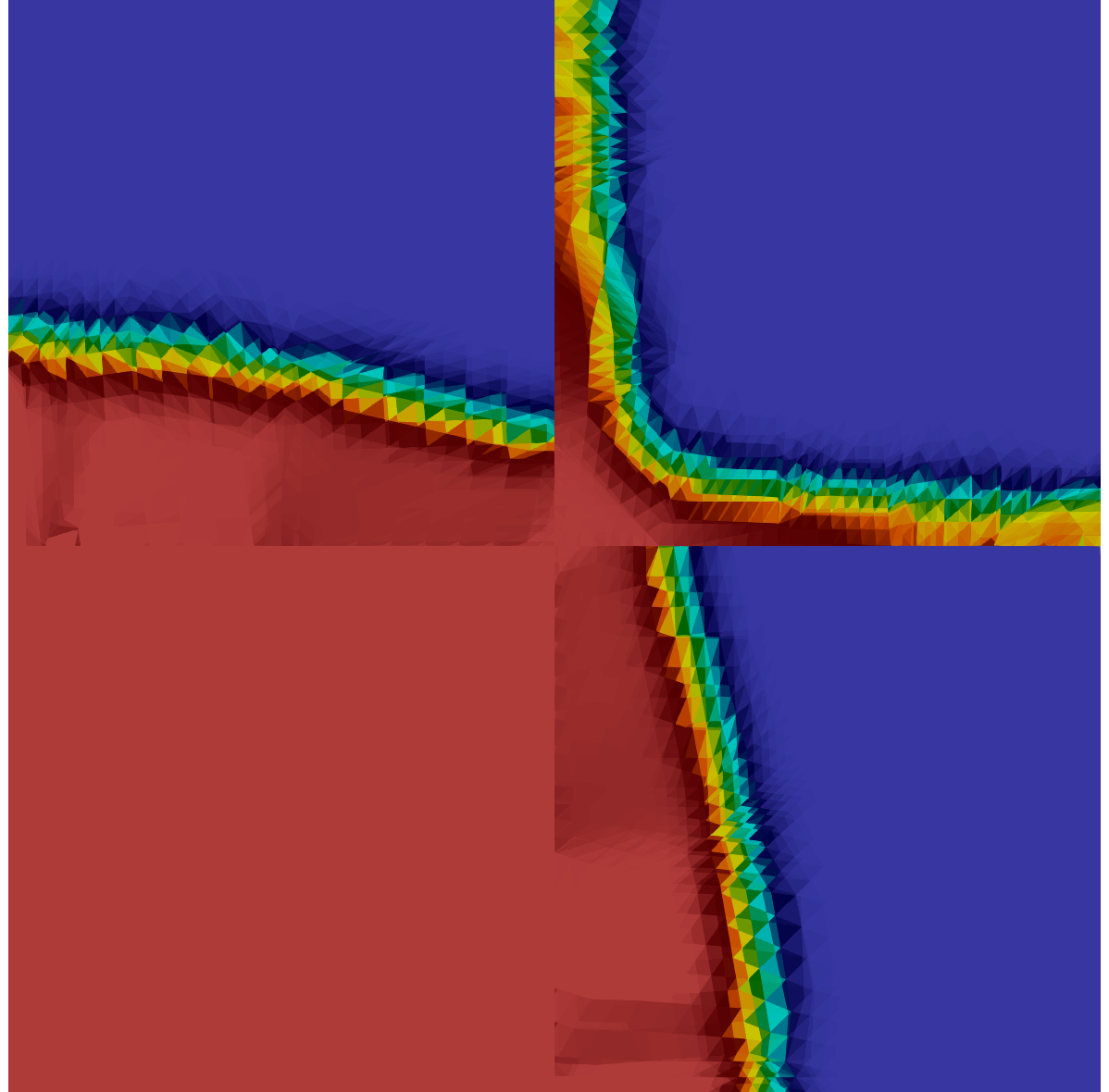}}
  \subfigure[C. $c_h$ on $x=0.3$]{\includegraphics[width = 1.2in]{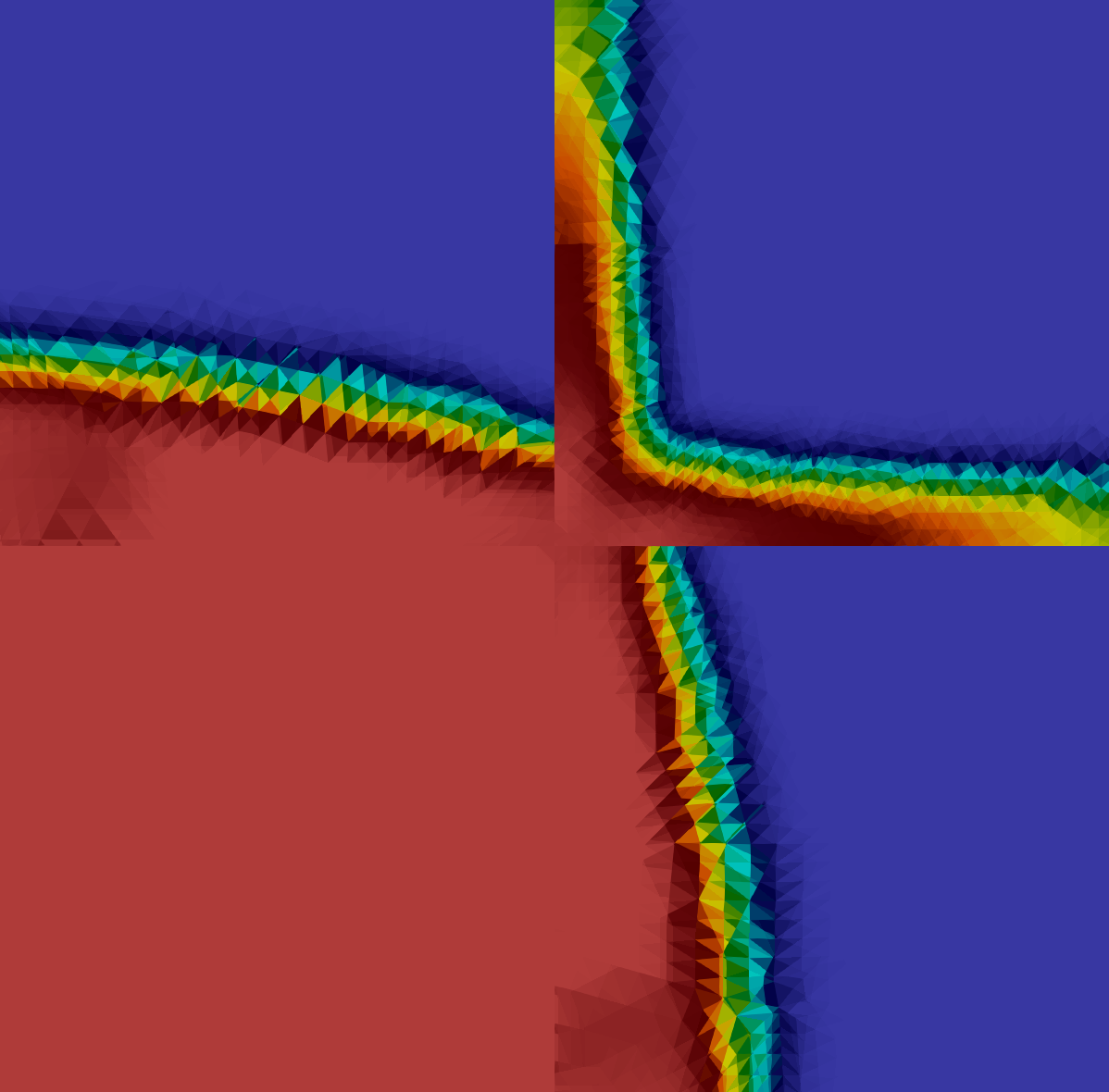}}
\subfigure[C. $c_h$ on $x=0.5$]{\includegraphics[width = 1.2in]{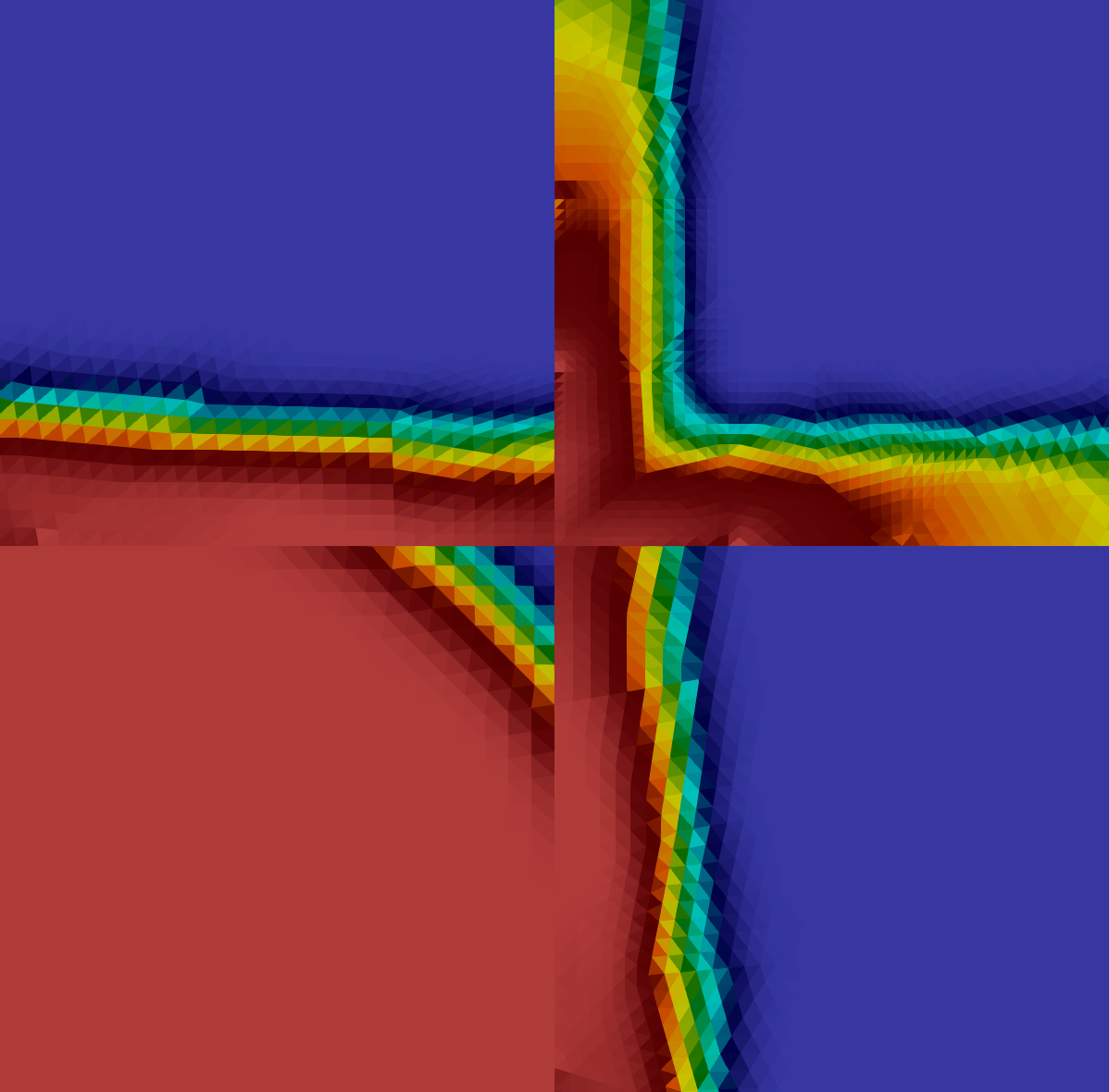}}
\subfigure[C. $c_h$ on $x=0.7$]{\includegraphics[width = 1.2in]{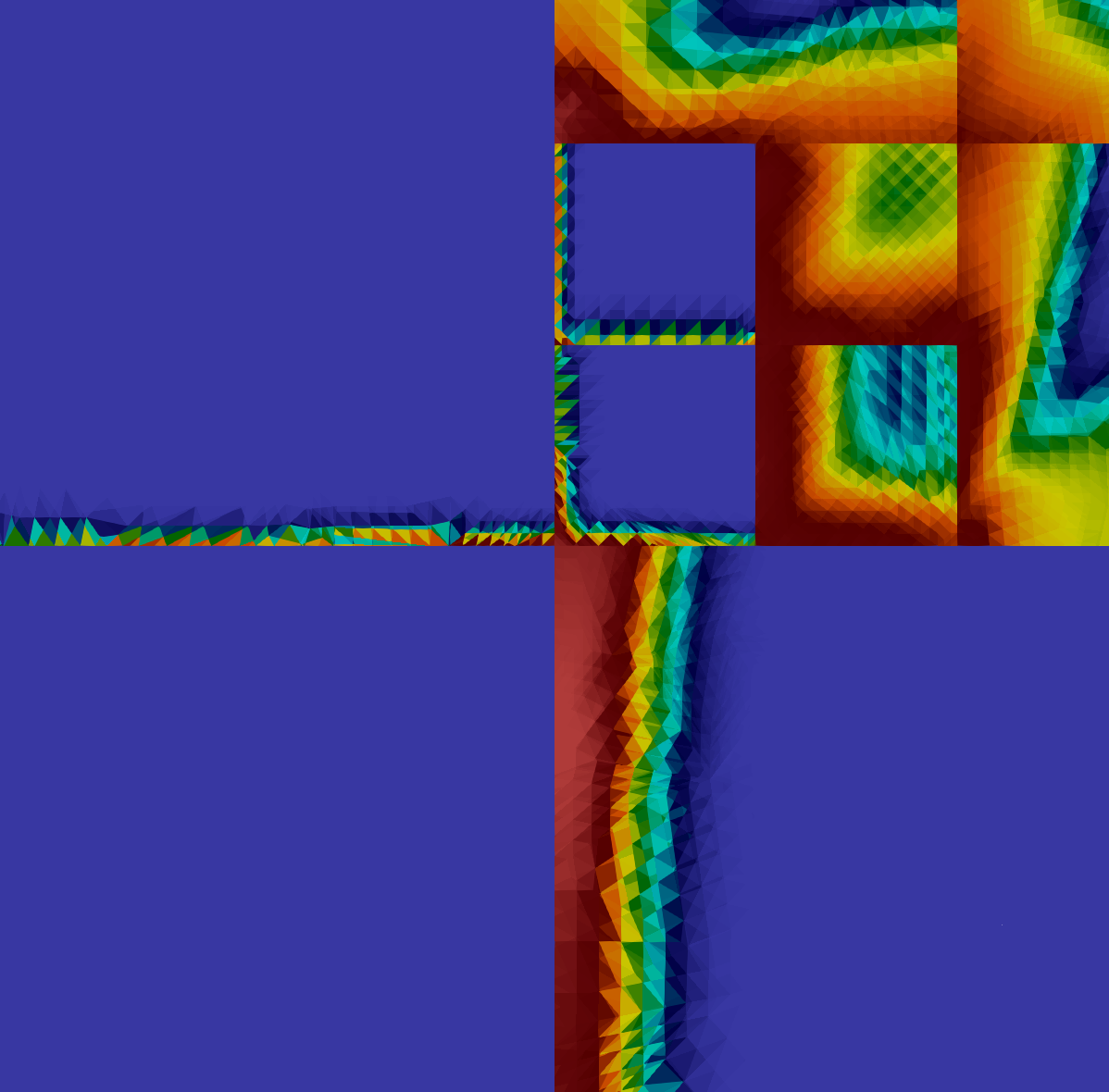}}
\subfigure[C. $c_h$ on $x=0.9$]{\includegraphics[width = 1.2in]{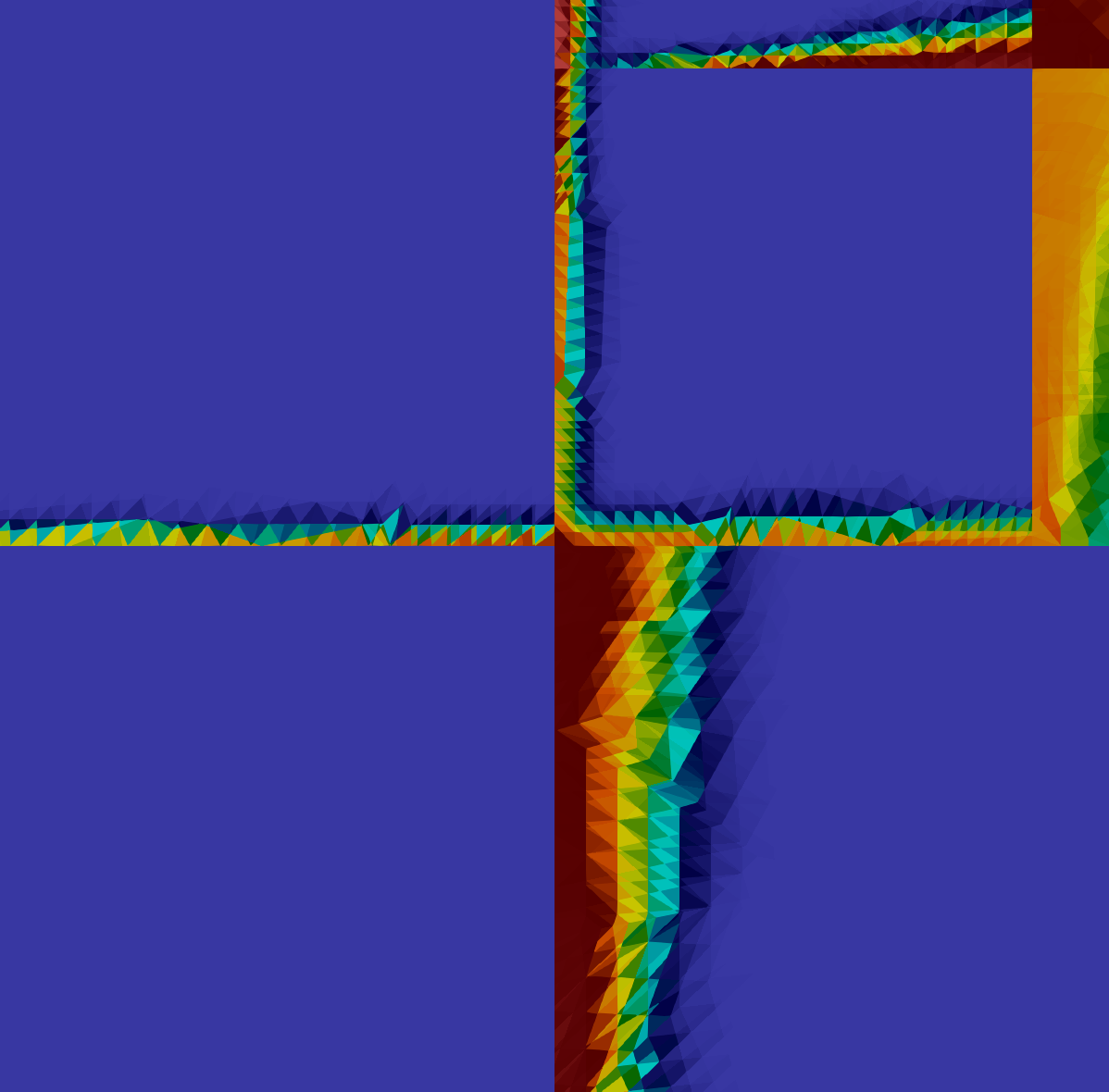}}
  \subfigure[B. $c_h$ on $x=0.1$]{\includegraphics[width = 1.2in]{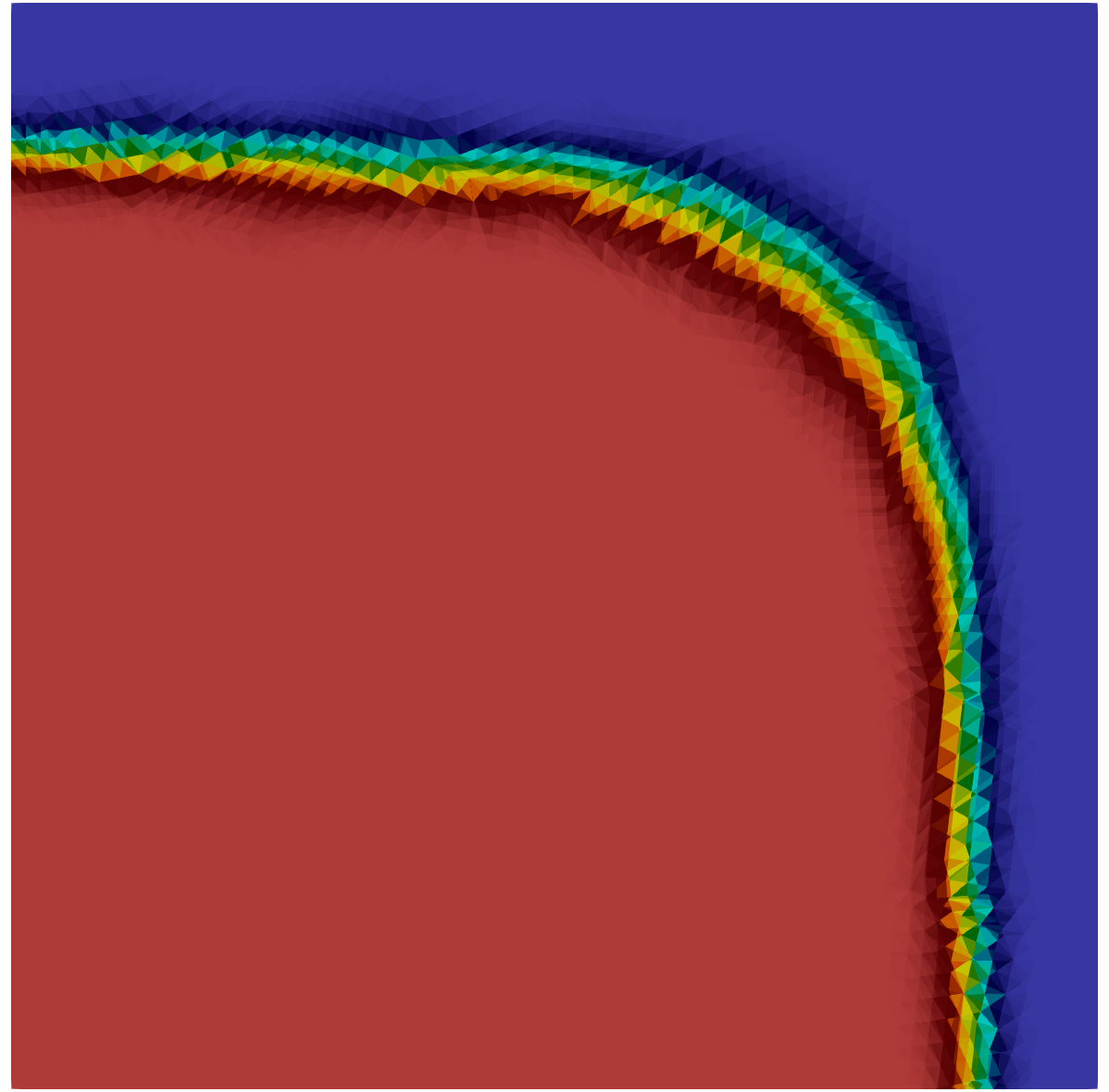}}
  \subfigure[B. $c_h$ on $x=0.3$]{\includegraphics[width = 1.2in]{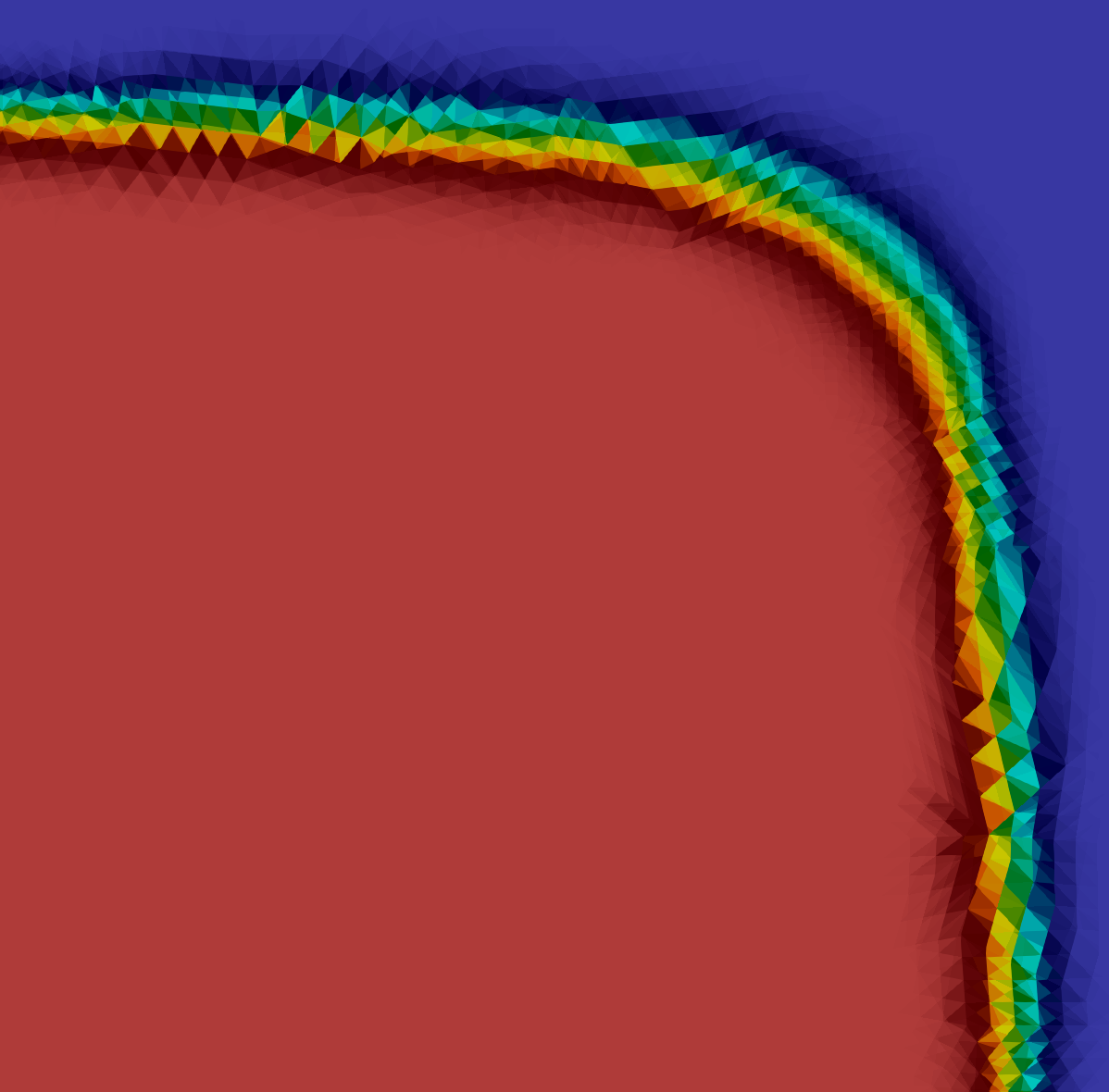}}
\subfigure[B. $c_h$ on $x=0.5$]{\includegraphics[width = 1.2in]{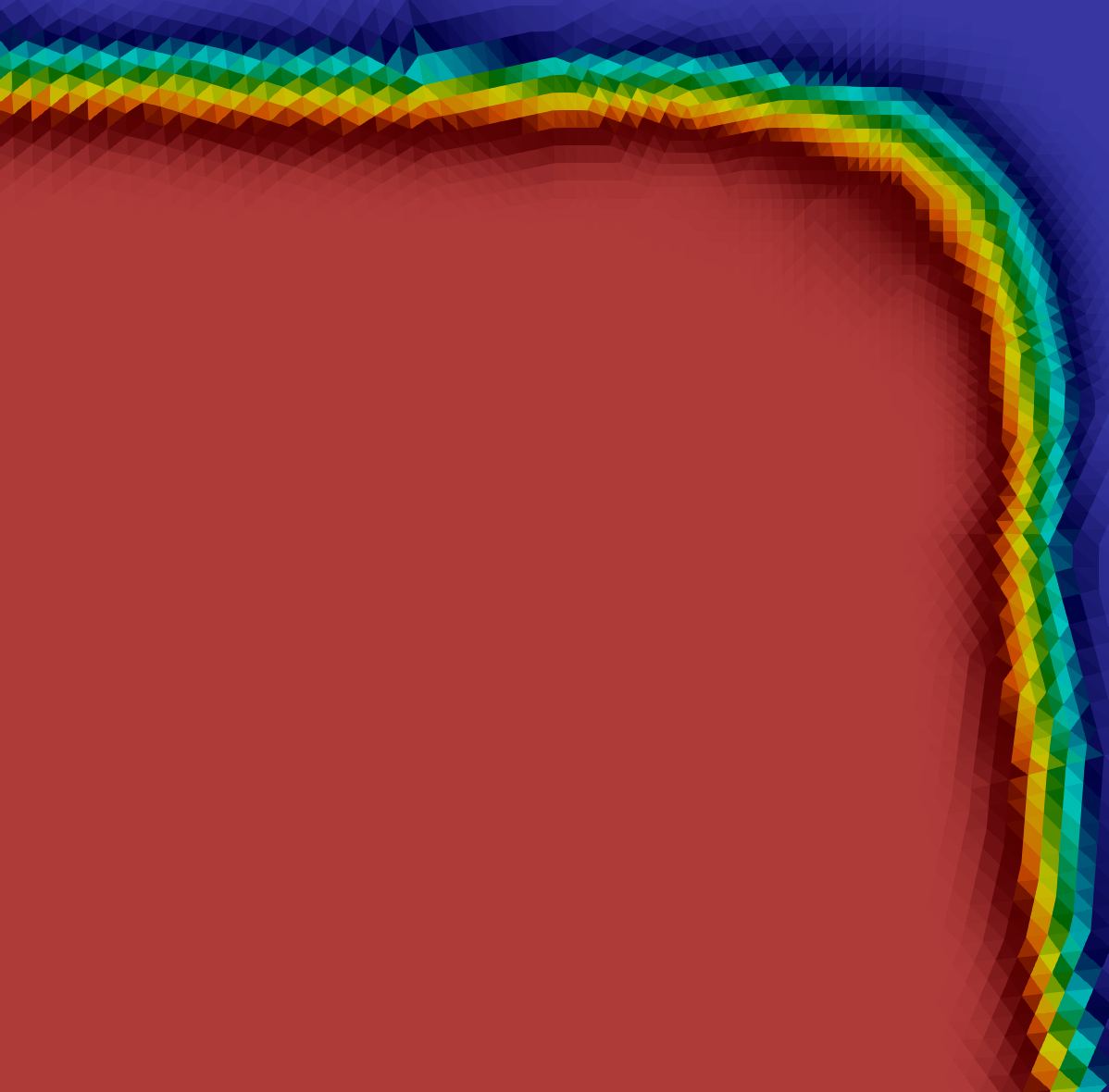}}
\subfigure[B. $c_h$ on $x=0.7$]{\includegraphics[width = 1.2in]{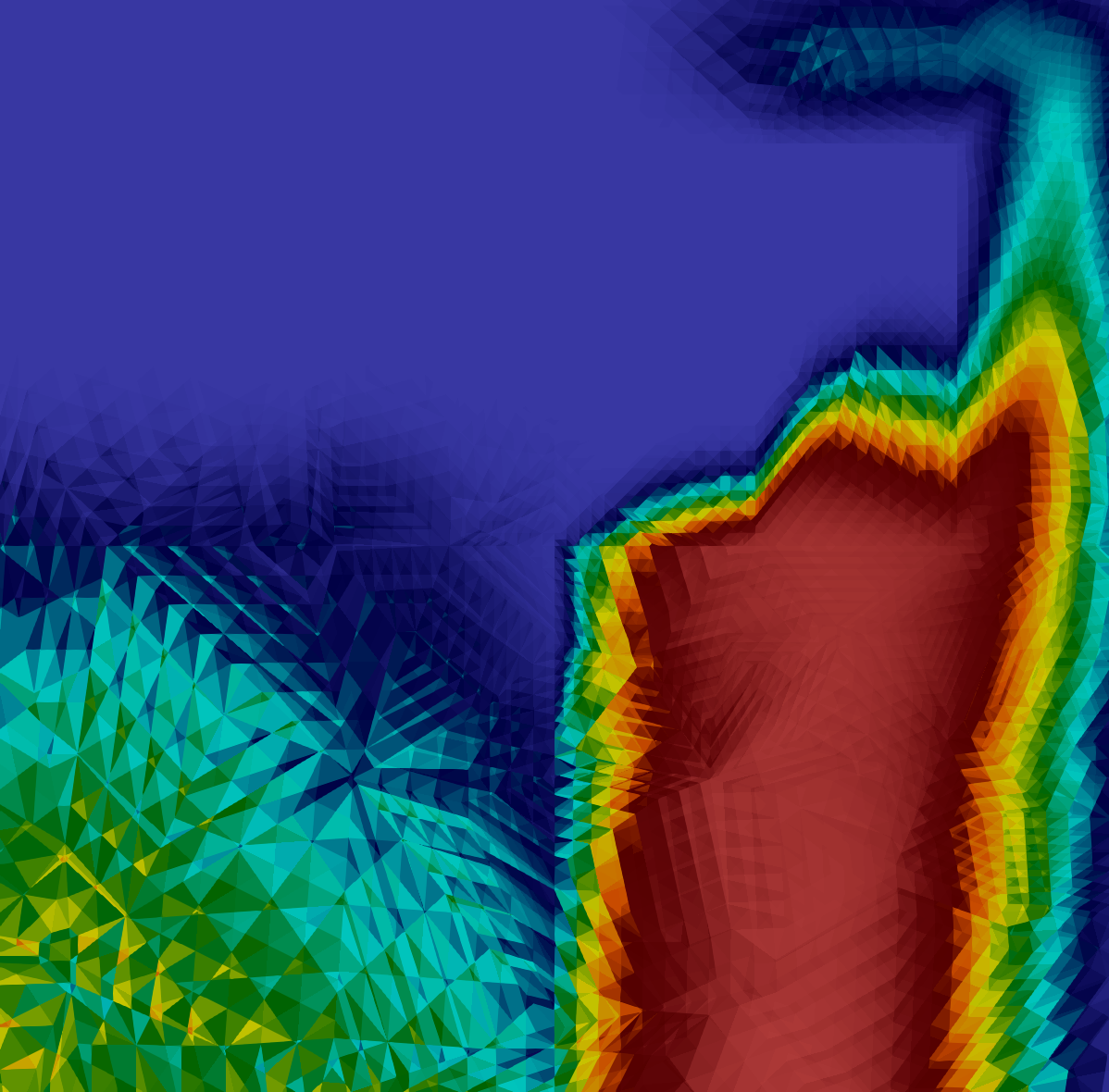}}
\subfigure[B. $c_h$ on $x=0.9$]{\includegraphics[width = 1.2in]{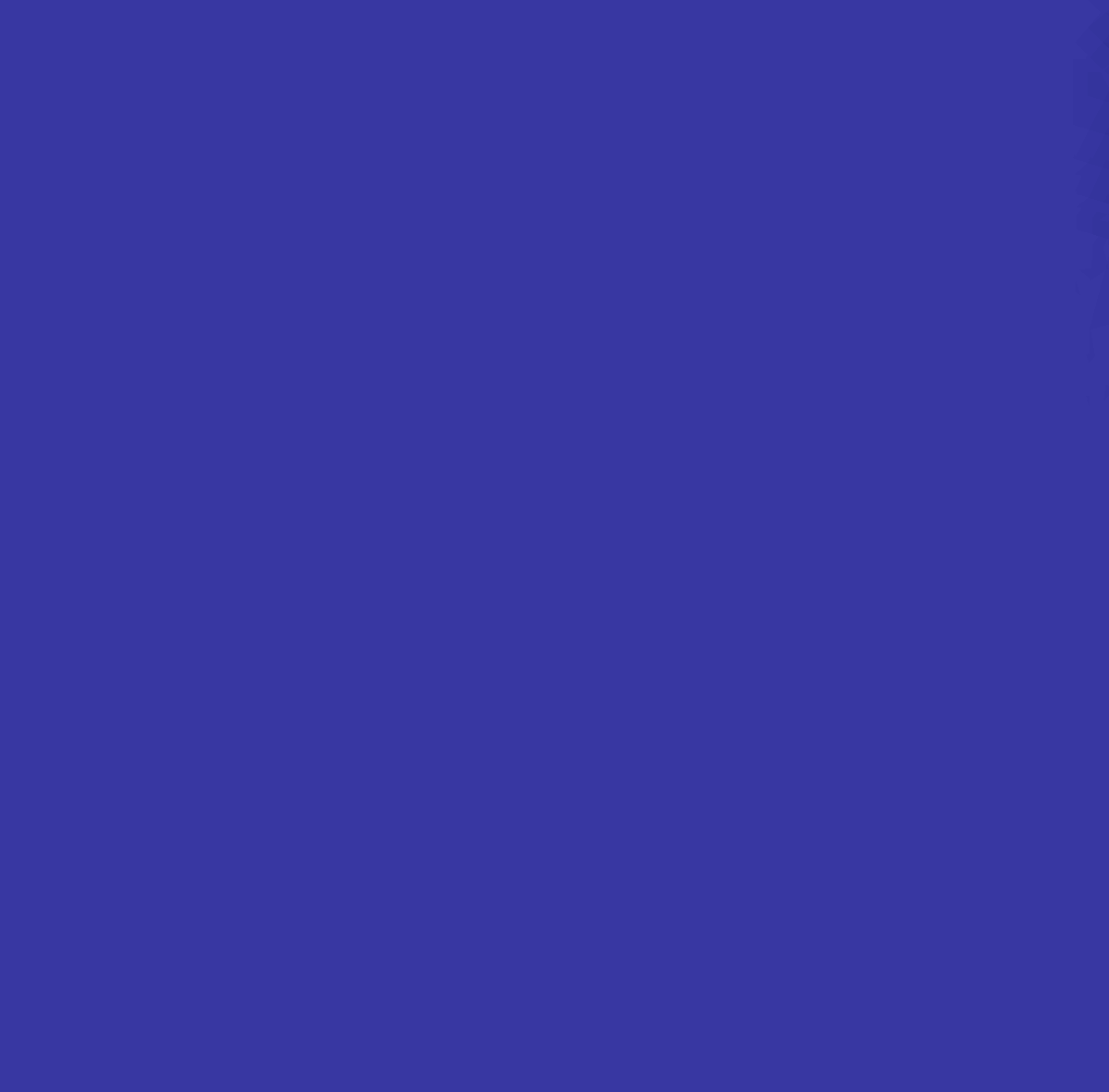}}
  \caption{Benchmark 6: Matrix concentration at time $t=0.25$
  along the five vertical planes $x=0.1$, $x=0.3, x=0.5, x=0.7$ and $x=0.9$.
Top row: conductive fractures. Bottom row: blocking fractures.
Color range: 0 (blue)-- 1(red).
  }
  \label{fig:regularY}
\end{figure}
\begin{align*}
    \Omega_A&:= (0.5m, 1m)\times (0m, 0.5m)\times (0m,0.5m), \\
    \Omega_B&:= (0.5m, 0.75m)\times (0.5m, 0.75m)\times (0.75m,1m), \\
    \Omega_C&:= (0.75m, 1m)\times (0.75m, 1m)\times (0.5m,0.75m).
\end{align*}
\begin{figure}[ht]
  \centering
  \subfigure[Conductive,  mean $c_h$ on $\Omega_A$]{\includegraphics[width = 2.0in]{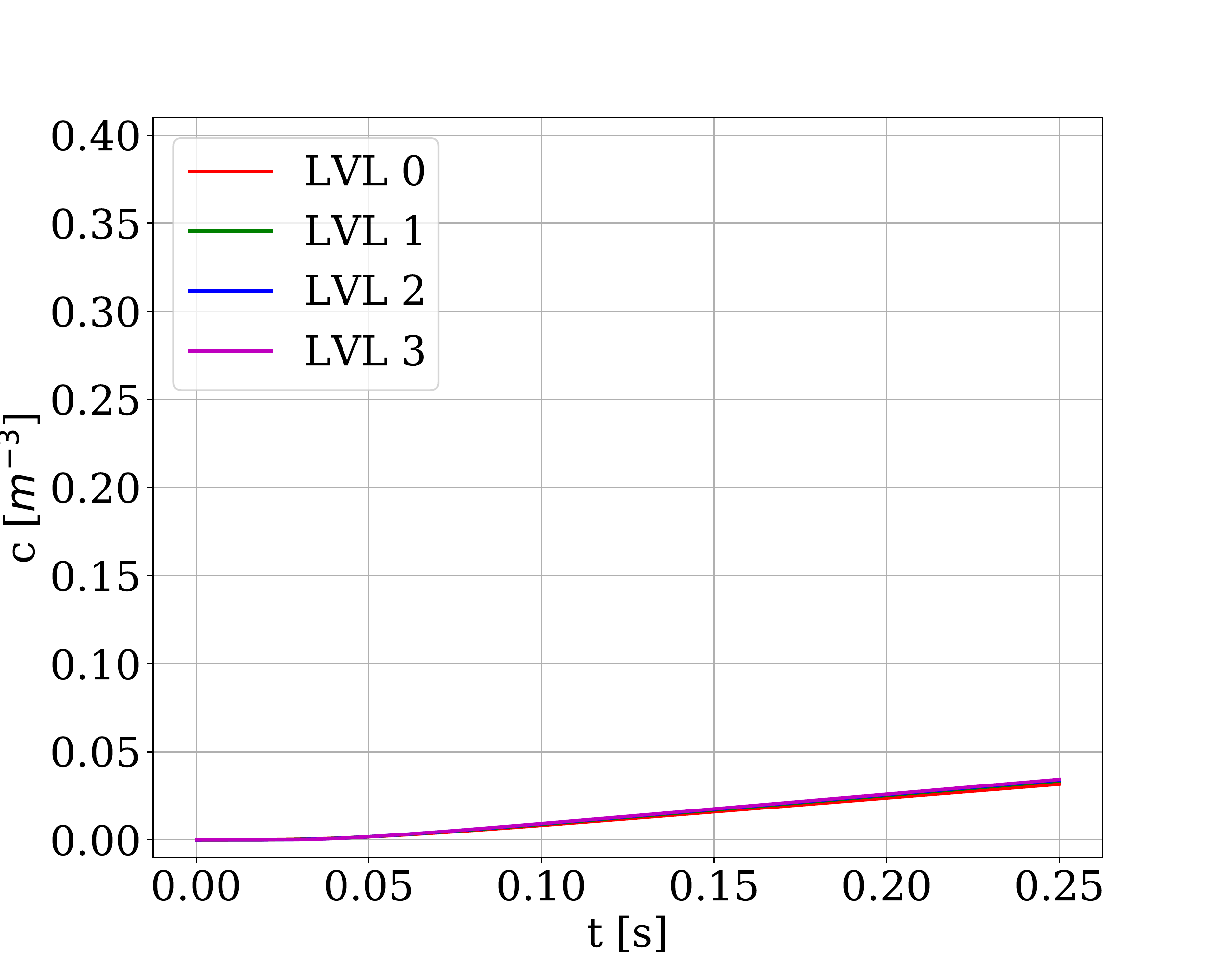}}
    \subfigure[Conductive,  mean $c_h$ on $\Omega_B$]{\includegraphics[width = 2.0in]{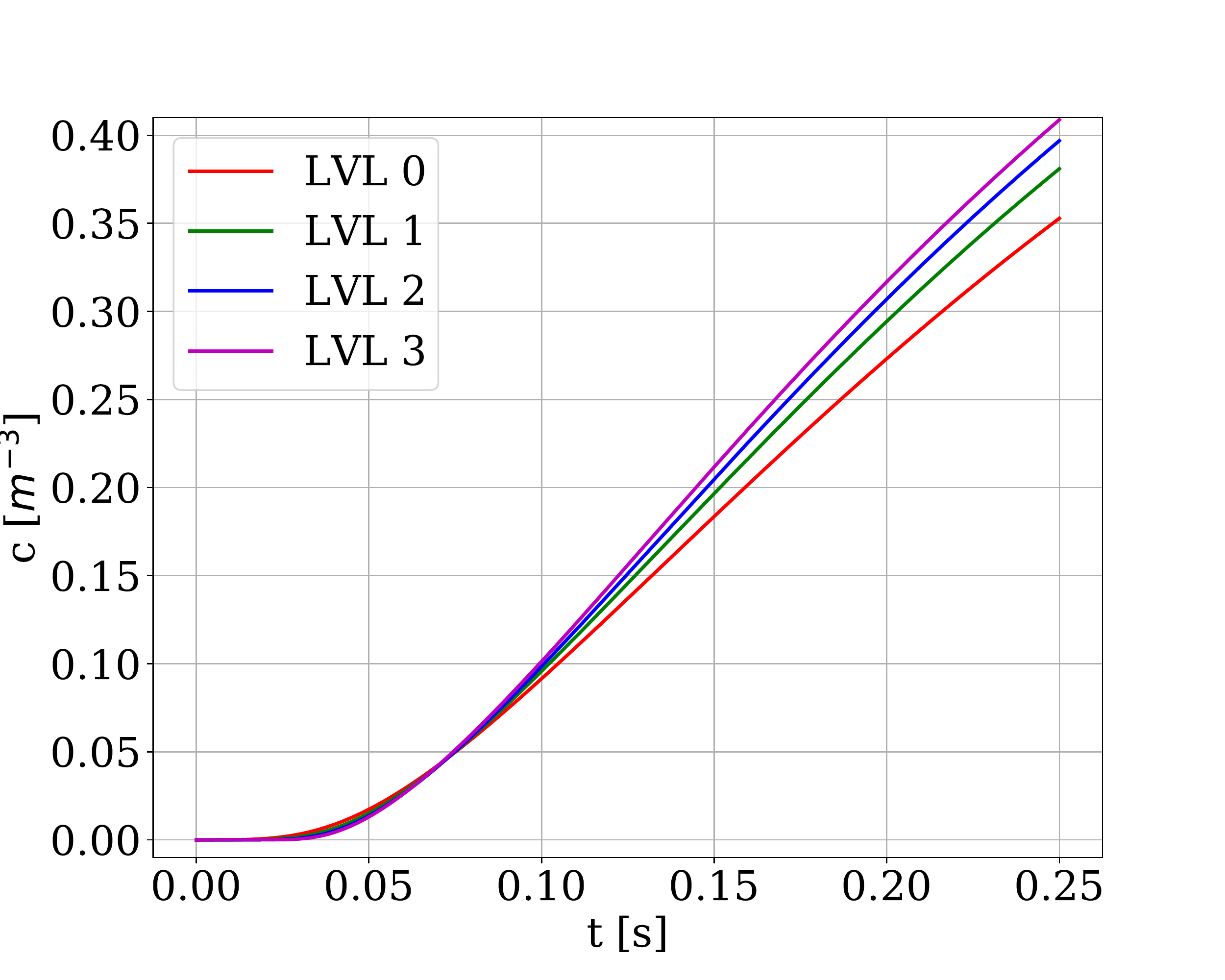}}
    \subfigure[Conductive,  mean $c_h$ on $\Omega_C$]{\includegraphics[width = 2.0in]{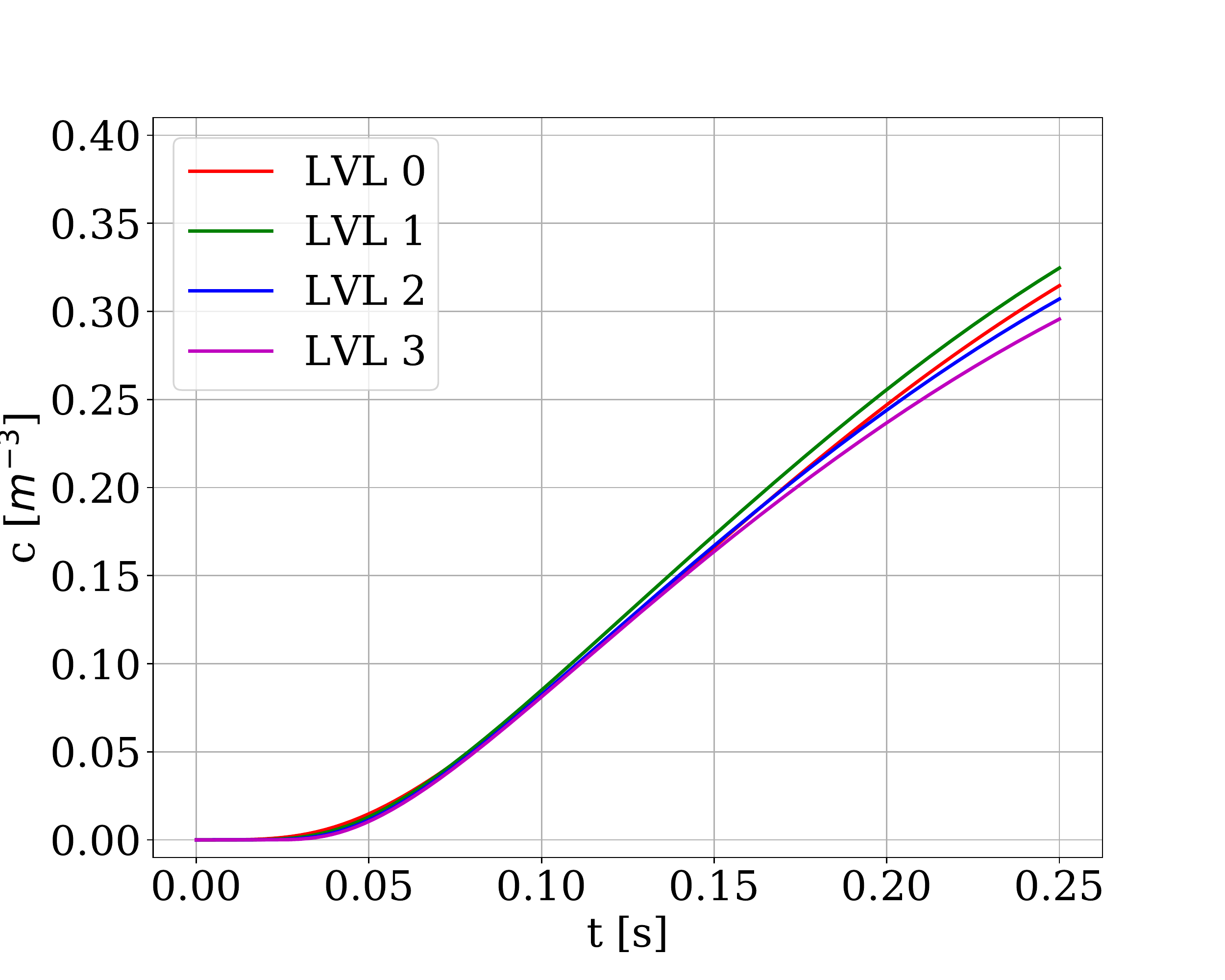}}
    \subfigure[Blocking,  mean $c_h$ on $\Omega_A$]{\includegraphics[width = 2.0in]{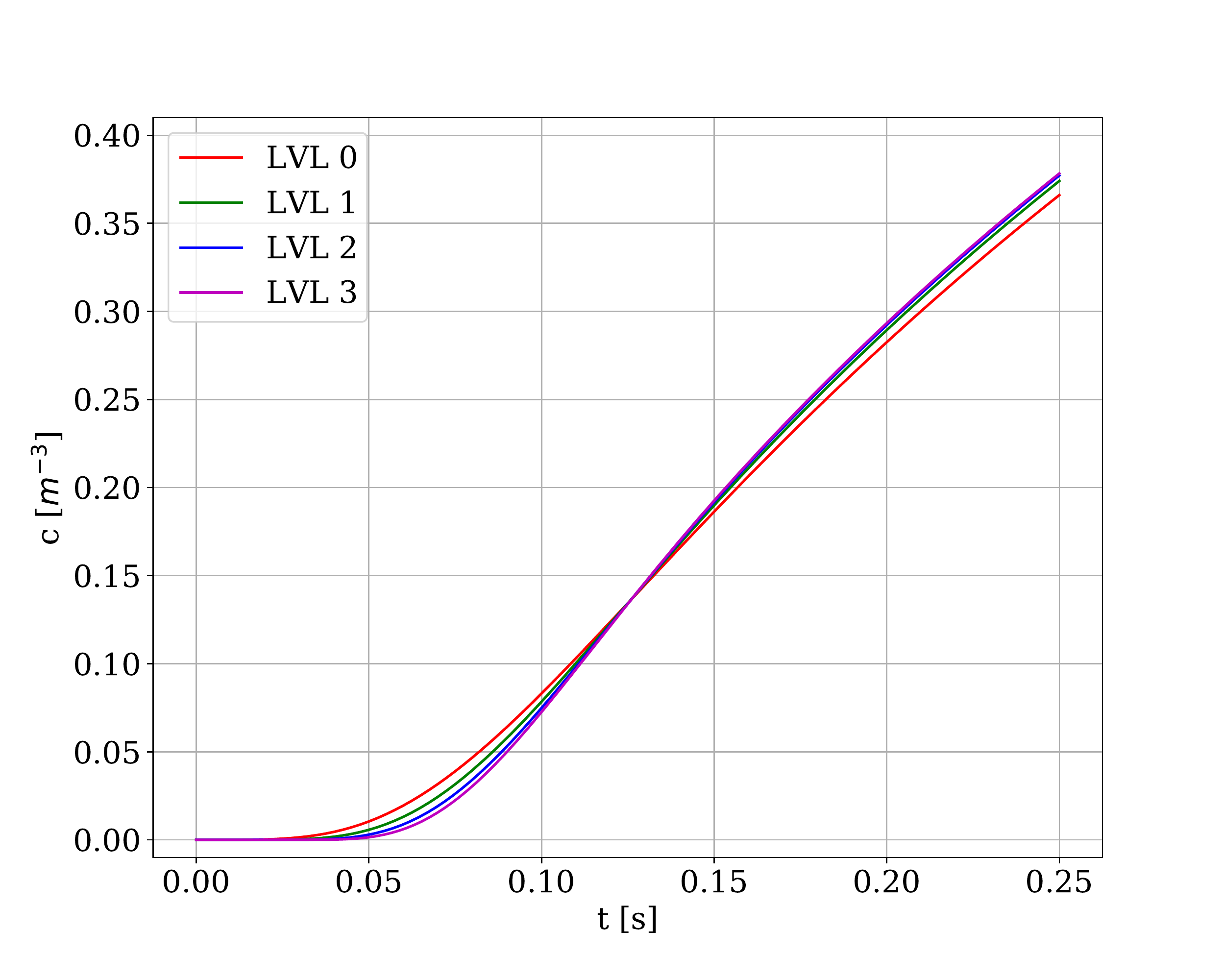}}
    \subfigure[Blocking,  mean $c_h$ on $\Omega_B$]{\includegraphics[width = 2.0in]{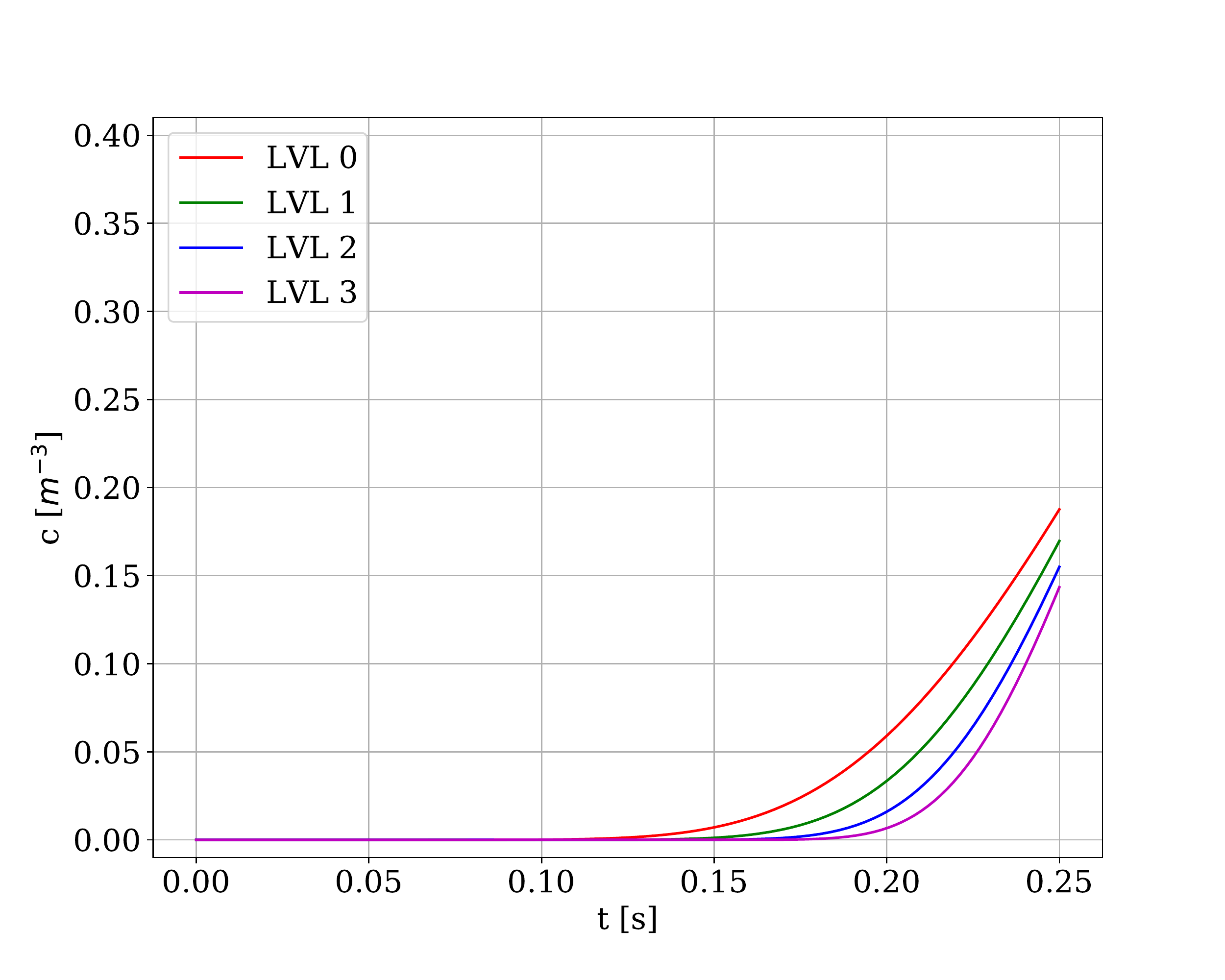}}
\subfigure[Blocking,  mean $c_h$ on $\Omega_C$]{\includegraphics[width = 2.0in]{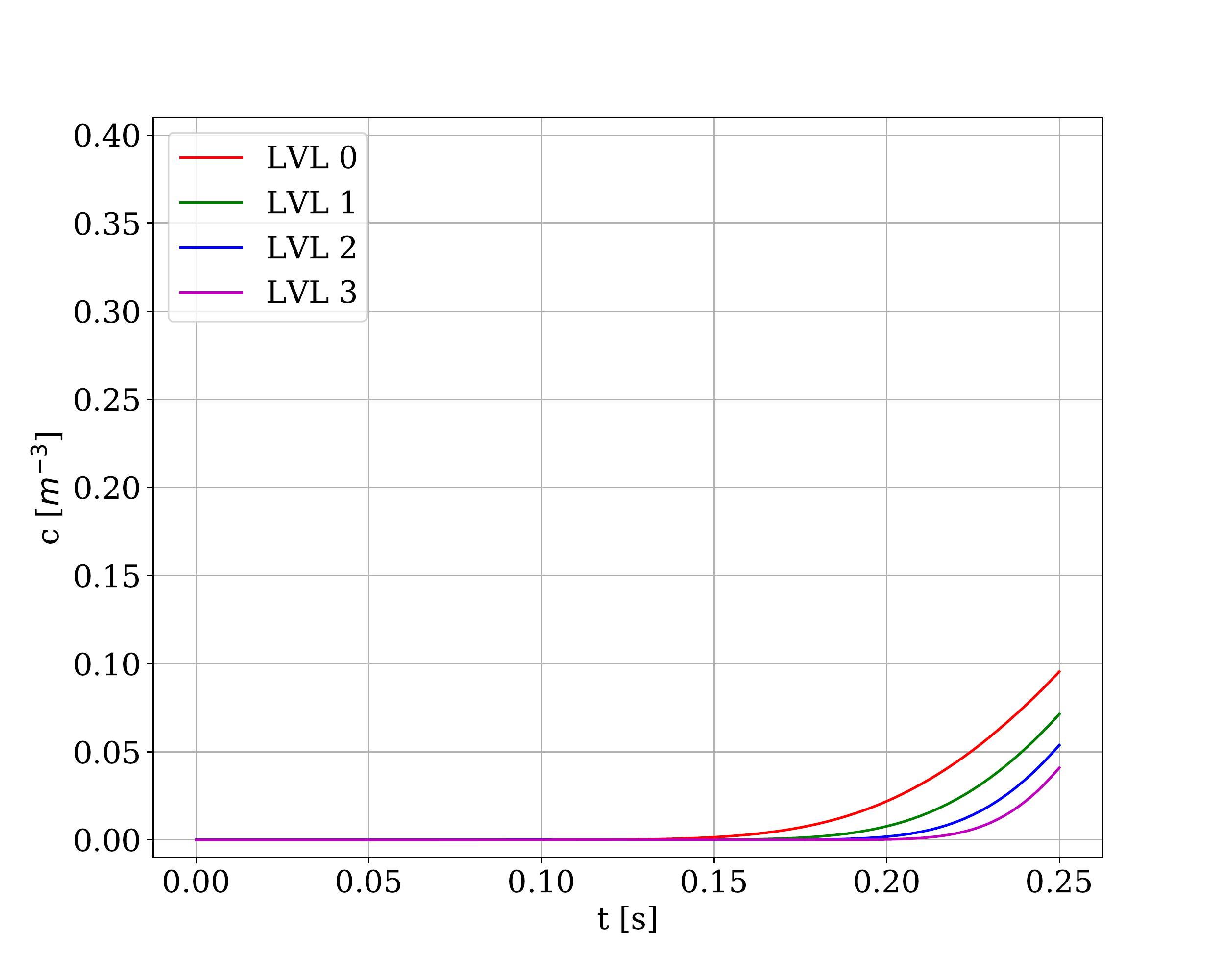}}
  \caption{Benchmark 6: Mean matrix concentration over time
  on $\Omega_A$ (left),  $\Omega_B$ (middle), and $\Omega_C$ (right).
Top row: conductive fractures. Bottom row: blocking fractures.
LVL stands for the number of mesh refinement levels.
  }
  \label{fig:regularX}
\end{figure}
From the results in Figure~\ref{fig:regularY}, we clearly observe the different flow pattern for the conductive fracture case in the first row and the
blocking fracture case in the second row.
We further note that the mean concentrations reported in Figure~\ref{fig:regularX} were presented in \cite[Figure 10]{Berre_2021} (only) on the coarse mesh with about $4k$ matrix elements and a coarse time step size $\Delta t = 2.5\times 10^{-3}s$.
Our results on four set of meshes are close to each other and improve slightly as the mesh and time step size refines, and they  are also qualitatively similar to the majority of the coarse-grid results in \cite[Figure 10]{Berre_2021}.

\subsection{Benchmark 7: Network with Small Features (3D)}
This is the third benchmark case proposed in \cite{Berre_2021}, in which small geometric features exist that may cause trouble for conforming meshing
strategies. The domain is the box
$\Omega = (0\mathrm{m}, 1\mathrm{m})\times (0\mathrm{m}, 2.25\mathrm{m})\times
(0\mathrm{m}, 1\mathrm{m})$, containing 8 fractures; see Figure \ref{fig:small}.
  \begin{figure}[ht]
  \centering
    \includegraphics[width=.7\textwidth]{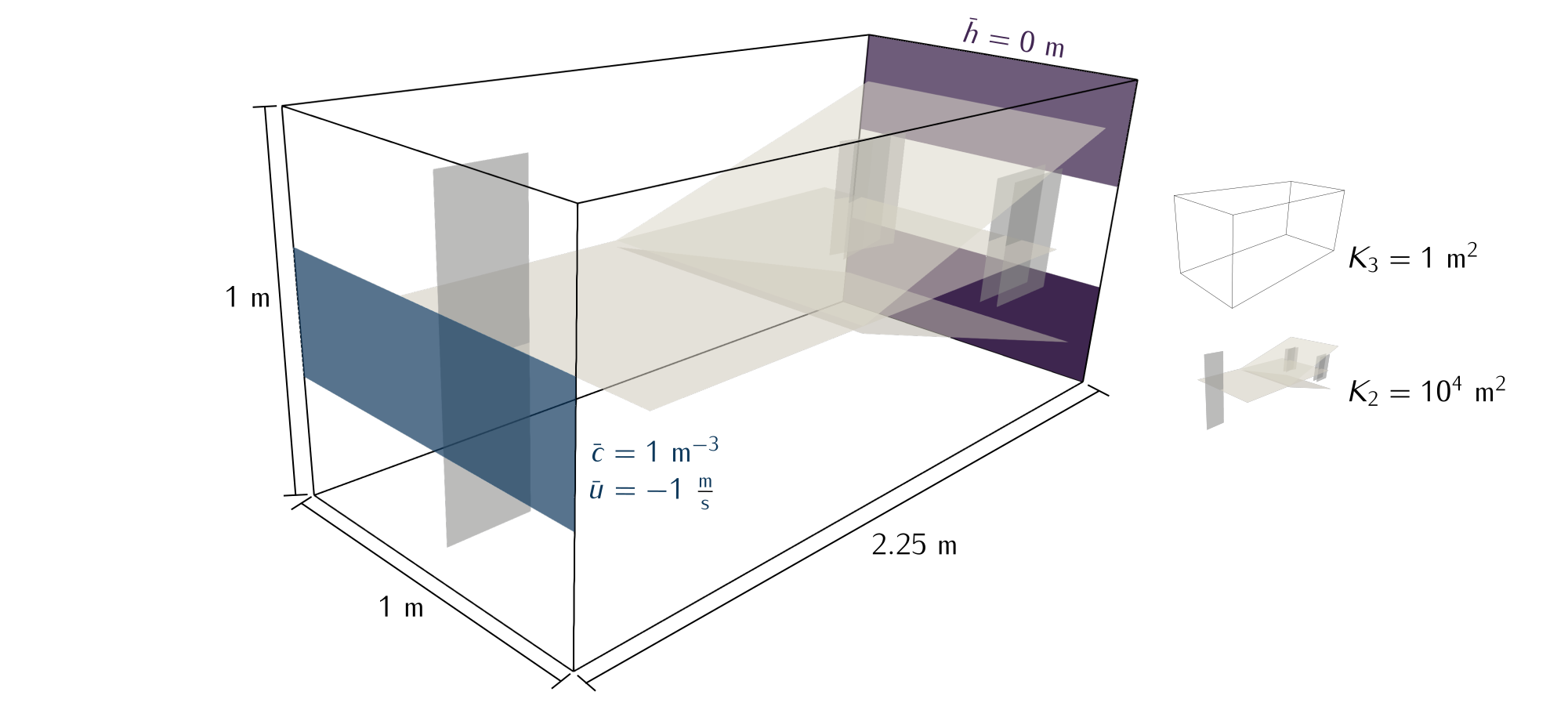}
    \caption{Benchmark 7: Conceptual model and geometrical description
    of the domain.}
    \label{fig:small}
  \end{figure}
Homogeneous Dirichlet boundary condition is imposed on
the outlet boundary
\[
  \partial \Omega_{out}:=\{
    (x,y,z): 0<x<1, y = 2.25, z < 1/3 or z > 2/3\},
\]
inflow boundary condition $\bld u\cdot\bld n = -1\mathrm{m/s}$ is imposed
on the inlet boundary
\[
  \partial \Omega_{in}:=\{
    (x,y,z): 0<x<1, y = 0, 1/3<z < 2/3\},
\]
and no-flow boundary condition is imposed on the remaining boundaries.
The conductivity in the matrix is $\mathbb K_m=1\mathrm{m^2}$, and that in the
fracture is $\mathbb{K}_c=10^4\mathrm{m^2}$. Fracture thickness is
$\epsilon=0.01\mathrm{m}$.

We perform the method \eqref{fem} on a coarse tetrahedral mesh with
  $31,812$ matrix elements and $3,961$ fracture elements and a
  fine tetrahedral mesh with $147,702$ matrix elements and
  $9,441$ fracture elements.
  The number of the globally coupled DOFs on the coarse mesh is $83,022$,
  while that on the fine mesh is $343,359$.
  The hydraulic head along the line $(0.5\mathrm{m}, 1.1\mathrm{m},
0\mathrm{m})$--$(0.5\mathrm{m}, 1.1\mathrm{m}, 1\mathrm{m})$
is shown in Figure \ref{fig:smallC}, where the reference data is obtained with
the {\sf USTUTT-MPFA} scheme on a grid with approximately $10^6$ matrix cells.
Here we observe a very good agreement with the reference data even on the coarse
mesh.
\begin{figure}[ht]
  \centering
  \begin{tabular}{cc}
    \includegraphics[width=0.48\textwidth]{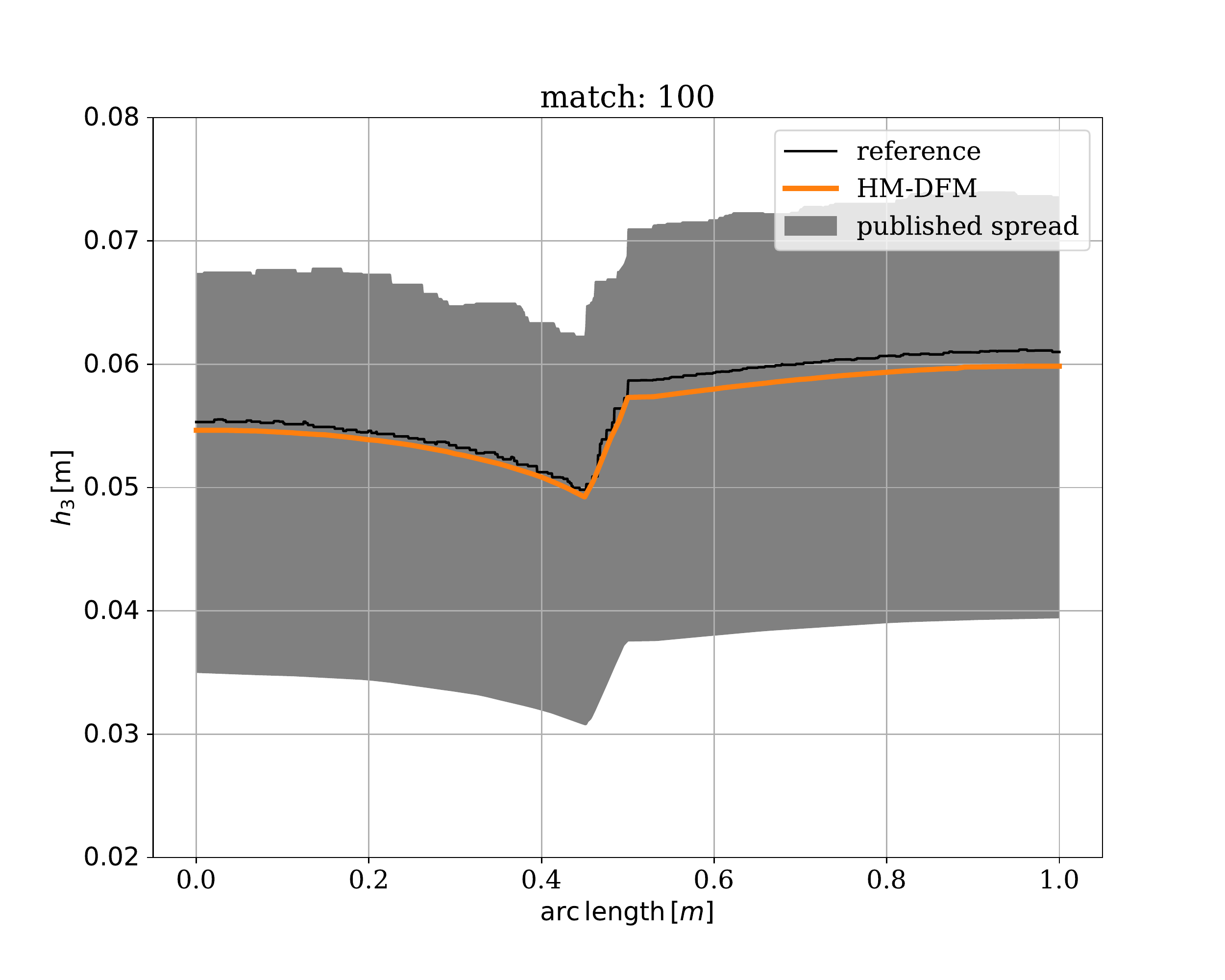}&
    \includegraphics[width=.48\textwidth]{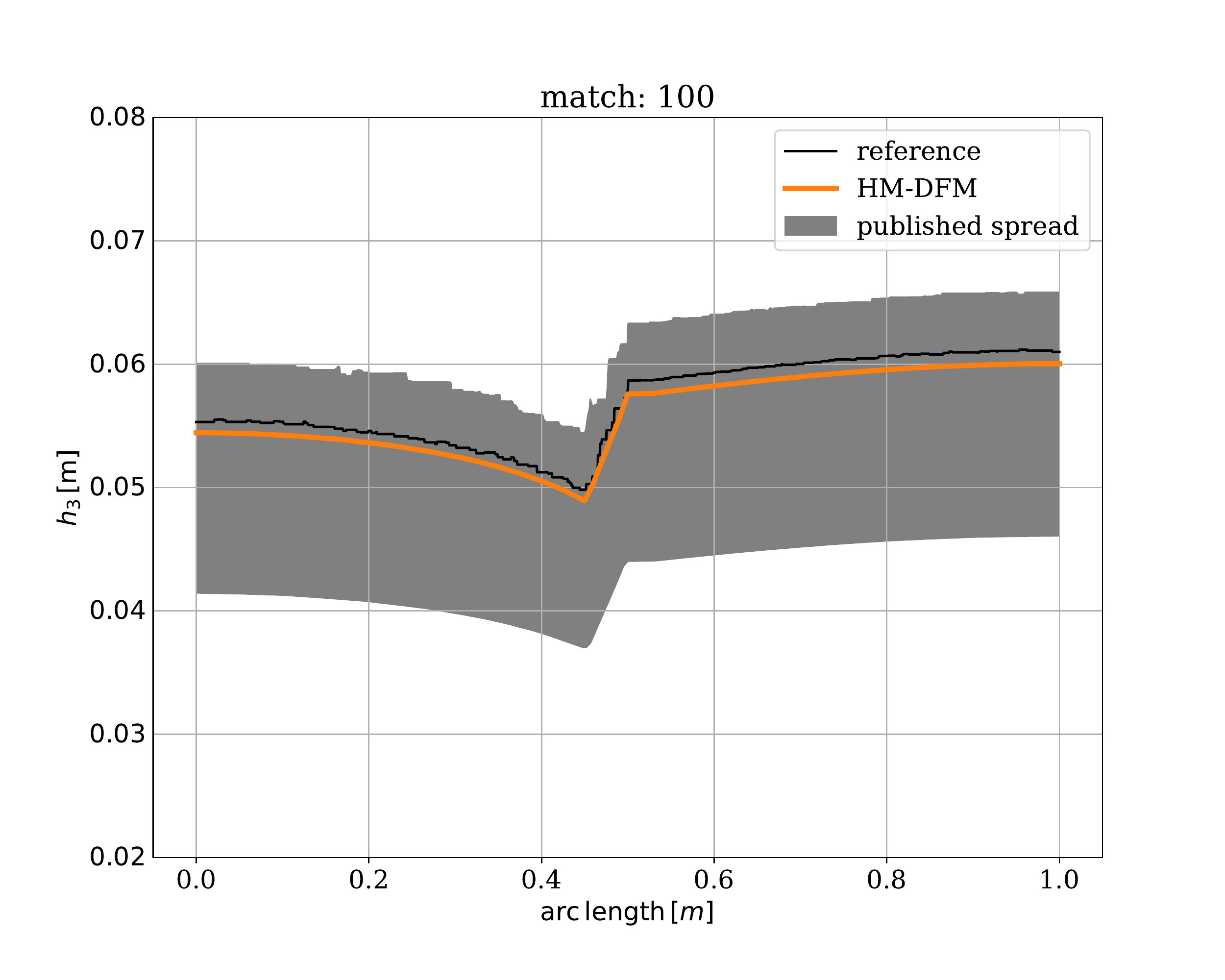}\\
    (a) $\sim 32k$ cells. &
    (b) $\sim 148k$ cells.
  \end{tabular}
    \caption{Benchmark 7: Hydraulic head in the matrix over the line
    $(0.5\mathrm{m}, 1.1\mathrm{m}, 0\mathrm{m})$--$(0.5\mathrm{m}, 1.1\mathrm{m}, 1\mathrm{m})$.
Left: results on a coarse mesh with about $32k$ cells.
Right: results on a fine mesh with about $148k$ cells.
  }
    \label{fig:smallC}
  \end{figure}

  \subsection{Benchmark 8: Field Case (3D)}
This is the last benchmark case proposed in \cite{Berre_2021}.
The geometry is based on  a postprocessed outcrop from the island of Algerøyna,
outside Bergen, Norway, which contains 52 fracture.
The simulation domain is the box $\Omega = (-500\mathrm{m}, 350\mathrm{m})\times
(100\mathrm{m}, 1500\mathrm{m})\times (-100\mathrm{m}, 500\mathrm{m})$.
The fracture geometry is depicted in Figure \ref{fig:field}.
Homogeneous Dirichlet boundary condition is imposed on
the outlet boundary
\[
  \partial \Omega_{out}:=
  \underbrace{
  \{-500\}\times(100, 400)\times (-100, 100)}_{\partial\Omega_{out,0}}
\;\cup\;
\underbrace{
  \{350\}\times(100, 400)\times (-100, 100)}_{\partial\Omega_{out,1}}
\]
uniform unit inflow $\bld u\cdot \bld n = 1\mathrm{m/s}$
is imposed on the inlet boundary
\[
  \partial \Omega_{in}:=
  \underbrace{\{-500\}\times(1200, 1500)\times (300, 500)}_{\partial\Omega_{in,0}}
\;\cup\;
\underbrace{
(-500, -200)\times \{1500\}\times (300, 500)}_{\partial\Omega_{in, 1}}.
\]
Conductivity is $\mathbb K_m = 1\mathrm{m^2}$ in the matrix, and
$\mathbb K_c = 10^4 \mathrm{m^2}$ in the fracture.
Fracture thickness is $\epsilon = 10^{-2}\mathrm{m}$.
  \begin{figure}[ht]
  \centering
    \includegraphics[width=.7\textwidth]{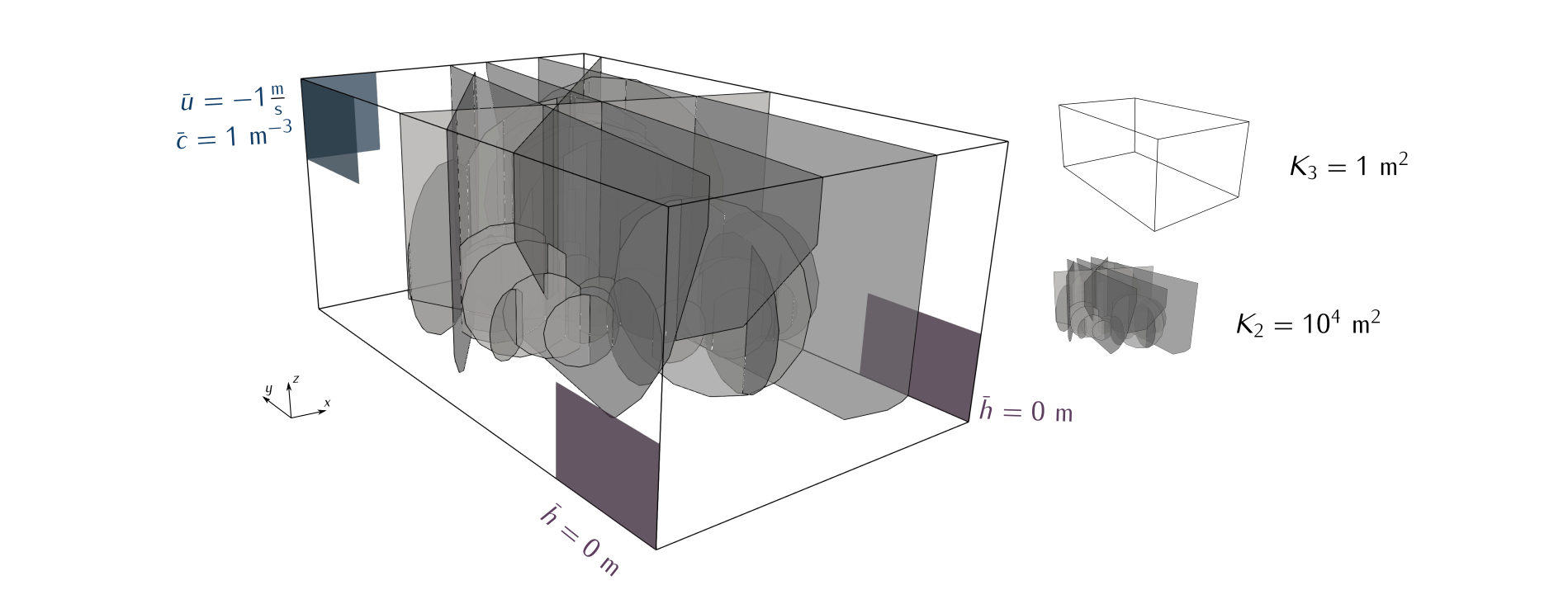}
    \caption{Benchmark 8: Conceptual model and geometrical description
    of the domain.}
    \label{fig:field}
  \end{figure}

We perform the method \eqref{fem} on a tetrahedral mesh with
  $241,338$ matrix elements and $47,154$ fracture elements.
  The number of the globally coupled DOFs is $696,487$.

  The hydraulic head along the two diagonal lines
  $(-500\mathrm{m}, 100\mathrm{m},
-100\mathrm{m})$--$(350\mathrm{m}, 1500\mathrm{m}, 500\mathrm{m})$
  and
  $(350\mathrm{m}, 100\mathrm{m},
-100\mathrm{m})$--$(-500\mathrm{m}, 1500\mathrm{m}, 500\mathrm{m})$
are shown in Figure \ref{fig:fieldC},
along with published results from \cite{Berre_2021}.
Similar to Benchmark 4 in 2D, no reference data on refined meshes was provided
for this problem
due to its complexity.
Comparing with the published results in Figure \ref{fig:fieldC}
we observe that our method still performs quite well.

\begin{figure}[ht]
  \centering
  \begin{tabular}{cc}
    \includegraphics[width=0.48\textwidth]{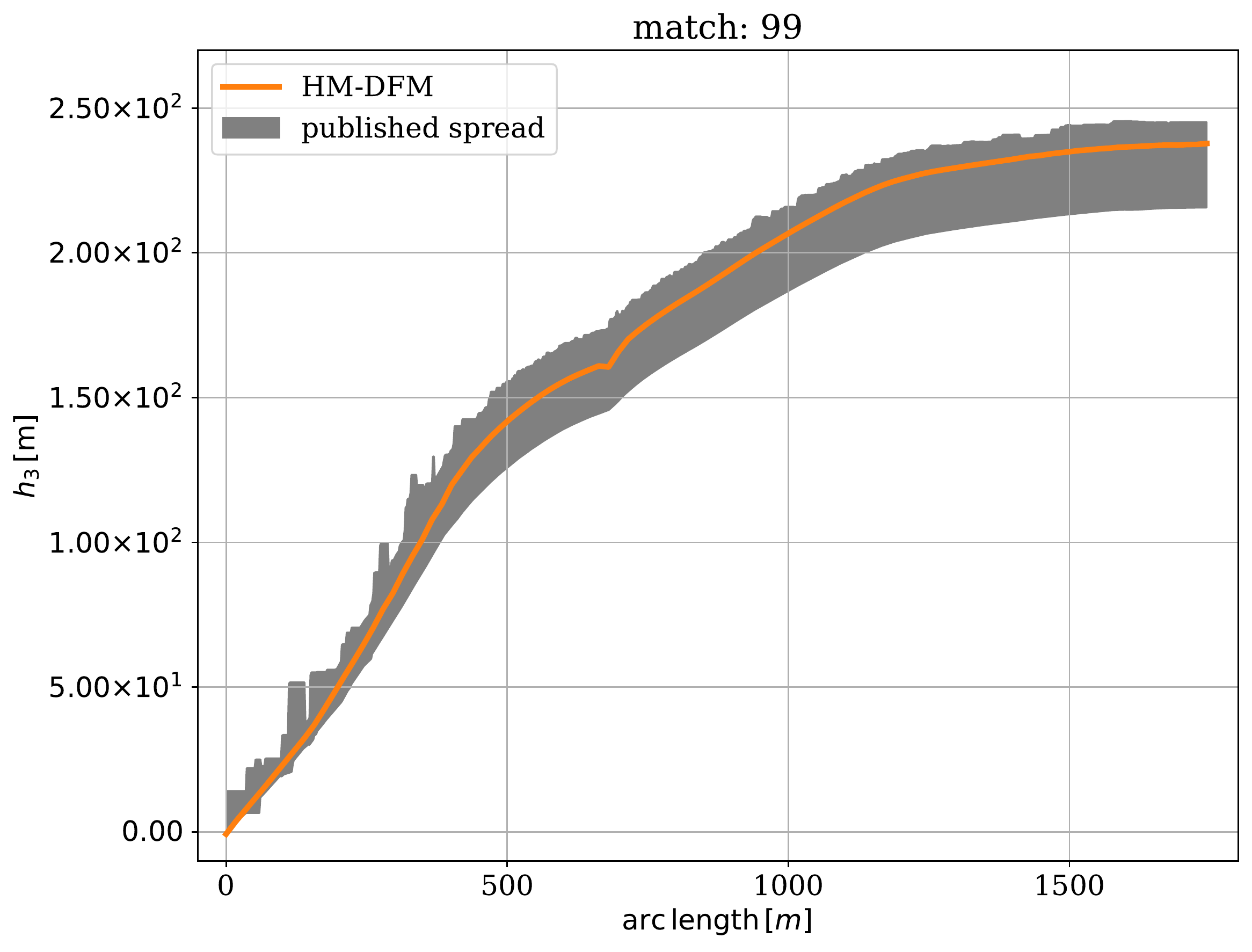}&
    \includegraphics[width=.48\textwidth]{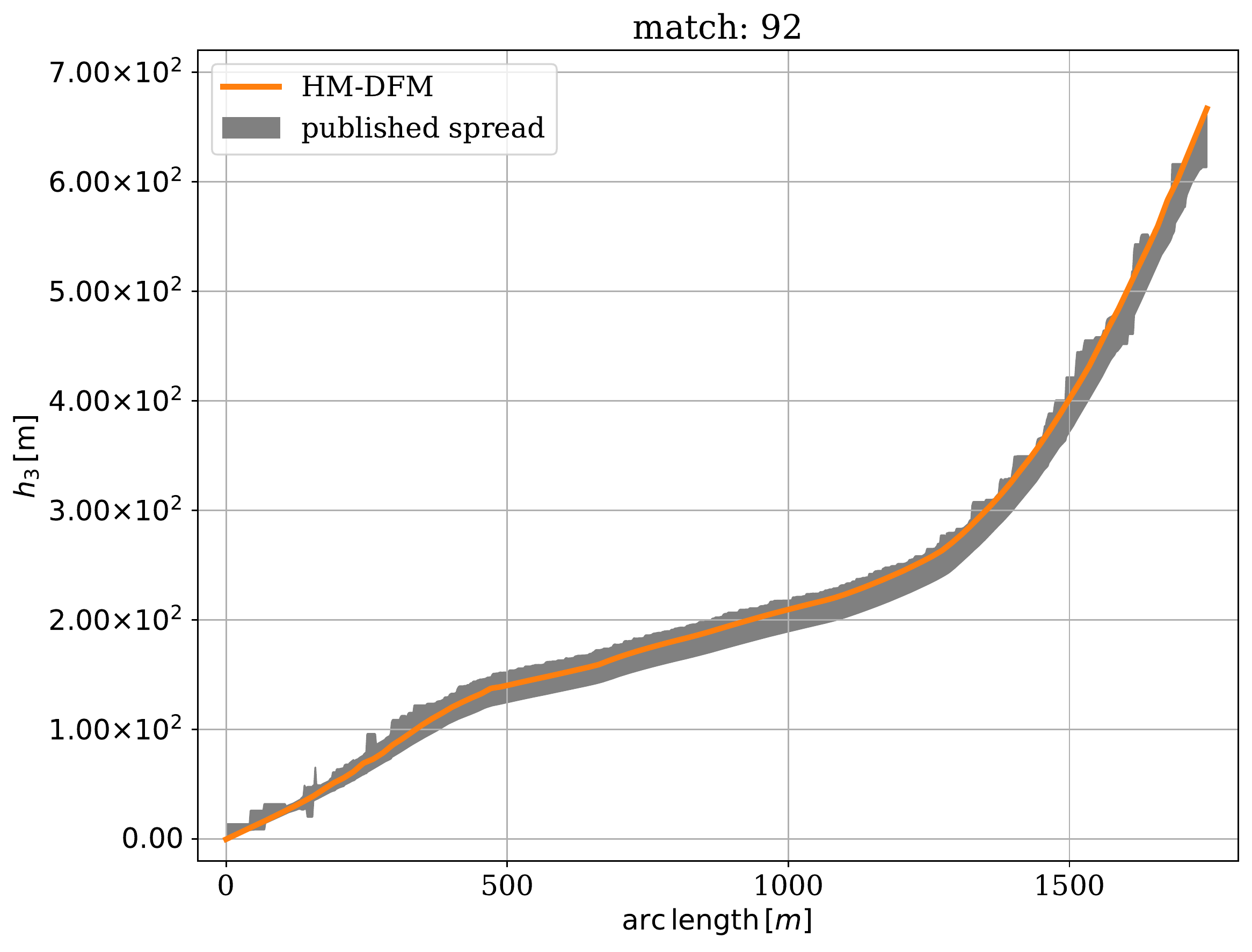}\\
    (a)
    &
    (b)
  \end{tabular}
    \caption{Benchmark 8: Hydraulic head across the domain.
    (a): Profile from outlet $\partial\Omega_{out,0}$
    towards the opposite corner.
    (b): Profile from outlet $\partial\Omega_{out,1}$
    towards the opposite corner $\partial\Omega_{in}$.
  }
    \label{fig:fieldC}
  \end{figure}

\section{Conclusion}
\label{sec:conclude}

  A novel hybrid-mixed method for single-phase flow in fractured porous media has been presented.
  Distinctive features of the scheme includes local mass conservation, symmetric positive definite linear system,
  and allowing the computational mesh to be completely non-conforming to the blocking fractures.
  
  Ample benchmark tests show the excellent performance of the proposed scheme, which is also highly competitive with existing work in the literature.
  Extension to the method to more complex fractured flow models and
  adaptation of the method to more general meshes consists of our on-going work.
 We will also investigate efficient preconditioning procedures for the associated linear system problem in the near future.

\bibliographystyle{ieeetr}

\end{document}